\documentclass[reqno]{amsart}
\usepackage[a4paper]{geometry}
\usepackage[utf8]{inputenc}

\usepackage{amssymb,mathrsfs}
\usepackage[colorlinks=true,linkcolor=blue]{hyperref}
\usepackage[foot]{amsaddr}
\DeclareMathOperator{\End}{End}
\DeclareMathOperator{\Hom}{Hom}
\DeclareMathOperator{\lie}{Lie}
\DeclareMathOperator{\cHom}{\mathcal{H}om}
\DeclareMathOperator{\Spec}{Spec}
\DeclareMathOperator{\Sym}{Sym}
\DeclareMathOperator{\Der}{Der}
\DeclareMathOperator{\tr}{tr}
\DeclareMathOperator{\id}{Id}

\DeclareMathOperator{\SL}{SL}
\DeclareMathOperator{\Aut}{Aut}

\DeclareMathOperator{\gr}{gr}

\def\cA{\mathscr{A}}

\def\cE{\mathscr{E}}
\def\cF{\mathcal{F}}

\def\cI{\mathscr{I}}
\def\cJ{\mathscr{J}}
\def\cK{\mathscr{K}}

\def\cM{\mathscr{M}}
\def\cO{\mathscr{O}}

\def\cT{\mathscr{T}}
\def\fg{\mathfrak{g}}

\def\gl{\mathfrak{g}\mathfrak{l}}

\def\bo{\mathbf{o}}

\swapnumbers
\newtheoremstyle{exps}{\topsep}{\topsep}{}{0pt}{\bfseries}{.}{0pt}{}

\newtheorem*{thm*}{Theorem}
\newtheorem*{prop*}{Proposition}
\newtheorem*{lem*}{Lemma}
\newtheorem*{cor*}{Corollary}
\newtheorem*{rem*}{Remark}
\newtheorem{thm}{Theorem}[section]
\newtheorem{prop}[thm]{Proposition}
\newtheorem{lem}[thm]{Lemma}
\newtheorem{ex}[thm]{Example}

\theoremstyle{definition}
\newtheorem*{defn*}{Definition}
\newtheorem*{exer*}{Exercise}

\newtheorem{defn}[thm]{Definition}
\newtheorem*{problem*}{Problem}
\newtheorem{rem}[thm]{Remark}
\newtheorem{nolabel}[thm]{ }

\theoremstyle{exps}
\newtheorem{examples}[thm]{Examples}

\numberwithin{equation}{section}

\DeclareMathOperator{\Ber}{Ber}
\DeclareMathOperator{\Pic}{Pic}
\DeclareMathOperator{\vir}{Vir}
\DeclareMathOperator{\cur}{Cur}
\DeclareMathOperator{\res}{res}
\DeclareMathOperator{\str}{str}
\DeclareMathOperator{\cExt}{\mathscr{E}xt}
\def\cD{\mathscr{D}}
\def\cV{\mathscr{V}}
\def\cY{\mathscr{Y}}
\def\fm{\mathfrak{m}}
\def\fh{\mathfrak{h}}
\def\bt{\mathbf{t}}
\def\bx{\mathbf{x}}
\def\bq{\mathbf{q}}
\def\bz{\mathbf{z}}
\def\vac{|0\rangle}

\newcommand{\Z}{\mathbb{Z}}

\newcommand{\git}{\mathbin{
  \mathchoice{/\mkern-6mu/}
    {/\mkern-6mu/}
    {/\mkern-5mu/}
    {/\mkern-5mu/}}}
\title[Characters of topological $N=2$ vertex algebras]{Characters of topological $N=2$ vertex algebras are Jacobi forms on the moduli space of elliptic supercurves}
\author[Heluani]{Heluani, R.$^1$}
\address{$^1$ IMPA, Rio de Janeiro.}
\author[Van Ekeren]{Van Ekeren, J.$^2$}
\address{$^2$ Technische Universit\"{a}t Darmstadt, Darmstadt.} 
\date{}
\begin{document}
\maketitle

\section{Introduction}
\begin{nolabel} Let $X$ be a smooth curve of genus $g$ over $\mathbb{C}$ and let $\{x_i\}$, $i=1, \dots, n$, be a collection of closed points in $X$. Let $V$ be a vertex operator algebra and $\{M_i\}$ a collection of modules, one for each point $x_i$. We may associate to this data the space of coinvariants $C(X, \{x_i\}, V, \{M_i\})$ \cite{beilinsondrinfeld,frenkelzvi,zhu,beauville-conformal-blocks}, this is known as the space of conformal blocks associated to the given data. Varying the curve $X$ and the points $\{x_i\}$ in the moduli space $\cM_{g,n}$ gives rise, roughly speaking, to a vector bundle together with a projectively flat connection (in general one obtains, under some finiteness conditions, a twisted $\cD$-module on $\cM_{g,n}$).

The vector space $C(X,x,V,M)$ is constructed roughly as follows (for simplicity we consider $n=1$). The vertex algebra $V$ gives rise to a chiral algebra $\cA$ on $X$. This is in particular a right $\cD_X$-module, and its de Rham cohomology $h(\cA)$ is a sheaf of Lie algebras on $X$. The Lie algebra of sections $h(\cA)(X\setminus x)$ naturally acts on the module $M$, and by definition $C(X,x,V,M) = \Hom_{h(\cA)(X\setminus x)}(M,\mathbb{C})$. The extension to multiple points is straightforward. The spaces of conformal blocks possess a key property regarding insertion of the adjoint $V$-module, namely \cite{frenkelzvi} $C(X,\{x_i\}_{i=1}^n \cup \{x_0\}, \{M_i\}_{i=1}^n \cup \{V\}) \simeq C(X, \{x_i\}_{i=1}^n, \{M_i\}_{i=1}^n)$. We will mostly be interested in the case when each module $M_i$ equals $V$, in this case spaces of conformal blocks are independent of the number of insertion points and we can denote them by $C(X,V)$.

If $X$ has genus $1$, that is, if $X$ is an elliptic curve, then we can write $X = \mathbb{C}^*/\mathbb{Z}$ where the action of $\mathbb{Z}$ on $\mathbb{C}^*$ is given by $n \cdot z = q^n z$ for some $q \in \mathbb{C}^*$. Zhu described explicitly the spaces $C(X,V)$ in \cite{zhu}. For a positive energy $V$-module $M$ he considered the trace function
\begin{equation}
\tr_M Y^M(z^{\Delta_1} a_1, z_1) \dots Y^M(z^{\Delta_n} a_n, z_n) q^{L_0},
\label{eq:traces}
\end{equation}
where $a_1, \ldots, a_n \in V$ with actions $Y^M(a_i,z_i) \in \End(V)[ [z_i, z_i^{-1}] ]$ on $M$, the operator $L_0$ corresponds to a conformal vector in $V$, with respect to which $a_i$ has conformal weight $\Delta_i$. He showed that the trace function converges on the domain $|q| < |z_n| < \dots < 1$, and that this collection of functions describes a vector in $C(X,V)$. In other words, each $V$-module $M$ gives rise to a vector in $C(X,V)$. In particular, characters of $V$-modules give rise to conformal blocks. Zhu also described explicitly the connection on the bundle $C(X,V)$ as $X$ is varied on $\cM_{1,n}$ by writing $q=e^{2 \pi i \tau}$ with $\tau \in \mathbb{H}$ the modular parameter describing $X= X_q \in \cM_{1,n}$ and showing that \eqref{eq:traces} satisfy a specific differential equation (see \eqref{eq:8.6.1} below). He used this to show that the characters of $V$-modules, appropriately normalised (multiplied by $q^{-c/24}$) form an invariant space under the modular group $\SL(2,\mathbb{Z})$.
\end{nolabel}
\begin{nolabel} In this article we extend the results above to the situation when $V$ is an \emph{supersymmetric} vertex algebra \cite{heluani3} and $X$ is an elliptic supercurve (see Section \ref{sec:elliptic-super}). In particular we describe spaces of conformal blocks associated to elliptic supercurves and we show that supertraces of vertex operators acting on $V$-modules produce conformal blocks. The situation in the super case is rather subtle, as there are different flavours of supercuves giving rise to different moduli spaces, and there are distinct classes of supersymmetric vertex algebras. In this article we restrict attention to $N=1$ supercurves (see Section \ref{sec:supermanifolds}), which roughly speaking correspond to pairs consisting of a curve $X$ and a line bundle over it. The vertex algebras that give rise to vector bundles over these supercurves are called $N_W=1$ SUSY vertex algebras in \cite{heluani3} and correspond to \emph{topologically twisted $N=2$ vertex operator algebras} in the physics literature. Alternatively, we could use $N=2$ SUSY curves (see \ref{no:susy}), these are roughly speaking pairs $(X, \cE)$ of a curve $X$ and a rank-$2$ bundle $\cE$ over it such that $\det \cE \simeq \omega_X$. The vertex algebras corresponding to $N=2$ SUSY curves are called $N_K=2$ vertex algebras in \cite{heluani3} and correspond to (untwisted) $N=2$ superconformal vertex operator algebras in the physics literature. (It was observed by Deligne that $N=1$ supercurves and $N=2$ SUSY curves are essentially equivalent, see \cite[pp. 46]{manin3}).

The one point functions that we consider here are closely related to those recently studied by Krauel and Mason in \cite{krauel}, and this article can be viewed as an algebro-geometric counterpart to \cite{krauel}. In the appendix we show that the augmented trace functions
\begin{equation}
\tr_M Y(z^{\Delta_1} a_1, z_1) \dots Y(z^{\Delta_n} a_n, z_n) y^{J_0} q^{L_0}
\label{eq:MasonKrauel-traces}
\end{equation}
converge to a holomorphic function in a domain of $\mathbb{C} \times \mathbb{H}$. Note that here a choice of $U(1)$ current $J_0$ appears in addition to the energy operator $L_0$. Krauel and Mason in \cite{krauel} showed that these characters, after appropriate normalization, transform as weak Jacobi forms with character under the Jacobi group $SL(2, \Z) \ltimes \Z^2$. In this article we show that superfield analogues of these traces give rise to superconformal blocks on a moduli space of elliptic supercurves.

There exists a family $E^0$ of elliptic supercurves obtained from deformations of the reduced underlying elliptic curve and deformations of the line bundle describing the odd coordinate in $\Pic(X_{\mathrm{rd}})$. The base of $E^0$ is a two dimensional purely even supermanifold and the parameters $y$ and $q$ in \eqref{eq:super-traces} can be thought of as coordinates in this space, just as we regard $q = e^{2 \pi i \tau}$ as parametrizing the elliptic curve $X_q$ above. There exists another family (section \ref{no:x.4.1}) of elliptic supercurves with $1|2$ dimensional base, its interpretation is less clear. 

In this paper we construct sections of the bundle of conformal blocks for both families at once by embedding them into a larger family. Then we describe a connection over the supermoduli space, and show that our sections are flat with respect to it (Theorem \ref{thm:connection.1}, \ref{thm:final}). Finally we restrict attention to $E^0$ and we show the connection is equivariant with respect to a natural action of the Jacobi group. Indeed we have the following.
\begin{thm*}
Let $V$ be a (topologically twisted $N=2$) rational super vertex algebra of central charge $C$ satisfying some finiteness conditions and let $\{M_i\}$ be the collection of its simple modules. Then we have
\begin{enumerate}

\item The linear functional $\varphi_M \in V^*$ defined by
\begin{align}\label{eq:super-traces}
\varphi_M(a) = \str_M a_0 q^{L_0} y^{J_0},
\end{align}
is a conformal block for the elliptic supercurve parametrized by $(\tau, \alpha) \in \mathbb{H} \times \mathbb{C}$. 
Here $L_0$ is the energy operator, $J_0$ is the $U(1)$-current, $q = e^{2\pi i \tau}$ and $y = e^{2\pi i \alpha}$.


\item The normalised functionals $\widetilde{\varphi}_M(v) = e^{2 \pi i \alpha C/6} \varphi_M(v)$ are flat with respect to the flat connection $\nabla$ defined by
\[
\nabla = d + \frac{1}{2 \pi i } \int   \widetilde{\zeta}(t)  \Bigl( Y^M[\widetilde{h}, t,\zeta] d \tau  + 2 \pi i \,  \zeta\, Y^M[\widetilde{j}, t,\zeta] d \alpha \Bigr) [dtd\zeta],
\]
where $\widetilde{\zeta}(t)$ is Weierstrass's $\zeta$ function, $h, j$ are the superconformal vectors in $V$ (generating the $N=2$ structure), $Y[u, \mathbf{t}] = \rho^{-1} Y(\rho u, \mathbf{t}) \rho$ is the conjugate of $Y(-, \mathbf{t})$ by the linear automorphism $\rho \in \text{End}\,(V)$ corresponding to the change of coordinates $\rho(t, \zeta) = (e^{2\pi i t}-1, e^{2\pi i t}\zeta)$, and $\widetilde{h}$ and $\widetilde{j}$ are the superconformal vectors of $V, Y[-, \mathbf{t}]$ (see \eqref{eq:huangs}). That is to say $\nabla \widetilde{\varphi}_M(b) = 0$.

\item The linear span of the functionals $\widetilde{\varphi}_M$ is invariant under a natural action of the Jacobi group $SL(2,\mathbb{Z}) \ltimes \mathbb{Z}^2$. For $v \in V$ primary of conformal weight $\Delta$ and charge $0$ this action is described by: 
\[
\begin{aligned}
 {[}\widetilde{\varphi}_M \cdot (m,n)](v; \tau, \alpha) &:= \exp \left( \frac{C}{6} 2 \pi i (m^2 \tau + 2 m \alpha) \right) \times \\ & \qquad  e^{\frac{C}{3} 2 \pi i n} \widetilde{\varphi}_M\left( v; \tau, \alpha + m\tau + n \right) \qquad (m,n) \in \mathbb{Z}^2 \\ 
\left[ \widetilde{\varphi}_M \cdot 
\begin{pmatrix}
a & b \\ c & d
\end{pmatrix}
\right](v; \tau, \alpha) &:= (c \tau + d)^{-\Delta}
  \exp \left( 2 \pi i \frac{C}{6} \left( \frac{- c \alpha^2}{c \tau + d} \right)  \right) \times \\ & \qquad  \widetilde{\varphi}_M\left( v; \frac{a \tau + b}{c \tau + d},\frac{\alpha}{c \tau+ d} \right), \qquad 
\begin{pmatrix}
a & b \\ c & d
\end{pmatrix} \in SL(2,\mathbb{Z}),
\end{aligned} \]
For not-necessarily primary $v$ see Theorem \ref{thm:final}.

\item The connection $\nabla$ is equivariant with respect to the action of the Jacobi  group on $\mathbb{H} \times \mathbb{C}$ and on the space of trace functions as in c). That is under the change of coordinates
\[
(\tau', \alpha') = (\tau, \alpha + m \tau + n)
\quad \text{and} \quad
(t, \zeta) \mapsto (t, e^{2 \pi i m t} \zeta)
\]
we have
\[
\begin{aligned}
& \nabla' = d + \frac{1}{2 \pi i } \int  \widetilde{\zeta}(t')  \Bigl( Y^M[\widetilde{h}, t', \zeta']d \tau' + 2 \pi i \,  \zeta'\, Y^M[\widetilde{j}, t',\zeta'] d \alpha' \Bigr)[dt'd\zeta'], \\
& \nabla' \left[ \widetilde{\varphi}_M \cdot 
(m,n) 
 \right] = 0, \qquad (m,n) \in \mathbb{Z}^2,
\end{aligned}
\]
and under the change of coordinates
\[
(\tau', \alpha') = \left( \tfrac{a \tau + b}{c \tau + d}, \tfrac{\alpha}{c\tau + d} \right)
\quad \text{and} \quad 
(t', \zeta') = \left( \tfrac{t}{c\tau + d}, e^{- 2 \pi i t \tfrac{c \alpha}{c \tau + d}} \zeta \right)
\]
we have
\[ 
\begin{aligned}
& \nabla' = d + \frac{1}{2 \pi i } \int  \widetilde{\zeta}(t')  \Bigl( Y^M[\widetilde{h}, t', \zeta']d \tau' + 2 \pi i \,  \zeta'\, Y^M[\widetilde{j}, t',\zeta'] d \alpha' \Bigr)[dt'd\zeta'], \\
& \nabla' \left[ \widetilde{\varphi}_M \cdot 
\begin{pmatrix}
a & b \\ 
c & d
\end{pmatrix}
 \right] = 0, \qquad 
\begin{pmatrix}
a & b \\ 
c & d
\end{pmatrix} \in SL(2, \mathbb{Z}).
\end{aligned}
\]

\end{enumerate}
\end{thm*}
\end{nolabel}
\begin{nolabel}\label{no:1.3as}
This theorem can be interpreted as follows: the spaces of conformal blocks form an $SL(2, \mathbb{Z}) \ltimes \mathbb{Z}^2$-equivariant vector bundle with flat connection on $\mathbb{H} \times \mathbb{C}$ (with coordinates $\tau, \alpha$). This can be identified with a left $\cD$-module on the orbifold quotient $\mathbb{H} \times \mathbb{C} \git SL(2, \mathbb{Z}) \ltimes \mathbb{Z}^2$. The conformal blocks $\widetilde{\varphi}_M$ associated to each irreducible $V$-module $M$ are flat sections.  
\end{nolabel}
\begin{nolabel}\label{no:1-warning1}
In the setup of \cite{krauel} modular invariance is proved by a careful computation and is based upon the work of Miyamoto \cite{miyamoto} and \cite{Li-twisted}. Our present result is based on the transformation properties of super-fields under coordinate changes.
\end{nolabel}
\begin{nolabel}\label{no:added-appendix}
This article is divided in two parts. The main body of the article is algebro-geometrically oriented and we do not deal with convergence issues. Trace functions are assumed to be convergent therefore giving rise to holomorphic sections of certain bundles over supercurves. In the appendix we treat these trace functions as formal power series and prove their convergence properties under certain finiteness conditions on the vertex algebra. Readers more acquainted with the theory of vertex operator algebras and their relation to modular forms might want to read the appendix first referring to the main text solely for the definitions. The appendix is mostly independent of the main part of the article. Readers more interested in the algebro-geometric counterparts of vertex algebras, namely chiral algebras, and their relations to moduli spaces of supercurves can assume that characters are convergent series and read the main body of the article first. 
\end{nolabel}
\begin{nolabel}\label{no:appendix-result}
The appendix consists mainly of an adaptation of Zhu's proof of theorem \cite[Thm 4.4.1]{zhu} to the supersymmetric case. Under a finiteness condition that we call \emph{charge cofiniteness} (see \ref{defn:c2-cofinite-super})  that turns out to be equivalent  to the $C_2$-cofiniteness condition in the usual case, we have the following 
\begin{thm*}
Let $V$ be a $C_2$-cofinite conformal topologically twisted $N=2$ super vertex algebra and let $M$ be a positive energy module. Denote by $L_0$ (resp. $J_0$) the conformal weight operator (resp. the charge operator). For a vector $a \in V$ homogeneous with respect to both conformal weight and charge decompositions,  let $Y^M(a,z)$ be the corresponding field of the action of $V$ on $M$ and $o(a) \in \End(M)$ be the Fourier mode preserving conformal weight and charge. Then the formal power series
\[ \varphi_M(a) = \str_M o(a) q^{L_0} y^{J_0}, \qquad q = e^{2 \pi i \tau}, \: y = e^{2 \pi i \alpha}, \quad (\tau, \alpha) \in \mathbb{H} \times \mathbb{C}\]
converges absolutely and uniformly in a domain of $\mathbb{H} \times \mathbb{C}$ to a holomorphic function. 
\end{thm*}
The proof of this theorem follows Zhu's approach and borrows from the theory of
Jacobi modular forms as developed by Eichler and Zagier in \cite{zagier-jacobi} and the
quasi-Jacobi forms of Libgober \cite{libgober2009elliptic} as well as general results from the theory of differential equations. Although the geometric constructions developed in the main body of the article are not strictly needed in the appendix, their importance is evident in understanding the kernel of the trace functions as the quotients of Jacobi modular forms of non-vanishing index that are used (see Prop. \ref{prop:a.13})  turn out to be meromorphic sections of the corresponding chiral algebras.   
\end{nolabel}
\begin{nolabel}
In section \ref{sec:supermanifolds} we collect some remarks on the theory of supermanifolds and schemes with particular attention to supercurves. This section is for the sake of quick reference and self-containedness. It can  be skipped on a first reading. In section \ref{sec:vertex-algebras} we collect some results on vertex algebras. In section \ref{sec:susy-vertex-algebras} we collect the analogous results on SUSY vertex algebras. In section \ref{sec:chiral-algebras} we briefly mention the algebro--geometric counterparts of chiral algebras and SUSY chiral algebras attached to (SUSY) vertex algebras on algebraic (super) curves. The reader more familiar with the theory of vertex operator algebras as opposed to their algebro-geometric counterparts can safely skip this section on a first reading. In section \ref{sec:elliptic-super} we construct the versal families of elliptic super-curves over which our chiral algebras will be defined. We describe the action of the modular and Jacobi groups on these families and explain the coordinate description of sections of chiral algebras on these elliptic super-curves. In section \ref{sec:conformal-blocks} we recall the definition of conformal blocks and show that indeed the supertrace functions give rise to conformal blocks. In section \ref{sec:connection} we show that the trace functions satisfy a first order system of partial differential equations with respect to the modular parameters $\tau$ and $\alpha$. In section \ref{sec:modularity} we show that these differential equations are equivariant with respect to the action of the Jacobi group.  In section \ref{sec:examples} we give an example on the SUSY vertex algebra associated to a  unimodular lattice and briefly mention the example of elliptic genera of Calabi-Yau manifolds and the chiral de Rham complex. 
\label{no:plan-of-article}
\end{nolabel}

\emph{Acknowledgements:} JVE would like to thank IMPA, where he was a postdoc while the bulk of this work was carried out, as well as the IH\'{E}S and the Alexander von Humboldt Foundation for financial support. RH would like to thank Nathan Berkovits for illuminating discussions and the explanation  in  \ref{no:x.4.1}. He would also like to thank Yongchang Zhu for providing a printed copy of his Yale Ph.D. thesis. It goes without saying that this works borrows heavily from Zhu's seminal article. Both authors would like to thank Matt Krauel for kindly pointing out about some upcoming changes in the preprint \cite{krauel}, this prompted the addition of the appendix in this manuscript. 

\section{Elements of supergeometry} \label{sec:supermanifolds}
\begin{nolabel} For the theory of supermanifolds we refer the reader to \cite{manin3,vaintrob1,deligne2} and references therein. A \emph{superspace} is a locally ringed space $(X, \cO_X)$ where $\cO_X= \cO_{X,\bar0} \oplus \cO_{X, \bar1}$ is a sheaf of supercommutative rings on a topological space $X$ with local stalks. A \emph{morphism} of superspaces is a pair $(f, f^\sharp)$ where $f: X \rightarrow X'$ is a continuous map and $f^\sharp: \cO_{X'} \rightarrow f_* \cO_{X}$ is a graded morphism of sheaves of rings. 
We often abuse notation and denote the pair simply by $X$. 
We denote by $\cJ$ the sheaf of ideals $\cO_{X,\bar1} \oplus \cO_{X,\bar1}^2 \subset \cO_X$. We can always define the super-space $X_{\mathrm{rd}} = (X,\cO_X/\cJ)$, it is a \emph{purely even} ($\cO_{X_{\mathrm{rd}}} = \cO_{X_{\mathrm{rd},\bar0}}$) subsuperspace of $X$. The category of superspaces has fibered products defined as in the non-super situation. In particular, given two morphisms $f: X \rightarrow Y$ and $f': X' \rightarrow Y$, we have a superspace $X \times_Y X'$ defined as a locally closed subsuperspace of $X\times X'$. 
\end{nolabel}
\begin{nolabel} Let $(X,\cO_X)$ be a superspace. The sheaf $\cO_X$ admits a filtration by powers of the ideal $\cJ$. Let $\gr X = (X,\gr\cO_X := \oplus_{i \geq 0} \cJ^i/\cJ^{i+1})$ be the associated superspace. The map $\gr_0 \cO_X \hookrightarrow \gr \cO_X$ induces a morphism $\gr X \rightarrow X_{\mathrm{rd}}$, and $X$ is called \emph{split} if it is isomorphic to $\gr X$. 
\end{nolabel}
\begin{nolabel} A \emph{superscheme} is a superspace $X$ such that the locally ringed space $(X,\cO_{X, \bar0})$ is a scheme and $\cO_{X, \bar1}$ is a coherent sheaf of $\cO_{X, \bar0}$-modules. Morphisms of superschemes are pairs $(f,f^\sharp)$ as above such that $(f, f^\sharp_{\bar 0})$ is a morphism of schemes (here $f^\sharp_{\bar 0}$ is the even part of $f^\sharp$).  

A module $M$ over a supercommutative ring $R$ is \emph{flat} if the functor $\otimes M$ is exact in the category $R$-mod. Let $f:X \rightarrow Y$ be a morphism of superschemes and $\cF$ a $\cO_X$-module. Let $y \in Y$ be a point in the underlying topological space $Y$, so that the stalk $\cO_{Y,y}$ of $\cO_Y$ at $y$ is a local supercommutative ring. Then $\cF$ is \emph{flat} over $y$ if, for every $x \in f^{-1}(y) \subset X$, the fiber $\cF_x = \varinjlim_{x \in U} \cF(U)$ is a flat $\cO_{Y,y}$-module. $\cF$ is flat over $Y$ if it is flat over every point $y \in Y$. The morphism $f:X \rightarrow Y$ is flat if $\cO_X$ is flat over $Y$.
\end{nolabel}
\begin{nolabel} Let $f:X \rightarrow Y$ be a morphism of superschemes. We have the corresponding adjoint functors $f^*$,  $f_*$ between the categories of $\cO_X$-modules and $\cO_Y$-modules defined as in the non-super case. 
Let $\Delta:X \rightarrow X \times_Y X$ be the relative diagonal and let $\cI$ be the sheaf of ideals corresponding to $\Delta(X)$. Define the sheaf $\Omega_{X/Y} = \Delta^*(\cI/\cI^2)$. A superscheme $X$ is called \emph{irreducible} (resp. \emph{finite type, separated})  if $X_{rd}$ is irreducible (resp. finite type, separated). Let $X$ be an irreducible superscheme of finite type over an algebraically closed field $k$. We say that $X$ is \emph{smooth} of superdimension $n|m$ if the sheaf $\Omega_X := \Omega_{X/k}$ is locally free of rank $n|m$. In the split situation this is equivalent to $X_\mathrm{rd}$ being a smooth scheme over $k$ of dimension $n$ and $\cJ/\cJ^2$ being a locally free $\cO_X/\cJ$-module of rank $m$.

A morphism $f: X \rightarrow Y$ between irreducible superschemes is called \emph{smooth} of relative dimension $n|m$ if it is flat and if for each point $x \in X$ the module $\Omega_{X/Y} \otimes k(x)$ is a free $k(x)$-module of rank $n|m$, where $k(x)$ is the residue field of the local ring $\cO_{X,x}$. The morphism is \emph{\'etale} if it is smooth of relative dimension $0|0$. 
 
\end{nolabel}
\begin{nolabel}\label{sec:berezinian} Let $X$ be a smooth superscheme of dimension $n|m$ and $\Omega_{X}$ be its sheaf of differentials. The dual $\cT_X$ of $\Omega_X$ is called the \emph{tangent sheaf}. It is a locally free sheaf of rank $n|m$. 
We let\cite{deligne2}
\[
\omega_X = \Ber \Omega_X := \cExt^n_{\Sym^* \cT_X} (\cO_X, \Sym^* \cT_X),
\]
be the \emph{Berezinian} bundle of $X$. It is a locally free $\cO_X$-module of rank $1|0$ if $m$ is even and of rank $0|1$ if $m$ is odd. 
Given a morphism $f:X \rightarrow Y$ we similarly define $\omega_{X/Y} := \Ber \Omega_{X/Y}$. 
\end{nolabel}
\begin{nolabel} We will be interested in smooth superschemes over $k=\mathbb{C}$.  In this situation it is useful to consider the analytic topology on $X$. Let $X$ be such a scheme of dimension $n|m$. For each point $x \in X$ in such a superscheme, there exists an open neighborhood $x \in U$ in the analytic topology such that the superspace $(U, \cO_{X}|_U)$ is isomorphic to the superspace $\mathbb{C}^{n|m}$ constructed as follows. The underlying topological space is $\mathbb{C}^n$ with its usual topology and for each $U \subset \mathbb{C}^n$ open, sections of the structure sheaf are defined to be the free supercommutative ring over the algebra $\mathbb{C}(U)$ of holomorphic functions on $U$, generated by an odd vector space of dimension $m$. Choosing holomorphic coordinates $\{t_i\}$, $i = 1, \dots, n$ in $U$ and a basis $\{\zeta_j\}$, $j=1,\dots, m$ for this vector space, we write this algebra as $\mathbb{C}(t_i)[\zeta_j]$ where the $t_i$ are commutative variables and $\zeta_j$ are anticommutative variables. We call the set $\{t_i, \zeta_j\}$ \emph{coordinates} near $x$.

In the algebraic setting, we consider $\mathbb{A}_\mathbb{C}^{n|m} := \Spec \mathbb{C}[t_i, \dots, z_n, \zeta_1, \dots, \zeta_m]$ and for a smooth scheme of dimension $n|m$ over $\mathbb{C}$,  an \'etale map $x \in U \rightarrow \mathbb{A}^{n|m}$ is called a set of coordinates near $x$. 
\label{no:super-affine}

\end{nolabel}
\begin{nolabel} Let $X$ be a smooth superscheme over $k$ and $\cT_X$ be its tangent sheaf. We define the sheaf of differential operators $\cD_X$ on $\cT_X$ as the filtered sheaf of $Z/2\mathbb{Z}$ graded rings as follows. $\cF^0 \cD_X = \cO_X\subset \End_k(\cO_X)$ (the inclusion is by left multiplication) and inductively define 
\[ \cF^{i+1}\cD_X = \left\{ S \in \End_k(\cO_X) | [S, \cO_X] \subset \cF^i \cD_X\right\}. \] The $\mathbb{Z}/2\mathbb{Z}$ grading is inherited from $\End_k(\cO_X)$. By $\cD_X$-mod we mean the category of right $\cD_X$-modules quasi-coherent as $\cO_X$-modules. We have $\omega_X \in \cD_X$-mod naturally via the Lie derivative. Similarly, by definition $\cO_X$ is a left $\cD_X$-module. The functor $\otimes \omega_X$ induces an equivalence of categories between the category of left $\cD_X$-modules and $\cD_X$-mod. If $\dim X = n|2m+1$ then this functor is \emph{odd}, namely it exchanges the $\mathbb{Z}/2\mathbb{Z}$-gradings in both categories (see \ref{sec:duality}).

For a morphism of smooth superschemes $f: X \rightarrow Y$ we have the natural functors $f_*, f_*, f_!, f^!$ defined as in the non-super situation. In particular, for the locally closed embedding $\iota: X_\mathrm{rd} \hookrightarrow X$, we have an equivalence of categories: $\iota^!: \cD_{X}$-mod $\rightarrow \cD_{X_\mathrm{rd}}$-mod \cite{penkov}. 

Given a smooth morphism $\pi:X \rightarrow S$ between smooth superschemes we have the sheaf of relative differential operators $\cD_{X/S}$ defined as follows. We first define vertical vector fields as usual by
\[ 0 \rightarrow \cT_{X/S} \rightarrow \cT_X \rightarrow \pi^* \cT_S \rightarrow 0. \]  We define the sheaf of relative differential operators as the filtered sheaf of rings $\cD_{X/S}$ on $X$ defined by $\cF^0 \cD_X = \cO_X \subset \End_{\pi^{-1} \cO_S}(\cO_X)$,  and inductively by 
\[ \cF^{i+1} \cD_{X/S} = \left\{ S \in \End_{\pi^{-1} \cO_S}(\cO_X) \: | \: [S, \cO_X] \subset \cF^i \cD_{X/S} \right\}, \]
\label{no:1.diff}
\end{nolabel}
\begin{nolabel}\label{sec:duality} Let $X$ be a superscheme of finite type over a field $k$, then there exists a unique (up to quasi-isomorphism) \emph{dualizing} complex $K_X \in D^b(X)$ \cite[\S 1.5.1]{vaintrob1}, that is the functor \[ M \rightarrow \mathrm{Hom} (M, K) \] is an exact anti-equivalence inducing natural isomorphisms $H^i(M) \simeq \mathrm{Ext}^{-i} (M, K_X)^*$. If $X$ is \emph{smooth} then $K_X$ is concentrated in one degree so it is an $\cO_X$-module up to a shift, and is just the Berezinian bundle $\omega_X := \mathrm{Ber }\, \Omega^1_X$ defined in \ref{sec:berezinian}. If $\dim X = n | m$, the \emph{natural} dualizing complex is given by $\omega_X[n]$ if $m$ is even and by $\Pi \omega_X[n]$ if $m$ is odd, where $\Pi$ is the change of parity functor.
\label{no:1.1}
\end{nolabel}
\begin{nolabel}
Let $f:X \rightarrow Y$ be a smooth proper morphism of relative dimension $n|m$ between smooth superschemes. We let $\omega_f = \Pi^m \omega_X \otimes f^* \omega_Y^* [n-m]$ be the relative dualizing sheaf. For a coherent $\cO_X$-module $\cF$ and $\cO_Y$-module $\cE$ (more generally for objects of the appropriate derived categories) we have a functorial isomorphism (Grothendieck-Verdier duality)
\[ Rf_* R\cHom (\cF, Lf^* (\cE) \otimes \omega_f) \simeq R \cHom(Rf_* \cF, \cE).\] 
\label{no:trace}
\end{nolabel}
\begin{nolabel}  Let $i:Y \hookrightarrow X$ be a closed sub-superscheme defined by a sheaf of ideals $\cI$. We obtain a closed sub-scheme $Y_\mathrm{rd} \subset X_\mathrm{rd}$ and an open sub-superscheme $j:X\setminus Y \hookrightarrow X$ obtained as the restriction of $\cO_X$ to $X_\mathrm{rd} \setminus Y_\mathrm{rd}$. For an $\cO_X$-module $M$ we have the module $i^! M$ consisting of sections of $M$ supported in $Y$, that is, sections killed by powers of the ideal $\cI$ (notice that this definition only uses $Y_\mathrm{rd}$). The functor $i^!$ is the adjoint functor to $i_!$ defined in the usual way (for closed $Y$ it coincides with $i_*$). We have an exact triangle $i_! i^! M \rightarrow M \rightarrow j_* j^* M$ inducing a long exact sequence in (local) cohomology \[\dots \rightarrow H^k(X,M) \rightarrow H^k(X\setminus Y, j^*M) \xrightarrow{\partial} H^{k+1}_Y (X,M) \rightarrow \dots , \] where as usual $H^k_Y$ denotes the k-th derived functor of the functor of sections supported in $Y$ \cite[Expos\'e I]{sga2}.
\label{no:1.2}
\end{nolabel}
\begin{nolabel}
We now apply the local duality theorem together with the previous long exact sequence. That is, we consider the case $X = \Spec A$ where $A$ is a local regular supercommutative ring of dimension $1|q$ with maximal ideal $\fm$ (we will be interested in the local ring at a closed point of a supercurve) and $Y = \Spec A \slash \fm$ its closed point (notice that $Y$ is of dimension $0|0$ hence it is not a divisor unless $q=0$). We obtain a map 
\[ \res_Y: H^0(X \setminus Y, \omega_X) \xrightarrow{\partial}H^1_Y(X,\omega_X) \rightarrow k, \] where the last map is the usual residue map \cite{hartshornelocal}. 

Choosing local parameters $\bz=(z;\theta^i)$ $i=1,\dots,q$ for $A$ so that $\fm$ is generated by $(z;\theta^i)$, we obtain a trivialization $[d\bz]:=[dzd\theta^1 \dots d\theta^q]$ of $\omega_X$. We use this trivialization to identify $H^0(X \setminus Y, \omega_X) \simeq H^0(X\setminus Y, \cO_X) \simeq A_\fm \simeq  k( (\bz))$ and the residue map takes the form  (see for example \cite[II.6.4]{manin3}) \[ \res_\bz f(\bz) =\Bigl\{ \text{coef. of } z^{-1}\theta^1 \dots \theta^q \Bigr\} f(\bz). \]
\label{no:residue}
\end{nolabel}
\begin{nolabel}
In the relative situation, letting $\pi: X \rightarrow S$ be a family of supercurves, so that $\pi$ is a proper map of smooth superschemes of relative dimension $1|n$ with connected fibers. The duality theorem of \ref{no:trace} gives rise to the trace isomorphism
\[ R^1 \pi_* \omega_{X/S} \simeq \cO_S, \]
\label{no:trace2}
where $\omega_{X/S} = \Pi^n \Ber \Omega^1_{X/S} [1]$ is the Berezinian of the sheaf of relative differentials, isomorphic to the relative dualizing sheaf $\omega_\pi$ of \ref{no:trace}. 
\end{nolabel}
\begin{nolabel} 
A family of $N=n$ supercurves is a smooth  morphism $f: X \rightarrow S$ of relative dimension $1|n$ with connected fibers. Its reduced space $X_{rd}$ is a family of usual curves over $S_{rd}$. The \emph{genus} of $X/S$ is the genus of $X_{rd}/S_{rd}$. In this article we will be concerned with \emph{elliptic} supercurves, by which we mean those with genus $g = 1$. 

There are other flavours of supercurves. We now briefly recall for example the $\text{SUSY}_N$ curves defined by Manin in \cite{manin2} (also called \emph{superconformal supercurves} by Rabin et. al in \cite{rabin1} and references therein). A (family of) $\text{SUSY}_N$ supercurve is a curve $X/S$ of relative dimension $1|N$ together with a locally free locally direct subsheaf $\cT' \subset \cT_{X/S} = \cT$ of rank $0|N$ for which the Frobenius form
\[ \varphi: \wedge^2 \cT' \rightarrow \cT / \cT', \qquad  \tau \wedge \tau' \mapsto [\tau,\tau' ] + \cT', \]
is non-degenerate and split, that is, locally has an isotropic direct subsheaf of maximal rank. In this case the sheaf of relative differential operators $\cO_X \rightarrow \cO_X$, relative to $\cO_S$ as in \ref{no:1.diff}, is generated by multiplication by $\cO_X$ and sections of $\cT'$ since the Frobenius form is non-degenerate,  

In the case $N=1$ we can choose local coordinates $\bz=(z,\theta)$ on $X/S$ so that a local trivialization of $\cT'$ is given by $D = D_\bz = \partial_\theta + \theta \partial_z$, In fact $D^2 = \partial_z$ and $D - \theta \cdot D^2 = \partial_\theta$ obtaining thus all generators. In the $N=2$ case we can choose local coordinates $(z,\theta^\pm)$ so that $\cT'$ is trivialized by $D_\pm = \partial_{\theta^\pm} + \theta^{\mp} \partial_z$. We have $[D^+, D^-] = 2 \partial_z$.  
\label{no:susy}
\end{nolabel}
\begin{nolabel}
Below we will need the following analogues of the algebraic groups $\mathbb{G}_a$ and $\mathbb{G}_m$ \cite[2.7.3]{manin3}. Consider $\mathbb{G}_a^{1|1} = \mathbb{A}^{1|1}$ with the multiplication law given by 
\[ (t,\zeta) \cdot (t', \zeta') = (t + t' + \zeta \zeta', \zeta + \zeta'). \]
It is a non-commutative smooth supergroup. It has a SUSY structure as in \ref{no:susy} given by $D = \partial_\zeta + \zeta \partial_t$. We have the supergroup $\mathbb{G}_m^{1|1}$ which is the open sub-superspace of $\mathbb{A}^{1|1}$ obtained by restriction to $\mathbb{G}_m \subset \mathbb{A}^1$ (it represents the functor $R \rightarrow R^*$ of invertible elements in a super commutative ring \cite[2.8.1]{manin3}). Upon choosing coordinates $(x,\theta)$ the multiplication law of $\mathbb{G}_m$ is given by
\[ (x,\theta) \cdot (x', \theta') = (x x' + \theta \theta', x \theta' + x' \theta). \]
The exponential map 
\begin{equation}\label{eq:exponential}
(t, \zeta) \mapsto \left( x = e^{2 \pi i t}, \theta = e^{2 \pi i t} \zeta \right)
\end{equation}
is a surjective group homomorphism sending the SUSY structure of $\mathbb{G}_a^{1|1}$ to $D_m = x \partial_\theta + \theta \partial_x$. 

We stress however that we will not use the group structure nor the SUSY structure on these super-schemes $\mathbb{G}_a^{1|1}$ and $\mathbb{G}_m^{1|1}$. We will only need their super-scheme structure and as a short notation for $\mathbb{C}^{1|1}$ and $\mathbb{C}^* \times \mathbb{C}^{0|1}$.
\label{no:supergroups}
\end{nolabel}

\section{Preliminaries on vertex algebras} \label{sec:vertex-algebras}
\begin{nolabel} A {\em vertex algebra} consists of a $k$-vector space $V$, a vector (the vacuum vector) $\vac \in V$, an operator $\partial = \End_k(V)$ and an operation
\[
\mu:V \otimes V \rightarrow V \otimes k((z)), \quad a \otimes b \mapsto \mu(a\otimes b) \equiv Y(a, z)b
\]
called the \emph{state-field correspondence} satisfying the following axioms
\begin{enumerate}
\item $\mu (V \otimes \vac) \subset V\otimes k[ [z]]$,  $\mu (a \otimes \vac)|_{z=0} = a $ and $\mu \left( \vac \otimes \cdot  \right) = \id_V \otimes 1$.
\item $[\partial, Y(a,z)] = \partial_z Y(a,z)$. 
\item $(z-w)^n[Y(a,z), Y(b,w)] = 0$ for $n \gg 0$. 
\end{enumerate}
Morphisms of vertex algebras are linear maps $\varphi \in \Hom_k (V, V')$ such that $\varphi \vac = \vac'$ and $\mu' \circ \varphi \otimes \varphi = \varphi \circ \mu$. Given a vertex algebra $V$, a $V$-module $M$ is the data of a $k$-vector space $M$ together with an operation
\[
\mu^M: V \otimes M \rightarrow M \otimes k( (z)), \quad \mu^M(a \otimes m)= Y^M(a,z)m
\]
satisfying $Y^M(\vac, z) = \id_M$ and for all $a,b \in V, \: m \in M$, we have
\[
Y^M(a,z) Y^M(b,w) m = Y^M(Y(a,z-w)b,w)m
\]
as elements of $M \otimes k( (z))( (w))$, where in the RHS we use the canonical map $k( (z-w))( (w)) \hookrightarrow k( (z))( (w))$ of expanding $(z-w)^n$ in positive powers of $w/z$. The fact that $V$ is naturally a $V$-module with $Y=Y^V$ is known as associativity. A submodule $N \subset M$ is a vector subspace stable under $Y^M(\cdot,z)$. The quotient space $M/N$ is naturally a $V$-module. A submodule $I \subset V$ is called an \emph{ideal}, the quotient $V/I$ is naturally a vertex algebra. 

Given a vertex algebra $V$ and its module $M$, we write the {\em fields} as $Y(a,z) = \sum_{n \in \mathbb{Z}} a_{(n)} z^{-1-n}$ and $Y^M(a,z)= \sum_{n \in \mathbb{Z}} a^M_{(n)} z^{-1-n}$ where, $a_{(n)} \in \End_k (V)$ and $a_{(n)}^M \in \End_k(M)$. 
The $\mathbb{Z}/2\mathbb{Z}$-versions of these definitions are straightforward generalizations.   
\label{no:vertex-algebra-def}
\end{nolabel}
\begin{nolabel} Consider the algebraic group $\mathbb{G}_a$ and its corresponding Hopf algebra: the polynomial ring $k[\partial]$. The comultiplication makes $k[\partial, \partial']$ into a $k[\partial]$-module, in particular it is a $k[\partial, \partial']-k[\partial]$ bimodule. Let $V \in k[\partial]$-mod, we obtain a $k[\partial, \partial']$-module $\Delta_* V :=k[\partial, \partial'] \otimes_{k[\partial]} V$. We also have the $k[\partial, \partial']$-module $V \otimes_k V$ by the isomorphism $k[\partial]\otimes_k k[\partial] \simeq k[\partial, \partial']$. A \emph{Lie conformal algebra} is an operation 
\[ \mu \in \Hom_{k[\partial, \partial']-mod} \left( V \otimes V,  \Delta_* V \right), \]
that satisfies skew-symmetry and the Jacobi axiom. The skew-symmetry condition is straightforward to write as $\mu \circ \sigma = - \sigma \circ \mu$ where $\sigma$ is the obvious isomorphism on $k[\partial, \partial']$-mod given by exchanging the two factors. The Jacobi condition is an identity in $k[\partial, \partial', \partial'']$-mod. To write it down it is customary in the literature to break the symmetry by using the isomorphism of vector spaces $\Delta_* V \simeq k[\partial]\otimes_k V \simeq k[-\lambda]\otimes V$ and write the operation $\mu$ as $[_\lambda ]: V \otimes V \rightarrow k[\lambda]\otimes V$. The Jacobi condition reads\[ [ a_\gamma [b_\lambda c]] = [ [a_\gamma b]_{\gamma+\lambda} c] + [ b_\lambda [a_\gamma c]] \in V \otimes k[\gamma,\lambda]. \]

A Lie conformal algebra with $\mu =0$ is called Abelian. A $k[\partial]$-submodule $W$ stable under $\mu$ (resp. $\mu W \otimes V \subset \Delta_* W$)  is called a Lie conformal subalgebra (resp. an ideal). A submodule $W$ with $\mu W \otimes V = 0$ is called an Abelian ideal. 

\label{no:conformaldefinition}
\end{nolabel}
\begin{examples}\hfill
\begin{enumerate}
\item  Most examples we will work will be of Lie conformal algebras that are of finite type over $k[\partial]$. Consider first the $k[\partial]$-module \[ \vir := k[\partial]/\partial k[\partial] \oplus  k[\partial] \overset{\text{Vect}}{\simeq} k \oplus k[\partial]. \] We will denote by $L$ the generator of the free module and $C$ the generator of the first summand (so that $\partial C = 0$). The only non-trivial bracket amongst generators is
\begin{equation}
[L_\lambda L] = (\partial + 2\lambda) L + \frac{\lambda^3}{12} C. 
\label{eq:virasoro1}
\end{equation}
 Notice that $C$ is central in this algebra.  This algebra is called the \emph{Virasoro} Lie conformal algebra
\item Let $\fg$ be a finite dimensional Lie algebra and let $(,) \in \Sym^2 \fg^*$ be an invariant bilinear form. We will consider the $k[\partial]$-module
\[ \cur(\fg) := k[\partial]/\partial k[\partial] \oplus k[\partial]\otimes_k \fg
\overset{\text{Vect}}{\simeq} k \oplus k[\partial]\otimes_k \fg, \] with non-trivial brackets amongst generators
\[ [a_\lambda b] = [a,b] + \lambda (a,b) K, \qquad a,b \in \fg,  \]where $K$ is the central element that generates $k[\partial]/\partial k[\partial]$. This algebra is called the \emph{current or affine Lie conformal algebra}. 
\item A particular example of b) is when $\fg$ is Abelian. In this case $\fg$ is just a finite dimensional vector space $V$ with a symmetric bilinear form $(,)$. The non-trivial brackets amongst generators now read:
\[ [v_\lambda w] = \lambda (v,w) K, \qquad v,w \in V. \]
This algebra is called the \emph{Heisenberg} or \emph{free Bosons} Lie conformal algebra.
\item Now let $V$ be a finite dimensional vector space and let $\langle,\rangle$ be a \textbf{skew-symmetric} bilinear form. We consider the $k[\partial]$-module 
\[ F(V) := k[\partial]/\partial k[\partial] \oplus k[\partial]\otimes_k V 
\overset{\text{Vect}}{\simeq} k \oplus k[\partial]\otimes_k V, \] with non-trivial brackets amongst generators
\[ [v_\lambda w] = \langle v, w\rangle K.\]
This algebra is known as the free (Bosonic) ghost system. In addition to replacement of a symmetric form by a skew-symmetric one, notice that this example is unlike the previous one in that the factor of $\lambda$ does not appear. For this Lie conformal algebra to be nontrivial the dimension of $V$ needs to be at least $2$. 
\item From the supersymmetric point of view, the previous example is a special case of the \emph{free Fermions} construction. to explain this we need a $\mathbb{Z}/2\mathbb{Z}$-graded version of the definition in \ref{no:conformaldefinition}.

Let $V$ be a finite dimensional super vector space and let $\langle, \rangle$ be a \emph{super-skew-symmetric} bilinear form, i.e., a bilinear form $\langle, \rangle$ such that $\langle V_{\bar{0}} , V_{\bar{1}} \rangle = \langle V_{\bar{1}}, V_{\bar{0}} \rangle = 0$ and the restrictions of $\langle, \rangle$ to $V_{\bar{0}}$ and to $V_{\bar{1}}$ are respectively skew-symmetric and symmetric. Then we use $\langle, \rangle$ to form $F(V)$ as above to obtain a super Lie conformal algebra. The particular case when $V$ is a purely odd vector space (therefore $\langle, \rangle$ is symmetric) is commonly known as the Lie conformal algebra of \emph{free Fermions}. Notice that in this case, the dimension of $V$ may be $1$ and still the algebra may not be Abelian. For this reason the example d) above is referred to as \emph{even Fermions} in some literature. In fact that algebra goes under several other names too: \emph{symplectic Bosons}, \emph{Weyl Lie conformal algebra}, and the $\beta \gamma$-system.  

Remaining for now in the case when $V$ is purely odd and can furthermore be written as $V_- \oplus V_+$ where $V_+$ and $V_-$ are each isotropic for $\langle,\rangle$, and are non-degenerately paired $V_- \otimes V_+ \rightarrow k$, the corresponding Lie conformal algebra $F(V)$ is commonly known as a \emph{Fermionic ghost system}, or a $bc$-system. We will stick to the supersymmetric point of view and refer to all these examples as \emph{free Fermions}. 
\item Consider the super $k[\partial]$-module $NS$ whose even part is $NS_{\bar{0}} := \vir$ and whose odd part is a free module $NS_{\bar{1}}:=k[\partial]$ with generator $G$. The only non-trivial brackets are \eqref{eq:virasoro1} and 
\begin{equation} [L_\lambda G] = \left( \partial + \frac{3}{2} \lambda \right) G, \qquad [G_\lambda G] = L + \frac{\lambda^2}{6} C, 
\label{eq:neveu}
\end{equation}
\item Lastly we define the \emph{topological} Lie conformal algebra. It is a sum of the free $k[\partial]$-supermodule of rank $2|2$ with even generators $L,J$, odd generators $Q, H$, and with a one dimensional centre generated by $C$. The brackets are
\begin{equation}
\begin{aligned}
		{[}{L}_\lambda {L}] &= (\partial + 2\lambda)
			{L}, & 
			{[}{L}_\lambda J] &= (\partial + \lambda)J +
			\frac{\lambda^2}{6}C, & 
			{[}{L}_\lambda Q] &= (\partial + 2\lambda) Q, \\
			{[}{L}_\lambda H] &= (\partial + \lambda) H, & 
			[H_\lambda Q] &= {L} - \lambda J +
			\frac{\lambda^2}{6} C. & [J_\lambda J] &= \frac{C}{3} \lambda,\\
[J_\lambda Q] &= Q, & [J_\lambda H] &= -H.
\end{aligned}
\label{eq:topological}
\end{equation}
\end{enumerate}
\label{ex:examples-conformal-algebras}
\end{examples}
\begin{nolabel}
A vertex algebra $V$ is a Lie conformal algebra with 
\[ [a_\lambda b] := \sum_{j \geq 0} \frac{\lambda^k}{k!} a_{(j)}b. \]
Conversely, given any Lie conformal algebra $V$ there exists a universal enveloping vertex algebra $\iota: V \hookrightarrow U(V)$ such that any other morphism of Lie conformal algebras $V \rightarrow W$ to a vertex algebra $W$ factors to a unique vertex algebra morphism from $U(V)$. A proper ideal of $V$ generates a proper ideal of $U(V)$. For $c \in k$ we will denote by $U^c(V)$ the quotient by the ideal generated by $(C-c)$ for $V$ in all the examples in \ref{ex:examples-conformal-algebras}. In particular, the vertex algebra $U^c(\vir)$ of Example \ref{ex:examples-conformal-algebras} a) will be called the \emph{Virasoro vertex algebra} of central charge $c$. In Example \ref{ex:examples-conformal-algebras} g) we will call $c$ also \emph{central charge} even though the corresponding Virasoro field has central charge $0$.  A vertex algebra $V$ is called \emph{conformal} if there exists a morphism $U^c(\vir) \rightarrow V$ -- therefore by the universal property of $U(\vir)$ a vector $L \in V$ satisfying \eqref{eq:virasoro1};  such that for  any $a \in V$ we have $[L_\lambda a] = \partial a + O(\lambda)$. This vector $L$ is usually denoted by $\omega \in V$. A vector $a$ in a conformal vertex algebra $V$ satisfying $[L_\lambda a ] = \partial a + \Delta \lambda a + \lambda^2 p$ for some $\Delta \in k$ and $p \in V \otimes k[\lambda]$ is said to be of \emph{conformal weight $\Delta$}. It is \emph{primary} if $p=0$. 

Expanding the field $L(z) := Y(L, z) = \sum_{n \in \mathbb{Z}} z^{-2-n} L_n$ (note the shifting in the power of $z$ and the difference between $L_n$ and $L_{(n)}$ as in \ref{no:vertex-algebra-def}) we obtain the commutation relations of the \emph{Virasoro Lie algebra}:
\begin{equation}
[L_m, L_n] = (m-n) L_{m+n} + \delta_{m,-n} \frac{m^3-m}{12}c, \qquad m,n \in \mathbb{Z},
\label{eq:virasoro-commutators}
\end{equation}
hence a conformal vertex algebra is naturally a representation of this Lie algebra with central charge $c$. It follows from the fact that $V$ is a $V$-module (the associativity in \ref{no:vertex-algebra-def}) that for any $a \in V$ we have
\begin{equation}
\left[ L_n, Y(a,z) \right] = \sum_{m \geq -1} \binom{n+1}{m+1} Y(L_m a, z)z^{n-m}, \qquad n \in \mathbb{Z}. 
\label{eq:quasi-conf-def}
\end{equation}
The Virasoro Lie algebra is a central extension of the algebra of polynomial vector fields on $S^1$, or the algebra of derivations of $k( (z))$\footnote{There is a difference in considering derivation of $k[z,z^{-1}]$ or the completed $k( (z))$ as we are considering here, but this difference is irrelevant for the purpose of this article.}. This algebra has an annihilation subalgebra consisting of those vector fields that can be extended to the whole disk, that is the algebra of derivations of $k[ [z]]$. In our basis above this corresponds to those $L_m$ with $m \geq {-1}$.  A vertex algebra is called \emph{quasi-conformal} \cite{frenkelzvi} if there exists operators $L_m \in \End_k (V)$, $m \geq -1$  satisfying \eqref{eq:virasoro-commutators} and \eqref{eq:quasi-conf-def} for any $a \in V$ and such that $L_{-1}= T$ and $L_0$ is diagonal with integer eigenvalues bounded below by $0$ and finite dimensional eigenspaces. We will require that our conformal algebras be quasi-conformal, namely, for a conformal vector $\omega$ such that $Y(\omega,z) = L(z)$ we will require that $L_0$ has integer eigenvalues bounded below by $0$ and with finite dimensional eigenspaces. 
\label{no:universal enveloping}
\end{nolabel}
\begin{nolabel}\label{no:c2-cofiniteness}
A conformal vertex algebra is called \emph{rational} if each module is completely reducible. It is called $C_2$ \emph{cofinite} if the subspace \[ C_2(V) = \left\{ a_{(-2)}b \: | \: a, b \in V \right\} \subset V\] is of finite codimention in $V$. A simple rational and $C_2$ cofinite vertex algebra has finitely many irreducible modules and its representation category is semisimple.
\end{nolabel}
\begin{nolabel}\label{no:lie-algebra}
Let $V$ be a Lie conformal algebra, then its quotient $V \otimes_{k[\partial]} k$ is a Lie algebra with bracket $[a,b] := [a_\lambda b]_{\lambda = 0}$. Letting $\cO = k[ [t]]$ and $\cK = k( (t))$ we have a conformal algebra $\hat{V} := V \otimes \cK$ with bracket given by 
\[ [a \otimes f(t)_\lambda b \otimes g(t)] = [a_{\lambda + \partial_t} b]\otimes f(t) g(t')|_{t = t'}. \] 
Now we combine these two constructions: for any Lie conformal algebra $V$ we have the Lie algebra $\mathrm{Lie}(V) := \hat{V} \otimes_{k[\partial]} k$. In particular, every vertex algebra gives rise to a Lie algebra in this manner. For $a \in V$ and $n \in \mathbb{Z}$, denote by $a_{\{n\}}$ the image of $a \otimes t^n$ in $\mathrm{Lie} (V)$. 
\end{nolabel}
\begin{nolabel}\label{no:fourier2}
The \emph{Fourier modes} $a_{(n)} \in \End(V)$ defined for each $a \in V$ and $n \in \mathbb{Z}$ for a vertex algebra $V$ form a Lie subalgebra of $\End(V)$. Their commutator is given by 
\[ [a_{(m)}, b_{(n)}] = \sum_{j \geq 0} \binom{m}{j} \bigl( a_{(j)}b \bigr)_{(m+n-j)}. \]
The map $a_{\{m\}} \mapsto a_{(m)}$ defines a Lie algebra morphism $\mathrm{Lie}(V) \rightarrow \End(V)$. This map however might have a nontrivial kernel. 
\end{nolabel}
\begin{nolabel}
Let us expand the fields of Example \ref{ex:examples-conformal-algebras}(g) as 
\[J(z) = \sum_{n \in \mathbb{Z}} J_n z^{-1-n}, \qquad Q(z) = \sum_{n \in \mathbb{Z}} Q_n z^{-2-n}, \qquad H(z) = \sum_{n \in \mathbb{Z}} H_n z^{-1-n},\]
then we obtain the commutation relations
\begin{equation}
\begin{gathered}
{[}L_m, L_n] = (m-n) L_{m+n}, \qquad
[L_m,J_n] = -n J_{m+n} + \frac{(m^2 + m)c}{6} \delta_{m,-n}, \qquad
[L_m,H_n] = -n H_{m+n} \\ 
[L_m,Q_n] = (m-n) Q_{m+n},  \qquad  [J_m, J_n] = \frac{c}{3} m \delta_{m, -n}, \qquad  [J_m, Q_n] = Q_{m+n} \\  
[J_m, H_n] = -H_{m+n}, \qquad  [H_m, Q_n] = L_{m+n} - m J_{m+n} + \frac{(m^2 - m)c}{6} \delta_{m,-n}. 
\end{gathered}
\label{eq:topological-commutators}
\end{equation}
\label{no:topological-commutations}
This Lie superalgebra is a central extension of the Lie superalgebra of polynomial vector fields on a super-circle $S^{1|1}$, or equivalently of derivations of $k( (z))[\theta]$. It will play the role of the Virasoro algebra in the supersymmetric setting.
\end{nolabel}
\section{SUSY vertex algebras} \label{sec:susy-vertex-algebras}
In this section we recall some definitions from \cite{heluani3} and define the notion of \emph{quasi-conformal} $N_W=1$ SUSY vertex algebra, which is a straightforward generalization of the non-super situation.
\begin{nolabel}
To fix notation let $k$ be a field (we will only need $k = \mathbb{C}$) and let $\cO = k[ [z]][\theta] =: k[ [\bz]] $ where $z$ is even and $\theta$ is odd, so that $\cO$ is the local ring at a point in $\mathbb{A}^{1|1}$ as in \ref{no:super-affine}. Let $\cK = k ( (z))[\theta] =: k( (\bz))$ be its ring of fractions. We will often write $k[ \partial, D]$ instead of $k[z, \theta]$ for the same algebra.
\label{no:notation-super}
\end{nolabel}
\begin{nolabel}
An $N_W=1$ SUSY vertex algebra consists of a $k[\partial, D]$-module $V$, a vector $\vac \in V_{\bar 0}$ and an even operation
\[
\mu: V \otimes V \rightarrow V \otimes \cK, \quad \mu(a \otimes b)=: Y(a,z,\theta)b
\]
satisfying 
\begin{enumerate}
\item $\mu(a \otimes \vac) \subset V \otimes \cO$, $\mu (a \otimes \vac)|_{z=\theta=0} = a$ and $\mu (\vac \otimes \cdot)  = \id_V \otimes 1$. 
\item $[D, Y(a,z,\theta)] = \partial_\theta Y(a,z,\theta)$ and $[\partial, Y(a,z,\theta)] = \partial_z Y(a,z,\theta)$. 
\item $(z-w)^n[Y(a,z,\theta), Y(b,w,\xi)] = 0$ for $n \gg 0$ and all $a,b \in V$.
\end{enumerate}
Morphisms of SUSY vertex algebras are defined as in the usual case as even linear maps preserving vacuum vectors and intertwining the products $V \otimes V \rightarrow V \otimes \cK$. Similarly, a module over a SUSY vertex algebra $V$ consists of a vector space $M$ together with a map
\[
\mu^M : V \otimes M \rightarrow M \otimes \cK, \quad \mu^M(a \otimes m) = Y^M(a,z,\theta)m
\]
satisfying 
\[
Y^M(a,z,\theta) Y^M(b,w,\xi)m = Y^M(Y(a,z-w,\theta-\xi)b,w,\xi)m,
\]
as elements of $M \otimes k( (z,\theta))( (w,\xi))$, where (as before) we expand the RHS using the natural map $k( (z-w))( (w)) \hookrightarrow k( (z))( (w))$ (note that the odd coordinates do not introduce any new complications here). 
\label{no:susy-def}
\end{nolabel}
\begin{nolabel}
Let $V$ be a SUSY vertex algebra and $M$ a $V$-module, then $V$ is in particular an ordinary vertex algebra (which is $\mathbb{Z}/2\mathbb{Z}$-graded) and $M$ is its module. The state field correspondence for $V$ as a usual vertex algebra is given by $a \otimes b \mapsto Y(a,z,\theta)b|_{\theta =0}$. For each $a \in V$ we let $a(z) = Y(a,z,\theta)|_{\theta=0}$. Notice that the \emph{super-fields} $Y(a,z,\theta)$ are naturally expanded as $Y(a,z,\theta) = a(z) + \theta Da(z)$. Similarly we let $a^M(z) = Y^M(a,z,\theta)|_{\theta=0}$ and the map $a \otimes m \mapsto a^M(z)m$ expresses $M$ as a usual module over the vertex algebra $V$. The super-field $Y^M(a,z,\theta)$ naturally expands as $Y^M(a,z,\theta) = a^M(z) + \theta (Da)^M(z)$. 
\label{no:relation-susy-usual}
\end{nolabel}
\begin{defn}
An $N_W=1$ SUSY vertex algebra is called \emph{quasi-conformal} if it carries a representation of the annihilation subalgebra of the topological algebra \eqref{eq:topological-commutators}, i.e., there are operators $L_n,  Q_n$, $n \geq -1$ and $J_n, H_n$, $n \geq 0$ satisfying the commutation relations \eqref{eq:topological-commutators} and such that for every $a \in V$ we have
\begin{equation} \label{eq:quasi-conformal-susy}
\begin{aligned}
{[J}_n, Y(a,z,\theta)] &= \sum_{j \geq 0} \binom{n}{j} z^{n-j} \Bigl( Y(J_j a, z,\theta) + \theta Y(Q_{j-1} a,z,\theta)\Bigr) \\
{[H}_n, Y(a,z,\theta)] &= \sum_{j \geq 0} \binom{n}{j} z^{n-j} \Bigl( Y(H_j a, z,\theta) - \theta Y\left( (L_{j-1} - j J_{j-1}) a,z,\theta \right)\Bigr) \\
[Q_n, Y(a,z,\theta)] &= \sum_{j \geq -1} \binom{n+1}{j+1} z^{n-j} Y(Q_{j} a,z,\theta)\\
[L_n, Y(a,z,\theta)] &= \sum_{j \geq -1} \binom{n+1}{j+1} z^{n-j} \Bigl( Y(L_j a,z,\theta) + (j+1) \theta Y(Q_{j-1} a, z, \theta) \Bigr)
\end{aligned}
\end{equation}
In addition, we require that $L_-1 = \partial$, $Q_{-1} = D$ and $L_0, J_0$ are diagonal operators with integer eigenvalues, finite dimensional joint eigenspaces and the $L_0$ eigenvalues are bounded below by $0$. A vector $v \in V$  killed by $L_n, J_n, Q_n, H_n, \, n > 0$ will be called \emph{primary}. Its $L_0$ eigenvalue is called its \emph{conformal weight} and its $J_0$ eigenvalue is called its \emph{charge}. These conditions on $V$ guarantee that the representation of the subalgebra $\Der_0 \cO^{1|1} = \text{span} \left\{ L_n, J_n, Q_n, H_n \right\}_{n \geq 0}$ can be exponentiated to an action of the group $\Aut \cO^{1|1}$ of continuous automorphisms of $\cO^{1|1} = k[ [z,\theta]]$. We will impose the additional condition that $V$ is generated, as an $\Aut \cO^{1|1}$-module, by primary vectors. 

The $N_W=1$ SUSY vertex algebra $V$ is called \emph{strongly-conformal} (or simply \emph{conformal} if no confusion can arise) if it is quasi-conformal and furthermore there exists an even vector $j$ and an odd vector $h$ such that their corresponding fields are given by
\begin{equation} \label{eq:expansion-usual} Y(j,z,\theta) = J(z) - \theta Q(z), \qquad Y(h, z, \theta) = H(z) + \theta \left( L(z) + \partial_z J(z) \right), \end{equation}
satisfying the commutation relations \eqref{eq:topological-commutators} and the above-mentioned integrability conditions. 
\label{defn:topological-quasi-conformal}
\end{defn}
\begin{defn}
Let $V$ be a quasi-conformal $N_W=1$ SUSY vertex algebra and $M$ its module. We say that $M$ is \emph{positive energy} if it carries a representation of the annihilation subalgebra of the topological algebra \eqref{eq:topological-commutators} satisfying the conditions of Definition \ref{defn:topological-quasi-conformal} with $Y$ replaced by $Y^M$. It follows that $V$ is a positive energy $V$-module. 
\label{defn:positive-energy-susy}
\end{defn}
\begin{nolabel}
We will need the following
\begin{lem*}
For any $\bx = (x, \theta)$, $\bz = (z,\chi)$  we have 
\[ \res_{\bx - \bz} i_{|z|>|x-z|} = \res_{\bx} \left( i_{|x|>|z|} - i_{|z|>|x|} \right), \]
as maps $V[ [\bx,\bz]][x^{-1}, z^{-1}, (x-z)^{-1}] \rightarrow  V( (\bz)) ( (\bx))$. Here $i_{|x|>|z|}$ denotes the map \[ V[ [\bx,\bz]][x^{-1}, z^{-1},(x-z)^{-1}] \rightarrow V( (\bz))( (\bx))\] of \emph{expansion in the domain} $|x|>|z|$. (Equivalently, since $V[ [\bx, \bz]] \subset V( (\bz)) ( (\bx))$ and $x, z$ and $(x-z)$ are not $0$ divisors, we obtain $i_{|x|>|z|}$ by the universal property of the ring of fractions.)
\end{lem*} 
\begin{cor*}
For any $f \in k[ [\bx,\bz]][x^{-1},z^{-1},(x-z)^{-1}]$ one has
\begin{multline*}
\res_{\bx-\bz} i_{|z|>|x-z|} f Y \left( Y(a,\bx-\bz)b, \bz \right)c \\
= \res_{\bx} i_{|x|>|z|} f Y(a,\bx)Y(b,\bz) c - (-1)^{ab} \res_{\bz} i_{|z|>|x|} f Y(b,\bz) Y(a,\bx) c,
\end{multline*}
where $(-1)^{ab}$ denotes $-1$ if both $a, b \in V_{\bar 1}$ and $1$ otherwise. 
\end{cor*}
\label{lem:associativity-in-coords}
\end{nolabel}

\section{(SUSY) Chiral algebras} \label{sec:chiral-algebras}
	We collect some results and definitions from \cite{beilinsondrinfeld} and \cite{heluani4}. Most definitions are mutatis mutandis the same ones as in \cite{beilinsondrinfeld}. The main difference in the supercurve case is that the diagonals are not divisors, however they still define a stratification of the power of a curve which one can use in the Cousin resolution. Another subtle difference discused in Remark \ref{rem:parity} is the fact that the \emph{unit} chiral algebra, the Berezinian bundle of a curve $\omega_X$ is \emph{odd} when the dimension of $X$ is $1|2k+1$. We also sketch the connection between chiral algebras and vertex operator algebras following \cite{frenkelzvi} and \cite{heluani4}.  The reader more familiar with the theory of vertex operator algebras as opposed to their algebro-geometric counterparts can safely skip this section on a first reading. The essential ingredient that will be needed below is Huang's change of coordinate formula for the (super) fields \eqref{eq:huangs}.  We deduce in \ref{ex:5.8} the extension of $\cD$-modules corresponding to the chiral algebra associated to the topological algebra of Example \ref{ex:examples-conformal-algebras} g).  The factorization algebra counterpart to these definitions in the supercurve case was given by Kapranov and Vasserot in \cite{kapranov-super}. Presumably one could obtain a global version of the results in this article if one would define following \cite{beilinsondrinfeld} the machinery of chiral homology on supercurves. The work of Kapranov and Vasserot \cite{kapranov-super} would be most relevant in the $N=1$ supersymmetric case.  
\begin{nolabel}
Let $X$ be a smooth superscheme over $k=\mathbb{C}$ of dimension $1|n$ for some $n \geq 0$. For a finite non-empty set $I$, let $X^I = \prod_{i \in I} X$ and denote $p_i : X^I \rightarrow X$ the $i$-th projection. For each surjection $\pi: I \twoheadrightarrow J$ we have a correponding diagonal map $\Delta_\pi : X^J \hookrightarrow X^I$ defined such that $p_i \circ \Delta_\pi = p_{\pi(i)}$. Consider $Z = \coprod \Delta_\pi X^{J} \subset X^I $ where the union is taken over the set of all projections $\pi: I \twoheadrightarrow J$ with $|J| = |I| - 1$. We let $Z = \{\}$ in the case $I=\left\{ * \right\}$. Let $U^{(I)}$ be the open sub-superscheme $X^I \setminus Z$. It is completely determined by the reduced structure of $X^I_\mathrm{rd}$ minus the diagonal divisor. In fact, we have the closed embedding $Z_\mathrm{rd} \subset X^I_\mathrm{rd}$ (the diagonal divisor) and letting $U^{(I)}_\mathrm{rd} = X^I_\mathrm{rd} \setminus Z_\mathrm{rd}$ we have that $U^{(I)}$ consists of $U^{(I)}_\mathrm{rd}$ as a topological space together with the restriction of $\cO_{X^I}$ as its sheaf of superalgebras. Let $j=j^{(I)} : U^{(I)} \hookrightarrow X^I$ be the inclusion and finally let $\Delta : X \rightarrow X^I$ be the total diagonal corresponding to the surjection $I \twoheadrightarrow \left\{ * \right\}$. 

For a finite set $\left\{ L_i \right\}_{i \in I}$ of $\cD_X$-modules and a $\cD_X$-module $M$ We define the set of chiral operations as \cite[3.1.1]{beilinsondrinfeld}
\begin{equation} \label{eq:chiral-operations} P^{ch}_I (\left\{ L_i \right\}, M) := \Hom_{\cD_{X^I}}\Bigl( j_* j^* \left( \boxtimes_{i \in I} L_i \right) , \Delta_* M \Bigr) \end{equation}
\label{no:5.1}
\end{nolabel}
\begin{nolabel}
For a surjective map of finite non-empty sets $\pi : J \twoheadrightarrow I$ we have the corresponding composition of chiral operations
\[\begin{aligned} P^{ch}_I \bigl( \left\{ L_i \right\}, M \bigr) &\otimes& \Bigl( \bigotimes_{i \in I} P^{ch}_{\pi^{-1}(i)} \bigl( \left\{ K_j \right\}_{j \in \pi^{-1}(i)}, L_i \bigr) \Bigr) &\rightarrow& P^{ch}_J \bigl(\left\{ K_j \right\}_{j \in J}, M \bigr) \\
\varphi \qquad &\otimes&   (\psi_i)\qquad \qquad \qquad  &\mapsto& \qquad  \varphi(\psi_i) \qquad \end{aligned}\]
defined as follows. First we let $j^{(\pi)} : U^{(\pi)} \hookrightarrow X^J$ be the open subsuperscheme defined by restriction of $\cO_{X^J}$ to the open set $U^{(\pi)}_\mathrm{rd} := \left\{ (x_j) \in X^J_{\mathrm{rd}} \: : \: \pi(j_1) \neq \pi(j_2) \Rightarrow x_{j_1} \neq x_{j_2}  \right\}$ -- i.e., $U^{(\pi)}$ is the open sub-superscheme obtained from $X^J$ minus the diagonal corresponding to $\pi$. Then we compose
\begin{multline*}
j^{(J)}_* j^{(J)*} \left( \boxtimes_J K_J \right) = j_*^{(\pi)} j^{(\pi)*}  \left( \boxtimes_{i \in I} \left( j_*^{(\pi^{-1}(i))} j^{(\pi^{-1}(i))*} \boxtimes_{j \in \pi^{-1}(i)} K_j\right) \right) \xrightarrow{\boxtimes \psi_i} \\  j_*^{(\pi)} j^{(\pi)*}\left( \boxtimes_I \Delta_*^{(\pi^{-1}(i))} L_i\right) = \Delta_*^{(\pi)} j_*^{(I)} j^{(I)*} \left( \boxtimes_I L_i \right) \xrightarrow{\Delta_*^{(\pi)} (\phi)} \Delta_*^{(\pi)} \Delta_*^{(I)} M = \Delta^{(J)}_* M.
\end{multline*}
Here the first equality is due to the exactness of the restriction to the open sub-schemes $U^{(J)} \subset U^{(\pi)} \subset X^J$. The second equality follows from the compatibility of the push-forward with the $\boxtimes$ products: if $f_i :X_i \rightarrow Y_i$ are morphisms between smooth superschemes then $(\prod f_i)_* \boxtimes L_i = \boxtimes f_{i*} L_i$. Since $\prod \Delta^{(\pi^{-1}(i))} = \Delta^{(\pi)}$ it follows that $\Delta^{(\pi)}_* \boxtimes L_i = \boxtimes \Delta_*^{(\pi^{-1}(i))} L_i$. This composition is associative in the obvious sense and defines a colored operad or pseudo-tensor category structure in $\cD_X$-mod in the sense of \cite[1.1.1]{beilinsondrinfeld}. 
\label{no:5.3}
\end{nolabel}
\begin{nolabel}\label{no:5.4b}
If one is given extra structure on the supercurve $X$ one could define other compound tensor categories associated to it. For example to a SUSY structure, one may consider the ring of differential operators preserving this structure (that is generated by sections of the distribution $\cT' \subset \cT$ as in \ref{no:susy}) and chiral operations as above. The corresponding chiral algebras would correspond to the $N_K$ family of SUSY vertex algebras of \cite{heluani3}.
\end{nolabel}
\begin{nolabel}Let $i_0 \in I$ and consider the projection $\pi_{I,i_0}: X^I \rightarrow X^{I\setminus i_0}$. The diagonal $\Delta^{(I)}$ factors through $\Delta^{(I \setminus i_0)}$ and for each $M \in \cD_X-$mod we have the short exact sequence:
\begin{equation} \label{eq:short-exact} 0 \rightarrow \Delta^{(I \setminus i_0)}_* M \boxtimes \omega_X \rightarrow j_* j^* \bigl( \Delta^{(I \setminus i_0)} M \boxtimes \omega_X \bigr) \rightarrow \Delta^{(I)}_* M \rightarrow 0,\end{equation}
where $j$ is the open inclusion of $\{ (x_i) \in X^I \: | \: x_{i_0} \neq x_i\, \forall i \neq i_0 \}$. 
Projecting to $X^{I \setminus i_0}$ using $\pi_{I,i_0}$ and using 
the trace morphism of \ref{no:trace2} we obtain morphisms: inducing for each $M \in \cD_X$-mod a morphism  $\tr_{i_0}: {\pi_{I,i_0}}_* \Delta^{(I)}_* M \rightarrow \Delta^{(I\setminus i_0)}_* M$. 

For $M \in \cD_X$-mod we set $h(M) := M \otimes_{\cD_X} \cO_X$ and call this the \emph{de Rham} functor. It is an augmentation functor in the sense of \cite[1.2.5]{beilinsondrinfeld} for the pseudotensor structure described in \ref{no:5.3}, that is for each element $i_0$ of a set $I$ of order $|I| \geq  2$ we have natural maps
\[ h_{I, i_0}: P_I^{ch} (\left\{ L_i \right\}, M) \otimes h(L_{i_0}) \rightarrow P_{I \setminus i_0}^{ch}(\left\{ L_i \right\}, M), \]
compatible with compositions (cf. \cite[1.2.5]{beilinsondrinfeld}). The maps $h_{I,i_0}$ are defined as follows. For $\phi \in P^{ch}_I (\left\{ L_i \right\},M)$ and a representative $l \in L_{i_0}$ of $\bar{l} \in h(L_{i_0})$, we define $h_{I,i_0} (\phi \otimes \bar{l})$ as the composition $a \rightarrow \phi( a\boxtimes l)  \in \Delta^{(I)}_* M \rightarrow \tr_{i_0} \phi( a \boxtimes l) \in \Delta^{(I \setminus i_0)}_* M$. Clearly this composition is independent of the representative $l$ chosen since the trace morphism kills total derivatives.  
\label{no:5.3b}
\end{nolabel}
\begin{defn}
Let $X$ be a smooth superscheme over $k$ of dimension $1 | n$. A \emph{chiral algebra} on $X$ is a (right) $\cD_X$-module $\cA$ together with an operation $\mu \in P^{ch}_{ \left\{ 1,2 \right\}} (\left\{ \cA,\cA \right\},\cA)$ satisfying skew-symmetry and the Jacobi-condition. The skew-symmetry condition is described as follows: if $\sigma$ is the automorphism of $X\times X$ exchanging the $2$-factors we have $\mu \circ \sigma = - \mu$. 

To describe the Jacobi condition we consider the projection  $\pi: \left\{ 1,2,3 \right\} \twoheadrightarrow \left\{ 1,2 \right\}$ given by $\pi^{-1}(2) = \left\{ 2,3 \right\}$. According to \ref{no:5.3} we have the associated composition $\mu_{ 1\left\{ 23 \right\}} := \mu(1 \otimes \mu )$. Permuting cyclically we obtain the compositions $\mu_{2 \left\{ 3 1 \right\}}$ and $\mu_{ 3\left\{12  \right\}}$. The Jacobi condition is simply $\mu_{1 \left\{23  \right\}} + \mu_{ 2\left\{ 31 \right\}} + \mu_{3\left\{ 12 \right\}} = 0$. 

The dualizing residue map $\res: j_* j^* \omega_{X^2} \rightarrow \Delta_* \omega_X$ (where $j: U^{(\left\{ 1,2 \right\})} \subset X^2$ is the open embedding of \ref{no:5.1}) gives a chiral algebra structure on $\omega_X$. A chiral algebra $\cA$ is called \emph{unital} if it comes endowed with a morphism of chiral algebras $\omega_X \hookrightarrow \cA$. We will omit the word unital below. 
\label{defn:chiral_def}
\end{defn}
\begin{rem}
One of the subtleties that occur in the supercurve case as opposed to the non-super situation is that the unit of a chiral algebra has a non-trivial parity (it is odd on curves of dimension $1|2k+1$). This is due to the non-standard equivalence of categories of  \emph{right} $\cD_X$-modules and \emph{left} ones. In fact, this phenomenon occurs in a similar fashion in the non-super case of Beilinson and Drinfeld in the definition of the $\otimes^!$ tensor product of the category of right $\cD_X$-modules. Left $\cD_X$-modules have a natural tensor product structure such that $\cO_X$ is the unital object in the category. Using any equivalence of categories to right $\cD_X$-modules one obtains a tensor structure on the latter. Instead of using the \emph{canonical} equivalence (2.1.1.1) of op.cit. $L \mapsto L \otimes \omega_X [\dim X]$ between left and right $\cD_X$-modules, the authors insisted on using the non canonical (but non-derived) equivalence $L \mapsto L \otimes \omega_X$. This has the advantage that the unital object is a sheaf (i.e., $\omega_X$) instead of a complex (i.e., $\omega_X[\dim X]$) at the expense that the signs in the action of the symmetric group on tensor products becomes more difficult to track. 

In the super case, in addition to the above mentioned problem, we have that the natural equivalence of categories between left and right $\cD_X$ modules carries a \emph{parity shift} if the dimension of $X$ is of the form $n | 2k +1$. The fact that the chiral unit has a non-trivial parity comes from  insisting in using a non-canonical equivalence of categories, so that $\omega_X$ is the unital object instead of $\Pi^m \omega_X[n]$ when $\dim X = n|m$. One could solve this problem by changing the parity of the module $\omega_X$ without shifting in the cohomological direction, however, as we will see shortly, the equivalence of categories between chiral and vertex algebras involves passing to left $\cD_X$-modules, and the solution to the parity problem on one side would express itself in the other side, for example by having the vacuum vector being an odd vector or the state-field correspondence reversing parity. 
\label{rem:parity}
\end{rem}
\begin{nolabel}
Applying the De Rham functor to a chiral algebra we obtain a sheaf of $k$-vector spaces $h(\cA)$ on $X$. The operation $\mu$ induces a Lie superalgebra structure on $h(\cA)$. Applying the De Rham functor along the first factor of $X^2$ and composing with the chiral multiplication $\mu$ we obtain an action of $h(\cA)$ on $\cA$ and in particular on its fiber $\cA_x$ at a point $x \in X$. Let $X$ be a smooth supercurve of dimension $1|n$, $\cA$ be a chiral algebra on $X$, and $x \in X$ be a closed point. We define the space of conformal blocks as \begin{equation} \label{eq:conformal-definition} C(X,x;\cA) := \Hom_{h(\cA)(X \setminus x)} (\cA_x, k). \end{equation} Here, as before, $X \setminus x$ denotes the open sub-superscheme of $X$ obtained by the restriction of $\cO_X$ to $X_\mathrm{rd} \setminus x_\mathrm{rd}$. 
\label{no:5.6}
\end{nolabel}
\begin{nolabel}
In this and in the next four sections we describe the connection between SUSY vertex algebras and chiral algebras. For simplicity we restrict ourselves to general $1|1$ dimensional supercurves and $N_W=1$ SUSY vertex algebras. For other kinds of supercurves and SUSY vertex algebras we refer the reader to \cite{heluani4}. A detailed and clear account of the non-super case is given in \cite{frenkelzvi}, and the supercase is a simple generalization modulo subtleties regarding signs. Let $X$ be a smooth supercurve of dimension $1|1$ over $\Spec S$ for some supercommutative ring $S$ which we can take to be a Grassmann algebra over $k = \mathbb{C}$. We allow $S$ to have odd elements in order to consider them as ``odd constants'' in what follows.  Let $x \in X$ be a closed $S$-point. We have the local ring $\cO_x$ and its ring of fractions $\cK_x$. Choosing local coordinates $\bt = (t, \zeta)$ near $x$ we obtain isomorphisms $\cO_x \simeq \cO_S^{1|1} := S [ [t]][\zeta]$ and $\cK_x \simeq \cK_S^{1|1} := S ( (t))[\zeta]$. Consider the group of local changes of  coordinates at $x$, i.e., of continuous automorphisms of $\cO_x$. This group consists of formal power series $\rho = \rho(\bt) = \bigl( F(t,\zeta), \Psi(t,\zeta)\bigr)$ with $F$ even and $\Psi$ odd, with nonzero Berezinian at $\bt = (0,0)$. Indeed if we write
\[
F(t,\zeta) = a_1 t^1 + \alpha_1 \zeta + a_2 t^2 + \alpha_2 t \zeta + \dots, \qquad \Psi(t,\zeta) = \beta_1 t^1 + b_1 \zeta + \beta_2 t^2 + b_2 t \zeta + \dots,
\]
where $a_i, b_i \in S_{\bar{0}}$ and $\alpha_i, \beta_i \in S_{\bar{1}}$, then the condition on $\rho$ is that
\[ \Ber \rho = \Ber 
\begin{pmatrix}
a_1 & \alpha_1 \\ 
\beta_1 & b_1 
\end{pmatrix} \in S_{\bar{0}}^*. \]
The corresponding Lie algebra $\lie \Aut \cO = \mathrm{Der}_0 \cO$ is naturally isomorphic to the positive part of the topological Lie algebra \eqref{eq:topological-commutators}, that is the subalgebra generated by $L_n, Q_n, J_n, H_n$ with $n \geq 0$. The isomorphism is given by 
\begin{equation}\label{eq:topological-derivations} \begin{aligned}
L_n &= - t^{n+1} \partial_t - (n+1) t^n \zeta \partial_\zeta, &&& J_n &= -t^n \zeta \partial_\zeta, \\ Q_n &= - t^{n+1} \partial_\zeta, &&& H_n &= t^n \zeta \partial_t. 
\end{aligned}
\end{equation}
These are the vector fields on the formal neighborhood of $x\in X$ vanishing at $x$. This Lie algebra is a semidirect product of $\gl(1|1)$ generated by $L_0, J_0, H_0, Q_0$ and a pro-nilpotent Lie algebra. This is reflected in the fact that $\Aut \cO$ is the semidirect product of $GL(1|1)$ (linear changes of coordinates) with a pro-unipotent group (changes of coordinates whose first jet is the identity). According to Definition \ref{defn:topological-quasi-conformal} every quasi-conformal $N_W=1$ SUSY vertex algebra $V$ is naturally an $\Aut \cO$-module. Similarly for a positive energy module $M$. Recall that $V$ and $M$ also have actions of $L_{-1}=- \partial_t $ and $Q_{-1}=-\partial_\zeta$.  

Every smooth curve has a natural $\Aut \cO$-torsor $\widehat{X}$ given by pairs $(x, \bt)$ where $x \in X$ is a point and $\bt = (t, \zeta)$ are local coordinates at $x$. This is called the \emph{Gelfand-Kazhdan} structure \cite{gelfandkazhdan} of $X$ and in fact is well defined for any smooth superscheme of dimension $n|m$ (not necessarily $1|m$). For a discussion in this context and the fact that $\widehat{X}$ is a well defined $\Aut \cO$ torsor  (locally trivial in the Zariski topology) we refer the reader to \cite[$\S$ 2.9.9]{beilinsondrinfeld}, \cite{frenkelzvi} in the non-supercase and \cite{heluani4} for supercurves.  It follows that every SUSY vertex algebra $V$ gives rise to a quasi-coherent sheaf (an infinite rank vector bundle)
\[ \cV = \widehat{X} \overset{\Aut \cO}{\times} V. \]
This vector bundle carries a flat connection $\nabla$ locally given in coordinates by
\[ \nabla = d + L_{-1} dt + Q_{-1} d \zeta. \]
Flatness follows from the fact that $L_{-1}$ and $Q_{-1}$ commute. It follows that $\cV$ is a left $\cD_X$-module.  
\label{no:5.7}
\end{nolabel}
\begin{nolabel}\label{no:huangs}
The essential part of the construction is the fact that the state field correspondence of $V$ produces a well defined section of $\cV^*$. Let us describe this in the analytic topology. Any coordinate system $\bt=(t,\zeta)$ near $x \in X$ trivializes $\widehat{X}$ and therefore identifies the fiber $V \overset{\sim}\rightarrow \cV_x $. Using this trivialization, sections of $V$ on the formal punctured disk $\cD_x^\times := \Spec \cK_x$ near $x$ are identified with $V \otimes \cK_x$. Therefore the state-field correspondence map 
\[ a \mapsto Y(a,t,\zeta), \] can be viewed as an $\End \cV_x$-valued section of $\cV^*$: $\cY_x (\cdot)$. We have 
\begin{thm*}\cite[Thm. 3.5]{heluani4}
$\cY_x(\cdot)$ is independent of the coordinates $\bt$ chosen, i.e. it is a well defined $\End \cV_x$-valued section of $\cV^*$ on $\cD_x^\times$. 
\end{thm*}
The key ingredient in the proof of this Theorem is a super-field version of Huang's change of coordinates formula. For any $\rho \in \Aut \cO$ and $a \in V$ we have:
\begin{equation}\label{eq:huangs}
Y(a, \bt) = \rho Y\bigl( \rho_\bt^{-1} a, \rho(\bt) \bigr) \rho^{-1}
\end{equation}
Here we abuse notation by identifying $\rho$ with its corresponding action on $V$ obtained by exponentiating the $\Der_0 \cO$-action. For a coordinate $\bt$ the automorphism $\rho_{\bt} \in \Aut \cO$ is defined by $\rho_\bt (\bx) = \rho(\bx + \bt) - \rho(\bx)$, where $\bx = (x,\theta)$, $\bt = (t , \zeta)$, and $\bx + \bt = (x + t, \theta + \zeta)$. If $\rho = (F, \Psi)$ then $\rho_{\bt} = (F_{\bt}, \Psi_{\bt})$ where $F_{\bt}(\bx) = F(\bx + \bt) - F(\bt)$, and $\Psi_{\bt} (\bx) = \Psi (\bx + \bt) - \Psi(\bt)$. 
\end{nolabel}
\begin{nolabel}\label{no:filtration}
The bundle $\cV$ has infinite rank,  it carries a filtration by $\cV_{\leq n}$ defined to be the $\Aut \cO$-submodule of $V$ generated by all vectors of conformal weight at most $n$. The corresponding quotients $\cV_{\leq n+1}/\cV_{\leq n}$ are finite rank vector bundles on $X$. 
\end{nolabel}
\begin{nolabel}\label{no:chiral-operation}
The sections $\cY_x(\cdot)$ can be globalized by moving the point $x \in X$, in this way we recover the chiral operation $\mu$ on the right $\cD_X$-module $\cA := \cV \otimes \omega_X$. In local coordinates $\mu$ takes the following form. We let $\bt = (t,\zeta)$ be coordinates near $x \in X$ and consider another set of coordinates $\bx= (x, \theta)$. We have a trivialization of $\omega_{X^2}$ near the diagonal  point $(x, x) \in \Delta \subset X^2$ given by the section of the Berezinian bundle $[d\bt d\bx]:=[dt dx d\zeta d\theta]$. Using these coordinates we obtain a local trivialization of the $\cD_{X^2}$-module $\cA \boxtimes \cA$. Its local sections are given by $f(\bt,\bx) (a \otimes [d\bt]) \boxtimes (b \otimes [d\bx])  = f(\bt,\bx)  (a \boxtimes b) \otimes [d\bt d\bx]$ where we used the natural identification $\omega_X \boxtimes \omega_X \simeq \omega_{X^2}$ and $f$ is a regular function of $X^2$ near $(x,x)$ and $a,b \in V$. Since we are interested in $j_* j^* \cA \boxtimes \cA$ as in \eqref{eq:chiral-operations} we will allow $f$ to have poles at $t=x$.  Similarly we obtain a local trivialization of the $\cD_{X^2}$-module  $\Delta_* \cA$. Its local sections are given by $f(\bt,\bx) [d\bt] \boxtimes a [d\bx] \mod \text{reg}$ where we allow $f$ to have poles at $t =x$ and $\mod \text{reg}$ means that we mod out by regular sections of $\omega_X \boxtimes \cA$ (that is local sections with $f$ regular).  We finally let 
\[  \mu : f(\bt,\bx) (a \boxtimes b) \otimes  [d\bt d\bx] \mapsto f(\bt, \bx) [d \bt] \boxtimes Y(a, \bt-\bx)b [d\bx] \mod \text{reg}.\]
We have
\begin{thm*}\cite[Thm 4.12]{heluani4} \cite[Thm 18.3.3]{frenkelzvi}
The map $\mu$ so defined is independent of the coordinates chosen and gives rise to a chiral algebra on the supercurve $X$ as in definition \ref{defn:chiral_def}. 
\end{thm*}
\end{nolabel}
\begin{nolabel}\label{no:universal property}
The Harish-Chandra setting of \cite[$\S$ 2.9.9]{beilinsondrinfeld} applies and the chiral algebra $\cA$ that one obtains from a quasi-conformal $N_W=1$ SUSY vertex algebra $V$ is \emph{universal} in the following sense. Our $V$ gives rise to a rule that assigns to each smooth family $X/S$ of relative dimension $1|1$ a quasi-coherent left $\cD_{X/S}$-module $\cV_{X/S}$ and to each fibrewise \'etale morphism of families $f: X'/S' \rightarrow X/S$ an identification $\cV_{X'/S'} \overset{\sim}{\rightarrow} f^* \cV_{X/S}$. 
\end{nolabel}

\begin{ex} \label{ex:5.8}
Let $V$ be the topological vertex algebra of example \ref{ex:examples-conformal-algebras}(g). It is conformal as defined in \ref{defn:topological-quasi-conformal}.  Consider a curve $X$ and the corresponding chiral algebra $\cA$ over $X$. The subspace $V_{\leq 1}$ generated by $\vac, h, j$ is a $2|1$ dimensional $\Aut \cO$-submodule of $V$ which generates a sub $\cD_X$-module of $\cA$. As a $\cO_X$-module, this module is an extension by $\omega_X$ (generated by the vacuum vector) of some rank $1|1$ locally free $\cO_X$-module which we proceed to identify.

Let $\bt = (t, \zeta)$ be local parameters near some point of $X$ and consider $\rho = (F, \Psi) \in \Aut \cO$ some change of coordinates $\bt  \mapsto \bigl(F(\bt), \Psi(\bt) \bigr)$. We expand
\[
\rho_\bt (\bx) = \bigl(F_\bt(\bx), \Psi_\bt(\bx) \bigr) := \bigl( F(\bt + \bx) - F(\bt), \Psi(\bt + \bx) - \Psi(\bt)\bigr)
\]
as a power series in $\bx = (x, \theta)$ as
\begin{equation}\label{eq:5.8.1}
\begin{aligned}
F_\bt (\bx) &= x \partial_t F(\bt) + \theta \partial_\zeta F(\bt) + \theta x \partial^2_{\zeta,t} F(\bt) + \frac{x^2}{2} \partial^2_{t} F(\bt) + \dots \\
\Psi_\bt (\bx) &= x \partial_t \Psi(\bt) + \theta \partial_\zeta \Psi(\bt) + \theta x \partial^2_{\zeta,t} \Psi(\bt) + \frac{x^2}{2} \partial^2_{t} \Psi(\bt) + \dots \\
\end{aligned}
\end{equation}
On the other hand $\rho_\bt = \exp (\mathbf{v}(\bt))$ for some $\mathbf{v} \in \mathrm{Der}_0 \cO = \mathrm{Lie} \Aut \cO$. We may take the general form
\begin{multline*}
\rho_\bt (\bx) = \exp \left(\sum_{i \geq 1} \tau_i (x^{i+1} \partial_x + (i+1) x^i\theta \partial_\theta) + \alpha_i x^i\theta_{\partial_\theta} + \varepsilon_i x^i \theta \partial_{x} + \delta_i x^{i+1} \partial_{\theta} \right) \\ \times (1 + \varepsilon_0 \theta \partial_x) q^{(x \partial_x + \theta \partial_\theta)} y^{\theta \partial_\theta} (1 + \delta_0 x \partial_\theta) \bx,
\end{multline*}
and expand to obtain
\begin{align*}
F_\bt(\bx) &= q x + q \varepsilon_0 \theta + q \left( \varepsilon_0 (\alpha_1+2 \tau_1) + \varepsilon_1 \right) \theta x + q \left(\tau_1 + \varepsilon_0 \delta_1  \right) x^2 + \dots \\ 
\Psi_\bt (\bx) &= q \delta_0 x + q (y-\varepsilon_0 \delta_0) \theta \\ & \quad \quad + q\left( \delta_0 \varepsilon_1 + (y-\varepsilon_0\delta_0)(\alpha_1 + 2\tau_1) \right) \theta x + q \left(\delta_0 \tau_1 + (y - \varepsilon_0 \delta_0)\delta_1 \right) x^2 + \dots.
\end{align*} 
Equating coefficients with \eqref{eq:5.8.1} then yields 
\begin{equation} \label{eq:5.1.1}
\begin{gathered}
q = \partial_t F, \qquad   \varepsilon_0 = - \frac{\partial_\zeta F}{\partial_t F}, \qquad   \delta_0 = \frac{\partial_t \Psi}{\partial_t F}, \qquad  y = \frac{\partial_\zeta \Psi}{\partial_t F} - \frac{\partial_\zeta F \partial_t \Psi}{(\partial_t F)^2} = \Ber \rho^{-1}\\
\tau_1 + \varepsilon_0 \delta_1 = \frac{1}{2} \frac{\partial^2_{t,t} F}{\partial_tF}, \qquad \varepsilon_0(\alpha_{1} + 2 \tau_1) + \varepsilon_1 = \frac{\partial^2_{t,\zeta} F}{\partial_t F} \\ 
(\partial_t \Psi) \tau_1 + (\partial_\zeta \Psi) \delta_1 = \frac{1}{2} \partial^2_{t,t} \Psi, \qquad (\partial_t \Psi) \varepsilon_1 + (\partial_\zeta \Psi) (\alpha_1 + 2 \tau_1) = \partial^2_{t,\zeta} \Psi
\end{gathered}
\end{equation}
Note that the Berezinian of the change of coordinates $\rho$ is given by 
\begin{equation}
\Ber \rho = \Ber 
\begin{pmatrix}
\partial_t F & \partial_t \Psi \\ \partial_\zeta F & \partial_\zeta \Psi
\end{pmatrix} = \frac{\partial_t F}{\partial_\zeta \Psi} + \frac{\partial_\zeta F \partial_t \Psi}{(\partial_\zeta \Psi)^2}.
\label{eq:5.8.2}
\end{equation}
Using \eqref{eq:topological-derivations} we compute the action of 
\[ \rho_\bt^{-1} = \exp(\delta_0 Q_0) y^{J_0} q^{L_0} \exp(-\varepsilon_0 H_0) \exp \left( \sum_{i \geq 1} \tau_i L_i + \alpha_i J_i - \varepsilon_i H_i + \delta_i Q_i \right)\]
on the vector space spanned by the vectors $\vac,h,j$ to be left multiplication by the matrix 
\begin{equation}\label{eq:5.8.2b}
\begin{pmatrix}
1 & \frac{c}{3}\delta_1 & \frac{c}{3} (\alpha_1 +\tau_1) \\ 
0 & q y^{-1} &  q y^{-1} \varepsilon_0 \\ 
0 & q y^{-1} \delta_0 & q \left( 1 - y^{-1} \varepsilon_0 \delta_0 \right)
\end{pmatrix}.
\end{equation}
We notice that the lower $2\times 2$ block is given by 
\begin{equation}\label{eq:5.8.3}
q y^{-1}
\begin{pmatrix}
1  & \varepsilon_0 \\ \delta_0 & y - \varepsilon_0 \delta_0 
\end{pmatrix} = \Ber \rho \cdot 
\begin{pmatrix}
\partial_t F & - \partial_\zeta F \\  \partial_t \Psi  & \partial_\zeta \Psi
\end{pmatrix}
\end{equation}
which are the transition functions of $\omega_X \otimes \Omega^{1}_X$. It follows that $\cA_{\leq 1}$ is an extension
\[ 0 \rightarrow \omega_X \rightarrow \cA_{\leq 1} \rightarrow \cT_X \rightarrow 0,\]
of the tangent bundle of $X$ by its Berezinian bundle showing that the topological algebra \ref{ex:examples-conformal-algebras} g) gives the supersymmetric analog of the Virasoro extension of \cite[$\S$ 2.5.10]{beilinsondrinfeld}. 
\end{ex}
\begin{nolabel}\label{no:subtlety}
There is a subtlety regarding the rule of signs when one considers matrices with entries in a supercommutative ring. Just as in the commutative case, every left module can be made into a right module in a compatible manner, however care must be taken on the signs that appear. For a full account we refer the reader to \cite{deligne2,manin2}. In the case at hand in Example \ref{ex:5.8} we consider $V$ and its subspace $V_{\leq 1}$  with basis $\left\{ \vac, h, j \right\}$ as right modules over the algebra $\mathbb{C}[\varepsilon, \delta]$. As such we write a vector of coordinates as a vertical matrix $\begin{pmatrix} \alpha & \beta & \gamma \end{pmatrix}^T$ to denote the vector $\vac \alpha + h  \beta + j \gamma$. An endomorphism of this space is written as a $3\times 3$ matrix with entries in $\mathbb{C}[\varepsilon,\delta]$ acting by multiplication on the left on the column of \emph{right coordinates}.  This explains the signs in the $2,3$-entry of \eqref{eq:5.8.2b} and the $1,2$-entry of \eqref{eq:5.8.3} which might seem strange at first sight. 
\end{nolabel}

\section{Families of elliptic (super) curves}\label{sec:elliptic-super}
In this section we construct two families $M^\text{even}$ and $M$ of elliptic supercurves, and we describe the chiral algebras over these families corresponding to a quasi-conformal $N_W=1$ SUSY vertex algebra. Moduli spaces of elliptic supercurves have been considered by many authors. In the case of $N=1$ supercurves or $N=2$ SUSY curves these moduli spaces where studied in detail in \cite{melzer}. Recently these moduli spaces have caught some attention in the physics literature (see \cite{witten-donagi1,witten-donagi2} and references therein) 
\begin{nolabel}
As a preliminary we describe the non super case. Let $V$ be a quasi-conformal vertex algebra and let $\cA$ be the chiral algebra over $\mathbb{A}^{1}$ corresponding to $V$. This chiral algebra is equivariant with respect to the group $Aff$ of affine transformations of $\mathbb{A}^{1}$. We have an action of $\mathbb{G}_m$ on $V$ defined by $a \mapsto q^{-L_0}a$ for $q \in \mathbb{G}_m$. After fixing a global coordinate $t$ on $\mathbb{A}^1$ we also have an action of $\mathbb{G}_m$ on $\mathbb{A}^1$ defined by $t \mapsto qt$. The coordinate $t$ trivializes $\mathbb{A}^1 \times V \simeq \cV$ via the map $(x,a) \mapsto (t_x, a)$, where $t_x$ is the induced local coordinate at $x \in \mathbb{A}^1$. It also furnishes a trivialization $dt$ of $\omega_{\mathbb{A}^1}$, and hence of $\cA = \cV \otimes \omega_{\mathbb{A}^1}$. This allows us to identify sections of $\cA$ with functions $a(x): \mathbb{A}^1 \rightarrow V$. Under this identification the isomorphism $q^* : \cA \rightarrow \cA$ corresponding to $q \in \mathbb{G}_m$ is $a(x) \mapsto q^{1-\Delta_a}a(qx)$, where $\Delta_a$ is the conformal weight of $a$.

In what follows we will consider $\cA$ as a chiral algebra on $\mathbb{G}_m \subset \mathbb{A}^1$ by restriction. We will also fix $q \in \mathbb{G}_m$ and consider the action of $\mathbb{Z}$ defined by 
\begin{equation}
n : a(x) \mapsto q^{n- n\Delta_a} a(q^n x).
\label{eq:x.1.1}
\end{equation}
\label{no:x.1}
\end{nolabel}
\begin{nolabel}
The exponential map $\exp : \mathbb{A}^1 \rightarrow \mathbb{G}_m$ identifies the translation $\tau \in \mathbb{G}_a$ with the automorphism $q$ above, where $q = e^{2\pi i \tau}$. 
The universal property \ref{no:universal property} implies an isomorphism $\cA_{\mathbb{A}^1}^l \simeq \exp^* \cA_{\mathbb{G}_m}^l$. After using the global coordinate $t$ to identify sections with functions, this isomorphism matches the constant section $a(x) = a$ of $\cA_{\mathbb{A}^1}^l$ with the section $a(x) = x^{\Delta_a} a$ of $\cA_{\mathbb{G}_m}^l$. Between the chiral algebras $\cA_{\mathbb{A}^1}$ and $\cA_{\mathbb{G}_m}$ the constant section $a$ is matched with $x^{\Delta_a - 1} a$. The action (\ref{eq:x.1.1}) corresponds in this way to $a(x) \mapsto a(x+n\tau)$. A priori this action is simpler, since it can be defined without the need of a quasi-conformal structure on $V$. However, the quasi-conformal structure is required to make the identification of chiral algebras associated to the exponential map.
\label{no:x.1b}
\end{nolabel}
\begin{nolabel}
Now we consider the relative situation. Let $S = \mathbb{H} \subset \mathbb{C}$ be the upper half complex plane and let $X = \mathbb{G}_m \times S$ be the trivial family over $S$, with global coordinate $t$. We put $q = e^{2 \pi i \tau}$ as usual. We have an action of $\mathbb{Z}$ on $X/S$ by $n : (x, \tau)  \mapsto (q^n \cdot x, \tau)$. Let $V$ and the action of $\mathbb{G}_m$ on $V$ be as above, and let $\cA$ be the chiral algebra over $X$ corresponding to $V$. Global sections of $\cA$ are now identified with functions $a(x;q): X \rightarrow V$, and the induced action $n^* \cA \simeq \cA$ on sections is given by \eqref{eq:x.1.1}. Often we will consider sections on $X$ flat over $S$ and will write them as $a(x)$ instead of $a(x;q)$.
\label{no:x.1c}
\end{nolabel}
\begin{nolabel}
Upon taking the quotient $\pi: X \twoheadrightarrow E := X/\mathbb{Z}$ we obtain a \emph{versal} family $E / S$ of elliptic curves. We let $\cA_E$ be the chiral algebra on $E$ associated to $V$. From the universal property \ref{no:universal property} we obtain an isomorphism $\cA \simeq \pi^* \cA_E$ which identifies sections of $\cA_E$ with $\mathbb{Z}$-equivariant sections of $\cA$ over $X$.

We may construct such a section as the sum
\begin{equation}
\sum_{n \in \Z} q^{n - n \Delta_a} a(q^n x)
\label{eq:x.1.d.1}
\end{equation}
of the $\Z$-translates of a given section $a(x)$ (flat over $S$). Whenever this sum is convergent we obtain a well defined $\Z$-equivariant section of $\cA$ over $X$, which is the same as a section of $\cA_E$ over $E$ in the appropriate category (eg. analytic, algebraic, etc). 
\label{no:x.1.d}
\end{nolabel}
\begin{nolabel}\label{no:usualuniversal}
To obtain a universal family of elliptic curves we must take a further quotient of $E/S$ by an action of the modular group $SL(2,\mathbb{Z})$ defined by
\begin{equation} \begin{pmatrix} a & b \\ c & d \end{pmatrix} \cdot ( \tau, t) = \left( \frac{a \tau + b}{c \tau + d}, \frac{t}{c \tau+ d} \right), \qquad ad - bc = 1.
\label{eq:sl2ext1}
\end{equation}
The $SL(2, \Z)$ is nontrivial on the base $S$ and the quotient $S / SL(2, \Z)$ is the coarse moduli space $M_{1,1}$ of elliptic curves. The orbifold quotient $\cE := E \git SL(2,\mathbb{Z})$ is the universal family of elliptic curves. 
 \end{nolabel}
\begin{nolabel}
The coarse moduli space $M_{1,1}$ has a compactification $\overline{M_{1,1}}$ obtained by adding the cuspidal elliptic curve at $q \rightarrow 0$. We wish to construct sections of $\cA$ that extend to this cuspidal curve. In other words we wish to construct sections possessing an expansion in positive powers of $q$ in \eqref{eq:x.1.d.1}.
\begin{lem*}
Let $z \in \mathbb{C}^*$, $a \in V$ of conformal weight $\Delta$, and let $k,l \geq 0$. Consider the section \[ a(x) = \frac{x^{k}a}{(x-z)^l} \in H^0(X \setminus z, \cA), \]which has a pole of order $l$ at the point $z$. If $\Delta + l > k+1 > \Delta$ then the sum \eqref{eq:x.1.d.1} is a well defined section in the formal neighborhood of the cuspidal curve in $\overline{M_{1,1}}$, that is viewed as a formal series, it does not have negative powers of $q$. Choosing $k = \Delta$ and $l=2$ is always a solution to this system of inequalities and there is no solution for $l = 1$. 
\end{lem*}
\label{no:x.1.e}
\end{nolabel}
\begin{proof}
The sum \eqref{eq:x.1.d.1} is in this case
\[
\sum_{n \in \mathbb{Z}} \frac{q^{n - n\Delta + n k} x^{k}a}{(q^n x - z)^l}.
\]
The numerator has only positive powers of $q$ as long as $n (k+1) >  n \Delta$. If $n> 0$
the denominator is regular as $q \rightarrow 0$ since $z \in \mathbb{C}^*$. It follows that the summands with $n > 0$ vanish at $q =0$ as long as $k+1 > \Delta$. When $n < 0$ the denominator has a zero of order $nl$ at $q=0$, it follows that the quotient has only positive powers of $q$ as long as $k + 1 - l < \Delta$. 
\end{proof}

\begin{nolabel}
Now we carry out similar constructions in the supersymmetric case. Let $V$ be a quasi-conformal $N_W=1$ SUSY vertex algebra and let $\cA$ be the corresponding chiral algebra over $X=\mathbb{A}^{1|1}$. This chiral algebra is equivariant with respect to the group $Aff^{1|1}$ of affine transformations of $\mathbb{A}^{1|1}$. Consider the even subgroup $\mathbb{G}_m \times \mathbb{G}_m \subset \mathbb{G}_m^{1|1}$, it acts on $V$ by $a \mapsto q^{-L_0} y^{-J_0} a$ for $(q, y) \in \mathbb{G}_m \times \mathbb{G}_m$. After fixing a global coordinate $\bt = (t,\zeta)$ on $\mathbb{A}^{1|1}$ we have an action
\begin{equation}
(t,\zeta) \mapsto (q t, q y \zeta),
\label{eq:x.2.0}
\end{equation}
of $\mathbb{G}_m \times \mathbb{G}_m$ on $\mathbb{A}^{1|1}$. The coordinates $(t,\zeta)$ trivialize $\mathbb{A}^{1|1} \times V \simeq \cA^l$ via the map $(x,\theta;a) \mapsto (t_x, \zeta_\theta; a)$ where $(t_x, \zeta_\theta)$ are the induced local coordinates at $(x,\theta)$. We also obtain from $(t, \zeta)$ the global trivialization $[dtd\zeta]$ of $\omega_{\mathbb{A}^{1|1}}$, and hence of $\cA = \cA^l \otimes \omega_{\mathbb{A}^{1|1}}$. This allows us to identify sections of $\cA$ with functions $a(x, \theta) : \mathbb{A}^{1|1} \rightarrow V$. Under this identification the isomorphism $\bq^* \cA \simeq \cA$ corresponding to $\mathbf{q}=(q,y) \in \mathbb{G}_m \times \mathbb{G}_m$ is $a(x,\theta) \mapsto q^{- \Delta_a} y^{-1-c_a} a(q x,q y \theta)$ (note that the Berezinian of \eqref{eq:x.2.0} is $y^{-1}$). Here $\Delta_a$ and $c_a$ are the conformal weight and charge of $a$ respectively.

As before we consider $\cA$ as a chiral algebra on $\mathbb{G}_m^{1|1} \subset \mathbb{A}^{1|1}$ by restriction. We also fix $\bq =(q,y) \in \mathbb{G}_m \times \mathbb{G}_m$ and consider the action of $\mathbb{Z}$ defined by
\begin{equation}
n : a(x,\theta) \mapsto q^{ - n\Delta_a} y^{-n - n c_a} a(q^n x, q^n y^{n}\theta).
\label{eq:x.2.1}
\end{equation}
\label{no:x.2}
\end{nolabel}
\begin{nolabel}
As in \ref{no:x.1b} we might have taken an $N_W=1$ SUSY vertex algebra $V$ (a priori without quasi-conformal structure) to obtain a chiral algebra on $\mathbb{A}^{1|1}$ equivariant by the translation group $\mathbb{G}^{1|1}_a$. If we equip $V$ with a quasi-conformal structure then we may use the exponential change of coordinates $\exp : \mathbb{G}_a^{1|1} \rightarrow \mathbb{G}_m^{1|1}$ from \ref{no:supergroups} and the universal property \ref{no:universal property} to recover the action \eqref{eq:x.2.1} of $\mathbb{G}_m^{1|1}$ in the exponentiated coordinates.
\label{no:x.2.b}
\end{nolabel}
\begin{nolabel}
Now let $S = S^{2|0} := \mathbb{H} \times \mathbb{C}$ and let $X = \mathbb{G}_m^{1|1} \times S$ be the trivial family with fiber $\mathbb{G}_m^{1|1} \simeq \mathbb{C}^* \times \mathbb{C}^{0|1}$, and let $(t, \zeta)$ be a choice of global coordinate for $\mathbb{G}_m^{1|1}$. For $\tau \in \mathbb{H}$ and $\alpha \in \mathbb{C}$ we put $q = e^{2 \pi i \tau}$, $y = e^{2 \pi i \alpha}$. We have an action of $\mathbb{Z}$ on $X/S$ by $n : \bigl( (x,\theta), (q,y) \bigr) \mapsto \bigl( (q^n x, q^n y^{n} \theta), (q,y) \bigr)$.

Let $V$ be a quasi-conformal $N_W=1$ SUSY vertex algebra and let $\cA$ be the chiral algebra over $X$ corresponding to $V$. The trivialization of $\cA$ induced by the global coordinates $(t,\zeta)$ allows us to write sections as functions $a(\mathbf{x},\bq): X \rightarrow V$ (here and below we write $\mathbf{x}$ for $(x,\theta) \in \mathbb{G}_m^{1|1}$ and $\bq$ for $(q,y) = (e^{2 \pi i \tau},e^{2 \pi i\alpha}) \in S$). The induced action of $\Z$ on sections is given by \eqref{eq:x.2.1}. 

We define the family $E^0$ of superelliptic curves to be the quotient $X \twoheadrightarrow E^0 := X/\mathbb{Z}$. Since the base $S^{2|0}$ of $E^0$ is even we refer to it as the \emph{even} family of elliptic supercurves. Let $\cA_{E^0}$ be the chiral algebra over $E^0$ associated to $V$. The universal property \ref{no:universal property} identifies sections of $\cA_{E^0}$ with $\mathbb{Z}$-equivariant sections of $\cA$ over $X$.

As in \ref{no:x.1.d} we may construct such an equivariant section as the sum 
\begin{equation}
\sum_{n \in \mathbb{Z}} q^{- n \Delta_a} y^{-n - n c_a} a( q^n x, q^n y^{n} \theta). 
\label{eq:x.2.3}
\end{equation}
of $\Z$-translates of a given section $a(x,\theta)$ of $\cA$ over $X$.
\label{no:x.2.c}
\end{nolabel}
\begin{nolabel}
The quotient $E^{0} = X/\mathbb{Z}$ may also be viewed as a quotient $\mathbb{A}^{1|1} \times S/\mathbb{Z}^2$ where the action of $\mathbb{Z}^2$ is generated by
\begin{equation}
A: (t,\zeta ; \tau, \alpha) \mapsto (t + 1, \zeta; \tau, \alpha), \quad \text{and}  \quad B: (t,\zeta; \tau, \alpha) \mapsto (t + \tau, e^{2 \pi i \alpha } \zeta; \tau, \alpha). 
\label{eq:z2ineven}
\end{equation}
\label{no:x.2.f}
\end{nolabel}
\begin{nolabel}
The base $S^{2|0}$ of our families carries an action of $SL(2,\mathbb{Z})$ 
defined by
\begin{equation} \begin{pmatrix} a & b \\ c & d \end{pmatrix} \cdot ( \tau, \alpha) = \left( \frac{a \tau + b}{c \tau + d}, \frac{\alpha}{c \tau+ d} \right), \qquad ad - bc = 1.
\label{eq:sl2ext2}
\end{equation}
It also carries an action of $\mathbb{Z}^2$ defined by
\begin{equation}\label{eq:sl2ext3} \gamma \cdot (\tau, \alpha) := (\tau, \alpha + m \tau +n), \qquad \gamma= (m,n) \in \mathbb{Z}^2.\end{equation}
It is easy to check that these combine into an action of the Jacobi group $SL(2, \mathbb{Z}) \ltimes \mathbb{Z}^2$. The product in this group is
\[
(A, \gamma) \cdot (A', \gamma') = (A A', \gamma A' + \gamma')
\]
where $A, A' \in SL(2, \mathbb{Z})$ and $\gamma, \gamma' \in \mathbb{Z}^2$. We claim that this action of $SL(2, \mathbb{Z}) \ltimes \mathbb{Z}^2$ on $S$ extends to an action on the family $E^{0}$. Indeed we have the following.
\begin{lem*}
Introduce the maps
\begin{equation} \label{eq:strangeequiva}
(t, \zeta) \mapsto \gamma\cdot (t,\zeta) := \left( t, e^{2 \pi i m t} \zeta \right),
\end{equation}
for $\gamma = (m,n) \in \mathbb{Z}^2$, and
\begin{equation}
A \cdot (t, \zeta) = \left( \frac{t}{c \tau + d}, e^{- 2 \pi i t  \frac{c \alpha}{c \tau + d}}\zeta   \right), \qquad ad-bc = 1, 
\label{eq:strangeequiv}
\end{equation}
for $A \in SL(2, \mathbb{Z})$. These induce isomorphisms $\gamma:E^{0}_{(\tau, \alpha)} \simeq E^{0}_{\gamma (\tau, \alpha)}$ and $A : E^{0}_{(\tau, \alpha)} \simeq E^{0}_{A \cdot (\tau, \alpha)}$ respectively. Here $E^{0}_{(\tau,\alpha)}$ is the fiber of $E^{0}$ at the point $(\tau, \alpha) \in S$. These isomorphisms are compatible in the sense that for $A, A' \in SL(2,\mathbb{Z})$ and $\gamma, \gamma' \in \mathbb{Z}^2$ we have
\begin{align*}
A \circ A' (t, \zeta) = (AA') (t, \zeta), \quad
\gamma \circ \gamma' (t, \zeta) = (\gamma + \gamma') (t, \zeta), \quad
\text{and} \quad
\gamma \circ A (t, \zeta) = A \circ (\gamma A) (t, \zeta).
\end{align*}
\end{lem*}

\begin{proof}
The first isomorphism is easy to see. Let us prove that \eqref{eq:strangeequiv} induces an isomorphism $E^{0}_{(\tau, \alpha)} \cong E^{0}_{(\tau', \alpha')}$ where $(\tau', \alpha') = A \cdot (\tau, \alpha)$ is given by \eqref{eq:sl2ext2}.  We start with the ansatz $t' = \tfrac{t}{c \tau + d}$ and $\zeta' = e^{2 \pi i x t} \zeta$ for some unknown $x$. The change $t' \mapsto t' + 1$ corresponds to $t \mapsto t + c\tau + d$ and therefore if (\eqref{eq:z2ineven}) is to be satisfied $\zeta \mapsto e^{2 \pi i c \alpha } \zeta$. This in turn implies that $\zeta' \mapsto e^{2 \pi i \left[ x (c \tau + d) + c  \alpha \right]} \zeta'$. We want $\zeta' \mapsto \zeta'$, but for this to be true there must exist $n \in \Z$ such that 
\begin{subequations} \label{eq:mandnexist}
\begin{equation}
x (c \tau + d) + c  \alpha = n. 
\label{eq:nexists}
\end{equation}
Similarly, the map $t' \mapsto t' + \tau'$ corresponds to the change $t \mapsto t + a \tau + b$ which in turn induces $\zeta \mapsto e^{2 \pi i a \alpha }\zeta$ and therefore $\zeta' \mapsto e^{2 \pi i \left[ x (a \tau + b) + a  \alpha \right]}\zeta'$. This in turn should be equal to $e^{2 \pi i \alpha'} \zeta' = e^{2 \pi i \tfrac{\alpha}{c\tau + d}} \zeta'$. We obtain that there should exist an integer $m$ such that
\begin{equation}
x (a \tau + b) + a \alpha = - m + \frac{\alpha}{ c \tau + d}. 
\label{eq:mexists}
\end{equation}
\end{subequations}
Multiplying \eqref{eq:nexists} by $(a\tau + b)$, \eqref{eq:mexists} by $(c \tau + d)$ and substracting we obtain that the two integer numbers  $m$ and $n$ should satisfy:
\begin{equation}
-\alpha = c\alpha (a \tau + b) - a \alpha (c \tau + d) = n (a\tau + b) + m (c \tau + d) - \alpha
\label{eq:mnequation}
\end{equation}
Comparing the linear and constant terms in $\tau$ we obtain the system:
\[
\begin{pmatrix}n & m \end{pmatrix} \begin{pmatrix} a & b \\ c & d \end{pmatrix}= \begin{pmatrix} 0 & 0 \end{pmatrix}\]
therefore $m = n = 0$ is the unique solution and from \eqref{eq:nexists} we read $x = - \tfrac{c \alpha}{c \tau + d}$. The rest of the statement is a straightforward verification.
\end{proof}
\label{no:jacobisuper}
\end{nolabel}
\begin{nolabel}\label{no:newfine}
The orbifold quotient $\cM^{0} := \mathbb{H} \times \mathbb{C} \git SL(2, \mathbb{Z}) \ltimes \mathbb{Z}^2$ is a fine moduli space for even families of elliptic supercurves, that is to say, any family $Y \rightarrow T$ of elliptic supercurves over a purely even base $T$ such that the module of odd coordinates forms a degree $0$ line bundle over the reduced curve $Y_{\mathrm{rd}}$ is the pullback of the universal family $\cE^{0} := E^{0} \git SL(2, \mathbb{Z}) \ltimes \mathbb{Z}^2$ by a map $T \rightarrow \cM^{0}$.

On the other hand $\cM^{0}$ is itself a universal family of elliptic curves as we saw in \ref{no:usualuniversal}. We may view this universal curve (or rather its Jacobian) as parametrizing pairs consisting of an elliptic curve and a degree zero line bundle over it. Any elliptic supercurve over an algebraically closed field splits, namely it consists of a usual elliptic curve $X_{\mathrm{rd}}$ together with a line bundle over it (defining the odd coordinates of $X$). Therefore purely even families of elliptic supercurves are the same as families of elliptic curves equipped with line bundle.

Note that while the orbit space $\mathbb{H} \times \mathbb{C} / SL(2, \mathbb{Z}) \ltimes \mathbb{Z}^2$ projects to $\mathbb{H} / SL(2, \mathbb{Z})$, the fiber over $\tau$ is not the elliptic curve $E_\tau$. Indeed the fiber is $E_\tau / \Aut(E_\tau)$ which is generically a $\mathbb{P}^1$. This is another instance where we should instead take the categorical quotient.
\end{nolabel}
\begin{nolabel}
The base $S = S^{2|0}$ of our family $E^{0}$ admits a compactification $\overline{S}$ which corresponds to adding the cuspidal curve at $q \rightarrow 0$, $y \rightarrow 0$. We wish to construct sections of $\cA$ using \eqref{eq:x.2.3} that extend to this cuspidal curve.
\begin{lem*}
Let $\mathbf{z}=(z,\xi), \bx = (x,\theta) \in \mathbb{G}_m^{1|1}$. Let $a \in V$ be of conformal weight $\Delta$ and charge $c$. Let $k,l \geq 0$ and $K \in \left\{ 0,1 \right\}$. Consider the section 
\[
a(\mathbf{x}) = a(x,\theta) = \frac{\mathbf{x}^{k|K}a}{(x - z)^{l}} \in H^0 (X \setminus \mathbf{z}, \cA).
\]
The sum \eqref{eq:x.2.3} is a well defined section in the formal neighborhood of the cuspidal curve $(q,y)=(0,0)$ whenever
\begin{enumerate}
\item $\Delta + l > K + k > \Delta$,
\item $K + k = \Delta$ and $K > 1 + c$ for any value of $l$, or
\item $K + k = \Delta + l$ and $K < 1 + c$. 
\end{enumerate}
In particular, if $l \geq 2$ the sections
\[ \frac{x^{\Delta + 1}a}{(x - z)^l} \quad \text{and} \quad \frac{x^{\Delta} \theta a}{(x-z)^l}, \]
give rise to convergent sums in \eqref{eq:x.2.3}.
\end{lem*}
\label{no:x.2.d}
\end{nolabel}
\begin{proof}
The sum \eqref{eq:x.2.3} is in this case 
\begin{equation} 
\sum_{n \in \mathbb{Z}} \frac{q^{n K + n k - n \Delta} y^{n K -n - n c} x^k \theta^K a }{(q^n x-z)^l}
\label{eq:x.2.d.1}
\end{equation}
The numerator has only positive powers of $q$ as long as $n (K + k) > n \Delta$. When $n > 0$ the denominator is regular at $q \rightarrow 0$ since $z \in \mathbb{C}^*$. If follows that the terms with $n > 0$ vanish at $q = 0$ as long as $k + K > \Delta$. When $n < 0$ the denominator has a zero of order $nl$ at $q = 0$, it follows that the quotient has only positive powers of $q$ as long as $k + K - l < \Delta$. It follows that if the inequality $\Delta + l > k + K > \Delta$ is satisfied, then the sum \eqref{eq:x.2.d.1} is well defined as a series in $y$ at the limit $q \rightarrow 0$. If $l \geq 2$ we have always a solution with $K = 0$, $k = \Delta + 1$ and another with $K = 1$, $k = \Delta$. 

In the case $K + k = \Delta$, the limit at $q \rightarrow 0$ is given by
\[ (-z)^{-l} x^k \theta^K \sum_{n \geq 0}  y^{n (K - 1 - c)}a, \]
which is well defined as long as $K >  1 + c$. 
Finally the case $K + k = \Delta  + l$, the limit is given by
\[ x^{k-l} \theta^K \sum_{n \geq 0} y^{-n (K - 1 - c)}a, \]
which is well defined as long as $K < 1 + c$. 
\end{proof}
\begin{nolabel}\label{no:odd-moduli}
A supercurve $X$ in the fiber of $E^0$ at $y=1$ consists simply of an elliptic curve $X_\mathrm{rd}$ and the trivial line bundle $\cO_{X_\mathrm{rd}}$ defining the odd coordinate. We see that as a $\cO_{X_{\mathrm{rd}}}$-module, the structure sheaf $\cO_X$ is free of rank $1|1$. Since the tangent bundle $\cT_X$ to the supercurve $X$ is free of rank $1|1$ we obtain that $\cT_X$ as a $\cO_{X_\mathrm{rd}}$-module is free of rank $2|2$. It follows that $H^1(X, \cT_X) = \mathbb{C}^{2|2}$ therefore we expect to have $2$ even moduli and $2$ odd moduli of deformations of the supercurve $X$. The two even moduli correspond to $\tau$ and $\alpha$ as constructed in \ref{no:x.2.c}.
\end{nolabel}

\begin{nolabel}
Now we will construct a \emph{versal} family of elliptic supercurves over a $2|2$ dimensional base $S=S^{2|2}:= \mathbb{H} \times \mathbb{A}^{1|2}$. The reduced part of $S^{2|2}$ coincides with the base $S^{2|0}$ of our even families above. We take coordinates $\tau, \alpha | \varepsilon, \delta$ of $S^{2|2}$, where $\tau, \alpha$ are the old coordinates of $S^{2|0}$, and $\varepsilon, \delta$ are odd. We continue to denote $q = e^{2\pi i \tau}$ and $y=e^{2\pi i \alpha }$. Let $X = \mathbb{A}^{1|1}\times S$ be the trivial family over $S$ with global coordinates $(t,\zeta)$ be global coordinates of $\mathbb{A}^{1|1}$. We have an action of $\mathbb{Z}^2$ on $\mathbb{A}^{1|1} \times S$ generated by
\[
\begin{aligned}
A: (t,\zeta; \tau, \alpha, \varepsilon, \delta) &\mapsto \bigl( t + 1, \zeta; \tau,\alpha,\varepsilon,\delta \bigr), \\ B:(t,\zeta; \tau, \alpha, \varepsilon, \delta) &\mapsto \Bigl( t + \tau + \frac{1}{2 \pi i} \varepsilon \zeta, e^{2 \pi i \alpha} \zeta + \delta; \tau, \alpha, \varepsilon,\delta \Bigr).
\end{aligned}
\]
The quotient $E = \mathbb{A}^{1|1} \times S/\mathbb{Z}^2$ is a family of elliptic supercurves. Upon restriction to $S^{2|0}$ we recover the quotient of \ref{no:x.2.f}, which is the family $E^{0}$ of \ref{no:x.2.c}. Since the generators $A, B$ defining the $\mathbb{Z}^2$ action are not merely translations as in the classical case, the explicit form of the general element of $\mathbb{Z}^2$ is rather complicated.

As in previous cases, we may use the exponential map of \ref{no:supergroups} to pass to an equivalent description of $E$ as a quotient of $\mathbb{G}_m^{1|1} \times S/\mathbb{Z}$. In this case the action of $\mathbb{Z}$ is generated by
\begin{equation} (x,\theta; q, y, \varepsilon,\delta) \mapsto \big( q (x +  \varepsilon\theta ), q (y - \varepsilon \delta)\theta +q  \delta x ; q, y, \varepsilon, \delta \bigr). 
\label{eq:x.3.generalzaction}
\end{equation}
\label{no:ellipticfull}
\end{nolabel}
\begin{nolabel}
The group $GL(1|1)$ of automorphisms of $k^{1|1}$ is naturally a Lie supergroup with Lie algebra $\gl(1|1)$, and is a subgroup of the group $Aff^{1|1}$ of affine transformations of $\mathbb{A}^{1|1}$. The sub superscheme $\mathbb{G}_m^{1|1} \subset \mathbb{A}^{1|1}$ is stable under the action of $GL(1|1)$. We wish to describe this action in terms of the basis vectors
\begin{equation}  J_0 = 
\begin{pmatrix}
0 & 0 \\ 0 & -1 
\end{pmatrix}, \quad Q_0 = 
\begin{pmatrix}
0 & -1 \\ 0 & 0 
\end{pmatrix}, \quad
H_0 = \begin{pmatrix} 0 & 0 \\ 1 & 0 \end{pmatrix}, \quad L_0 = 
\begin{pmatrix} -1 & 0 \\ 0 & -1 \end{pmatrix}
\label{eq:gl11}
\end{equation}
of $\gl(1|1)$, which correspond to the vector fields
\[
J_0 \mapsto -\theta \partial_\theta, \qquad Q_0 \mapsto - x \partial_\theta, \qquad H_0 \mapsto \theta \partial_x, \qquad L_0 \mapsto - x\partial_x - \theta \partial_\theta
\]
on $\mathbb{G}_m^{1|1}$ (see \ref{eq:topological-derivations}). The $\Z$-action generated by the transformation \eqref{eq:x.3.generalzaction}, which we henceforth denote $\bq$, can be expressed as
\begin{equation} \bq \binom{x}{\theta} =  
(1 +  \varepsilon \theta \partial_x) 
y^{\theta \partial_\theta} q^{x \partial_x + \theta\partial_\theta} 
(1+ \delta x \partial_\theta)  \binom{x}{\theta},
\label{eq:xbqgl11}
\end{equation}
from which we deduce
\begin{equation} \label{eq:x.6.12}
\bq = 
\exp \bigl( \varepsilon H_0 \bigr)\exp \bigl( - 2 \pi i \alpha J_0 \bigr) \exp \bigl( - 2 \pi i \tau L_0 \bigr) \exp \bigl( - \delta Q_0 \bigr).
\end{equation}
This change of coordinates has Berezinian $\mathrm{Ber} \, q = y^{-1}$. Note that when $\varepsilon = \delta = 0$ we recover the diagonal action of \eqref{eq:x.2.0}. 
\label{no:x.3.1}
\end{nolabel}
\begin{rem}
The general action in \eqref{eq:x.3.generalzaction} does not preserve the SUSY structure on $\mathbb{G}_m^{1|1}$.  In fact the SUSY structure $D_m = x \partial_{\theta} + \theta \partial_x$ of \ref{no:supergroups} is mapped under the change of coordinates \eqref{eq:x.3.generalzaction} to 
\[ \left( y x - \varepsilon \theta - y^{-1} \delta \theta \right) \partial_\theta + \left( y^{-1} \theta - \varepsilon x - y^{-1}\delta \theta \right) \partial_x. \]
We see immediately that in order to preserve $D_m$ we need $y = \pm 1$ and $\varepsilon = \mp \delta$. The case $y = 1$  gives rise to the odd family of SUSY elliptic curves of \cite[7.4]{manin3}.
\label{no:x.4.0}
\end{rem}
\begin{nolabel}
Now consider the fibre of $E/S$ over a point in $S = S^{2|2}$ with $y \neq 1$. Consider the change of coordinates
\[
x'= x + \frac{\varepsilon \theta}{1 -y}, \qquad \theta' = \theta - \frac{\delta x}{1 -y}.
\]
These new coordinates induce an isomorphism between the elliptic supercurve over $q, y, \varepsilon, \delta$ and the elliptic supercurve with
\[
q' = q \left(1 + \frac{\varepsilon \delta}{1 - y}\right), \qquad y' = y, \qquad \varepsilon' = \delta' = 0.
\]
We see therefore that the restriction of $E$ to the locus in $S$ of $y \neq 1$ is trivial in the odd directions $\varepsilon, \delta$ (note though that to do this trivialization we are forced to add nilpotents to the even moduli $\tau$). This shows that the supercurve $X$ with parameter values $y=1,\: \varepsilon= \delta = 0$, as considered in \ref{no:odd-moduli}, represents a singular point in the moduli space of elliptic supercurves\footnote{We owe Nathan Berkovits for an explanation on this point.}. We have a family with two even moduli $q, y$ (equivalently $\tau$, $\alpha$) of deformations of this curve (which is $E^0$) and another family with one even moduli $q$ and two odd moduli $\varepsilon, \delta$ which is the restriction of $E$ to $y=1$. This accounts for the $2|2$ moduli computed in \ref{no:odd-moduli}. These two families fit together in the family $E$.
\label{no:x.4.1}
\end{nolabel}
\begin{nolabel}
Let $V$ be a quasi-conformal $N_W=1$ SUSY vertex algebra and let $\cA$ be the chiral algebra over $\mathbb{A}^{1|1}$ corresponding to $V$. There is an action of the group $Aff^{1|1}$ of affine transformations of $\mathbb{A}^{1|1}$ on $V$, with the restriction to its subgroup $GL(1|1)$ defined by the operators \eqref{eq:gl11}. In the remainder of this section we examine the action of the element $\bq \in GL(1|1)$ of \eqref{eq:x.6.12} on $V$ and on sections of $\cA$.

Under the $GL(1|1)$ action $V$ decomposes into irreducible modules. There are two families of irreducible representations of $GL(1|1)$  in which $L_0$ and $J_0$ act diagonally: one is $1$-dimensional (either even or odd) and the other is $1|1$-dimensional. Modules of the first type are generated by a vector $a$ with $L_0 a = Q_0 a = H_0 a = 0$ and $J_0$ acting by the charge $c_a$ of $a$ which is a scalar. Any SUSY vertex algebra contains at least one such vector, namely the vacuum vector $\vac$ generating the trivial representation of $GL(1|1)$. Then $\bq \in GL(1|1)$ acts by
\[ \bq \cdot a = y^{-c_a} a. \]
Now consider a $1|1$ dimensional representation of $GL(1|1)$. Select nonzero eigenvector $a$ of $L_0$ and $J_0$ such that $H_0 a = 0$ (such $a$ exists because $H_0$ permutes eigenvalues and is nilpotent). The module is spanned by $a$ and its \emph{superpartner} $Q_0 a \neq 0$. Let the $L_0$, $J_0$ eigenvalues (i.e., conformal weight and charge) of $a$ be $\Delta$ and $c$. Since $[J_0, Q_0] =  Q_0$ we see that $Q_0 a$ has charge $c + 1$.

In the following discussion we adhere to the rule of signs for supercommutative matrix multiplication discussed in \ref{no:subtlety}. With this in mind, the even elements $\delta Q_0$, $\varepsilon H_0$ $L_0$ and $J_0$ of $\gl(1|1)$ act by the matrices
\[ 
\delta Q_0 = 
\begin{pmatrix}
0 & 0 \\ \mp \delta & 0 
\end{pmatrix}, \qquad
\varepsilon H_0 = 
\begin{pmatrix}
0 & \pm \varepsilon \Delta \\ 0 & 0 
\end{pmatrix}, \qquad
L_0 = 
\begin{pmatrix}
\Delta & 0 \\ 0 & \Delta 
\end{pmatrix}, \qquad
J_0 = 
\begin{pmatrix}
c & 0 \\ 0 & c+1
\end{pmatrix}
\]
on the $GL(1|1)$-representation with ordered basis $a, Q_0 a$ (the $\pm$ sign here is interpreted as $+$ if the vector $a$ is even, and $-$ if it is odd). It follows that the element $\bq$ acts in this space as the matrix 
\begin{equation}
q^{-\Delta}
\begin{pmatrix}
1 & \pm \Delta \varepsilon \\ 
0 & 1 
\end{pmatrix}
\begin{pmatrix}
y^{-c} & 0 \\ 0 & y^{-(c+1)}
\end{pmatrix}
\begin{pmatrix}
1 & 0  \\ \pm \delta & 1
\end{pmatrix}
= 
q^{-\Delta} y^{-(c + 1)} \begin{pmatrix}  y + \Delta \varepsilon \delta & \pm \Delta \varepsilon \\  \pm  \delta & 1 
\end{pmatrix}
\label{eq:6.16.1-matrix}
\end{equation}
We restrict $\cA$ to $\mathbb{G}_m^{1|1}$ and identify its sections with functions $a(x,\theta): \mathbb{G}_m^{1|1} \rightarrow V$. We now describe the action of $\mathbf{q}$ on these sections. Recall that the passage from $V$ to $\cA$ involves tensoring with the Berezinian $\omega_{X/S}$ of $X/S$. Hence the Berezinian (which is $y^{-1}$) of $\bq$ enters. For a pair of superpartners $\left\{ a, Q_0 a \right\}$ the action of $\bq$ as in \eqref{eq:x.6.12} on sections of $\cA$ is given by left multiplication by the matrix
\begin{equation}\label{eq:matrix-final}
q^{-\Delta} y^{-(c+2)} 
\begin{pmatrix}
y + \Delta \varepsilon \delta & \pm  \varepsilon \\ 
\pm \delta &  1 
\end{pmatrix}
\end{equation}
extended by the left $S$-linear action of $\bq$ on the global coordinates $(x, \theta)$. That is to say we consider the left $\cO_S$-module $\cO_X$, so that an $S$-point is described by a \emph{row of left coordinates} $(\alpha, \beta)$ which we interpret as the point $\alpha x + \beta \theta$. Notice that $\bq$ acts therefore on the row vector $(\alpha,\beta)$ by right multiplication by the matrix
\begin{equation}\label{eq:matrix-coordinate-final}
q \begin{pmatrix}
1 & \varepsilon \\ \delta & y - \varepsilon \delta
\end{pmatrix}.
\end{equation}
For each quadruple $(q, y, \varepsilon, \delta) \in S^{2|2}$ we have the corresponding element $\bq \in GL(1|1)$ given by \eqref{eq:xbqgl11} and the corresponding action of $\mathbb{Z}$ on sections of $\cA$ generated by powers of $\bq$. 
\label{no:generalelliptic}
\end{nolabel}
\begin{nolabel}
One case of the above, which we will require later on, is especially nice. This is the case of a pair of superpartners with $a$ odd of conformal weight $\Delta=1$ and charge $c = -1$. We see that \eqref{eq:matrix-coordinate-final} is the inverse of \eqref{eq:matrix-final}. A general \emph{linear} section of $\cA$ with image in this $1|1$ dimensional subspace can be expressed as
\[ 
\begin{pmatrix}
a & Q_0 \, a
\end{pmatrix}
\begin{pmatrix}
A & B \\ C & D
\end{pmatrix}
\begin{pmatrix}
x \\ \theta
\end{pmatrix}, \]
and the action of $\bq$ on this section is given simply by the conjugation
\begin{equation}\label{eq:action-linear-vectors}
\begin{pmatrix}
A & B \\ C & D 
\end{pmatrix}\mapsto 
\begin{pmatrix}
1 & \varepsilon \\ \delta & y - \varepsilon \delta
\end{pmatrix}^{-1}
\begin{pmatrix}
A & B \\ C & D
\end{pmatrix}
\begin{pmatrix}
1 & \varepsilon \\ \delta & y - \varepsilon \delta
\end{pmatrix}.
\end{equation}
\label{no:linear-vectors}
\end{nolabel}

\section{Conformal Blocks} \label{sec:conformal-blocks}
In this section we show that the trace functions \eqref{eq:super-traces} annihilate regular meromorphic sections of $\cA$ with poles at a single point. That is, they lie in the space \eqref{eq:conformal-definition} of conformal blocks on elliptic supercurves. We begin by recalling Zhu's construction in \ref{no:zhu0}, here the identification of the chiral algebra on $\mathbb{A}^1$ and on $\mathbb{G}_m$ via the exponential map plays an integral part. We then apply the same formalism to the supersymmetric case in \ref{no:zhu1}, the super exponential map \ref{eq:exponential} plays a corresponding role in our setup.
\begin{nolabel}\label{no:zhu0}
Let $V$ be a quasi-conformal vertex algebra. Consider the change of coordinates $x - 1 = \rho(t) = e^{2 \pi i t}-1$. We have $\rho_t(s) = e^{2 \pi i t} \rho(s)$ hence $\rho_t = \rho \cdot e^{-2 \pi i t L_0}$. Denoting by the same letter $\rho$ its action on $V$ we have from Huang's formula \eqref{eq:huangs}
\begin{equation}\label{eq:huangs2}
Y(a, z) = \rho Y([\rho \cdot e^{-2\pi i z L_0}]^{-1} a, \rho(z)) \rho^{-1}.
\end{equation}
Thus was Zhu led to introduce the new vertex algebra structure
\[
Y[a, t] = Y(e^{2\pi i t L_0} a, \rho(t)) = x^{\Delta_a} Y(a, x-1) = \rho^{-1} Y(\rho \cdot a, t) \rho
\]
on $V$, isomorphic to the old one via the linear isomorphism $\rho \in \Aut(V)$. This isomorphism is none other than the isomorphism $\exp^* \cA_{\mathbb{G}_m} \simeq \cA_{\mathbb{A}^1}$ of \ref{no:x.1b}. The identification of sections matches the constant section $a(t) = a$ of $\cA_{\mathbb{A}^1}$ with the section $a(x) = x^{\Delta_a-1} a$ of $\cA_{\mathbb{G}_m}$ (the factor $x^{-1}$ comes from the Jacobian of the change of coordinates). In this way we obtain a map $V \rightarrow \Gamma (\mathbb{G}_m, \cA_{\mathbb{G}_m})$. The coordinate $t$ on $\mathbb{A}^1$ identifies $V$ with the fiber of $\cA_{\mathbb{A}^1}$ at any point, in particular we have identifications $V = (\cA_{\mathbb{A}^1})_{t=0} = (\cA_{\mathbb{G}_m})_{x=1}$. Composing these yields the map
\begin{equation}\label{eq:fiber-trivial-1}
\iota: (\cA^l_{\mathbb{G}_m})_{1} \rightarrow \Gamma(\mathbb{G}_m, \cA^l_{\mathbb{G}_m}), \qquad a \mapsto x^{\Delta_a} a. 
\end{equation}
\end{nolabel}
\begin{nolabel}
\label{no:zhu0b}
Let $M$ be a positive energy $V$-module and define the linear operator $\Theta_M \in V^*$ by 
\begin{equation}\label{eq:varphi-non-super-def}
\Theta_M(a) = \tr_M x^{\Delta_a} Y(a,x) q^{L_0}.
\end{equation}
As defined it is a formal power series in $q$, however Zhu proved in \cite[4.4.1]{zhu} that if $V$ satisfies the finiteness conditions in \ref{no:c2-cofiniteness} then this series converges absolutely and uniformly to a holomorphic function. Note also that this series does not depend on $x$. Indeed it follows from \eqref{eq:quasi-conf-def} that $[L_0, a_n] = -n a_n$ and so in the expansion
\[
Y(a,x) = \sum_{n \in \mathbb{Z}} x^{-n - \Delta_a} a_n,
\]
only the term $a_0$ contributes to the trace and $\Theta_M(a) = \tr_M a_0 q^{L_0}$. Using the map $\iota$ of the previous section we define the corresponding linear functional on the fiber of  $\cA_{\mathbb{G}_m}$ at the point $x_0$ given by $x = 1$ as follows. The module $M$ gives rise to a $\cA$-module $\cM$ supported at $x_0$. For $a \in (\cA_{\mathbb{G}_m})_{x_0}$ we have the corresponding section $\iota(a)$ and therefore the endomorphism of the fiber $\cM_{x_0}$ given by $\cY^\cM_{x_0}(\iota a)$. We also have the action of $q^{1} \in \mathbb{G}_m$ on $\cM_{x_0}$ by $q^{-L_0}$. Denote this action by $q_M$. We define $\varphi_M \in (\cA_{\mathbb{G}_m})_{x_0}^*$ by
\begin{equation}\label{eq:7.1.1.prior}
\varphi_M(a) = \tr_{\cM_{x_0}} \cY^\cM(\iota a) q^{-1}_M.
\end{equation}
Upon trivializing using the local coordinate, $\varphi_M$ is seen to coincide with $\Theta_M$.
\end{nolabel}
\begin{rem}
Note that the negative power $q^{-1}_M$ appears in \eqref{eq:7.1.1.prior} because of the conventions in \ref{no:x.1} which in turn come from the sign in the identification  $L_0 = - t \partial_t$ of \cite{frenkelzvi}. 
\label{rem:nuevo-remark}
\end{rem}
\begin{nolabel}\label{no:zhu0c}
Since we evaluate at the marked point $x=1$ in the end, observe that we might equally well have omitted the factor $x^{\Delta_a}$ from the definition of $\Theta_M$. This is equivalent to considering global sections of $\cA$ over $\mathbb{G}_m \subset \mathbb{A}^1$ which are \emph{constant} functions with respect to the trivialization induced by the affine coordinate of $\mathbb{A}^1$, rather than using the map $\iota$ to produce global sections of $\cA$. The cost of this modification is to make the trace dependent on $x$. Since all of our $1$-point functions will be evaluated at the marked point $x=1$ (or $t=0$) these two approaches are essentially equivalent. 
\end{nolabel}

\begin{nolabel}\label{no:zhu1}
We pass to the supersymmetric setting now. Let $\rho : \mathbb{A}^{1|1} \mapsto \mathbb{G}_m^{1|1}$ be the map (c.f., \eqref{eq:exponential})
\[
\bt = (t, \zeta) \mapsto (x-1, \theta) = \rho(\bt) :=  \left( e^{2 \pi i t} - 1,e^{2 \pi i t} \zeta \right).
\]
\begin{lem*}
We have
\[
\rho_\bt = \rho \cdot e^{- \zeta Q_0} e^{-2 \pi i t L_0}.
\]
\end{lem*}
\begin{proof}
For another set of parameters $\bz = (z,\xi)$ we have
\[
\rho_\bt(\bz) = \rho(\bz + \bt) - \rho(\bt) = \Bigl( e^{2 \pi i t} \left( e^{2 \pi iz}-1 \right), e^{2 \pi i t}\zeta \left( e^{2 \pi i z} -1 \right) + e^{2\pi i t} e^{2 \pi i z} \xi \Bigr) = \sigma \Bigl( \psi \bigl(\rho(\bz) \bigr) \Bigr),
\]
where $\sigma$ and $\psi$ are the automorphisms
\[
\sigma(z,\xi) = (e^{2 \pi i t} z, e^{2 \pi i t} \xi), \qquad \psi(z, \xi) = (z, \xi + \zeta z).
\]
The lemma follows using the basis of the topological algebra \eqref{eq:topological-derivations}. 
\end{proof}
Let $V$ be a quasi-conformal $N_W=1$  SUSY vertex algebra and let $\cA_{\mathbb{A}^{1|1}}$ (resp. $\cA$) be the chiral algebra on $\mathbb{A}^{1|1}$ (resp. $\mathbb{G}_m^{1|1}$) corresponding to $V$. The point $\bt_0 \in \mathbb{A}^{1|1}$ with coordinates $\mathbf{t} = (0, 0)$ is mapped by the exponential to $\bx_0 \in \mathbb{G}_m^{1|1}$ with coordinates $\bx = (1,0)$. As in \ref{no:zhu0} we obtain a map from the fiber $\cA_{\bx_0}$ to global sections, explicitly 
\begin{equation}\label{eq:fiber-trivial-2}
\iota: \cA_{\bx_0} \rightarrow \Gamma( \mathbb{G}_m^{1|1}, \cA), \qquad a \mapsto x^{\Delta_a} a + x^{\Delta_a - 1} \theta Q_0 a.
\end{equation}
\end{nolabel}
\begin{nolabel}\label{no:primero}
Let $M$ be a positive energy $V$-module. We have actions of the annihilation algebra of the topological algebra \eqref{eq:topological-commutators} on $V$ and $M$ that exponentiate to actions of $\Aut \cO^{1|1}$. Let $\bq \in GL(1|1)$ be a generic element written as \eqref{eq:x.6.12} and denote by $\bq_V$ and $\bq_M$ the corresponding automorphisms of $V$ and $M$. Define $\Theta_M$ as the formal series in $\bq$ with values in $V^*$ given by the formula
\begin{equation}\label{eq:varphi_def_1}
\Theta_M(a) =  \str_M  Y\Bigl( x^{\Delta_a} a + x^{\Delta_a -1} \theta Q_0 a,\bx \Bigr) \bq^{-1}_M.
\end{equation}
\begin{lem*}
$\Theta_M$ is independent of $\bx$. 
\end{lem*}
\begin{proof}
By our assumptions we can decompose $V$ in irreducible representations of $GL(1|1)$. We expand the superfield of any such vector as
\[
Y(a,\bx) = \sum_{n \in \mathbb{Z}} x^{-n - \Delta_a}\,  a_{n} + \sum_{n \in \mathbb{Z}} \theta x^{-n - \Delta_a - 1 } (Q_{-1} a)_n.
\]
We first note that the operators $a_n \in \End (V)$ are well defined since $Q_{-1}a$ has conformal weight $\Delta_a + 1$. We may moreover assume that $H_0 a = 0$. The superpartner (if it exists) of $a$ is $Q_0 a$. Since $Q_0 a$ has conformal weight $\Delta_a$ it follows that the field in the RHS of \eqref{eq:varphi_def_1} has an expansion as 
\begin{equation}\label{eq:varphi_rhs_2}
\sum_{n \in \mathbb{Z}} x^{- n} a_n + \sum_{n \in \mathbb{Z}} \theta x^{-n-1} \left( (Q_{-1}a)_n + (Q_0 a)_{n} \right).
\end{equation}

It follows from \eqref{eq:quasi-conformal-susy} with $n = 0$ that
\begin{equation}\label{eq:quasi-caso-part}
\begin{aligned}
{[}L_0, Y(a, \bx)] &= (x \partial_x + \theta \partial_\theta + \Delta_a) Y(a, \bx),\\ 
[J_0, Y(a, \bx)] &= (\theta \partial_\theta  + c_a)Y(a,\bx),\\
[Q_0, Y(a,\bx) ] &= x \partial_\theta Y(a, \bx) + Y( Q_0 a, \bx), \\
[H_0, Y(a, \bx)] &= -\theta \partial_x Y(b, \bx).
\end{aligned}
\end{equation}
From the first of these equations we see that $[L_0, a_n] = - n a_n$, hence only the terms $a_n$, $(Q_{-1}a)_{n}$ and $(Q_{0}a)_n$ with $n = 0$ contribute to the trace. Applying the third equation to the vacuum we obtain $(Q_{-1}a)_0 + (Q_{0}a)_0 = 0$. Hence $\Theta_M(a) = \str_M a_0 \bq^{-1}_M$ is independent of $\bx$.  
\end{proof}
Convergence of the formal series \ref{eq:varphi_def_1} is proved in the appendix under a finiteness condition on $V$ which we call \emph{charge cofiniteness} (see Def. \ref{defn:c2-cofinite-super}) and which turns out to be equivalent to $C_2$ cofiniteness. Alternatively one may ignore convergence issues and work in the formal neighborhood of the cuspidal curve as in \ref{no:x.2.d}.

In the same way as in \ref{no:zhu0b} we obtain a linear functional on the fiber $\cA_{\bx_0}$ over the point $\bx_0 = (1, 0) \in \mathbb{G}_m^{1|1}$, of $\cA$. In more detail: the module $M$ gives rise to a $\cA$-module supported at $\bx_0$. For an element $a$ in the fiber $\cA_{\bx_0}$ we have the corresponding section $\iota a \in \Gamma(\mathbb{G}_m^{1|1}, \cA)$ and its action $\cY^\cM_{\bx_0} (\iota a) \in \End (\cM_{\bx_0})$. We define $\varphi_M \in \cA_{\bx_0}^*$ by
\[
\varphi_M (a) = \str_{\cM_{\bx_0}} \cY^\cM_{\bx_0} (\iota a) \bq^{-1}_M.
\]
Upon trivializing using the local coordinate, $\varphi_M$ is seen to coincide with $\Theta_M$.
\end{nolabel}
\begin{nolabel}\label{no:7.1}
Similar remarks to \ref{no:zhu0c} apply in the super case. When evaluating at the point $\bx_0 = (1,0)$ we might equally well have omitted the factor $x^{L_0} \exp(\theta x^{-1} Q_0)$ from the definition of $\Theta_M$. This corresponds to not identifying the chiral algebras in $\mathbb{A}^{1|1}$ and $\mathbb{G}_m^{1|1}$ via the exponential map, at the cost of our trace becoming $\bx$-dependent. Since all of our $1$-point functions will be evaluated at $\bx_0 = (1,0) \in \mathbb{G}_m^{1|1}$ or equivalently at $\bt_0 = (0,0) \in \mathbb{A}^{1|1}$ these two approaches are equivalent. In particular we can define the linear functional $\varphi_M \in \cA_{\bx_0}^*$ by 
\begin{equation}\label{eq:7.1.1}
\varphi_M(a) = \str_{\cM_{\bx_0}} \cY^\cM_{\bx_0}(a) \bq^{-1}_M
\end{equation}
which in coordinates is simply given by $a \mapsto \str_M Y(a,\bx) \bq^{-1}_M|_{\bx = (1,0)}$. 
\end{nolabel}
\begin{nolabel}\label{no:relative}
All we have said carries over verbatim to the relative situations of trivial $\mathbb{G}_m^{1|1}$ and $\mathbb{A}^{1|1}$ fibrations over $S = \mathbb{H} \times \mathbb{C}^{1|2}$ as in \ref{no:ellipticfull}. That is, we will let $X = \mathbb{G}_m^{1|1} \times S$ and let $\cA = \cA_X$ be the chiral algebra associated to $V$ on the family $X/S$.  
\end{nolabel}
\begin{nolabel}
Consider the $\mathbb{Z}$-orbit $\{ \bq^n \cdot {\bx_0}\:|\: n \in \mathbb{Z}\}$ of ${\bx_0}$ in $X$, and let $U = \mathbb{G}_m^{1|1} \setminus \mathbb{Z}\cdot {\bx_0}$. We have the sheaf of Lie algebras $h(\cA)=\cA \otimes_{\cD_X} \cO_X$, and restriction to the punctured disk $D_{\bx_0}^\times = \Spec \cK_{\bx_0} \subset U$ gives an action of the Lie algebra $h(\cA)(U)$ on $\cA_{\bx_0}$.  

We have $U = \cap_N U_N$ where for $N \geq 0$ the open sub-superspace $U_N \subset \mathbb{G}_m^{1|1}$ is defined to be
\[
\mathbb{G}_m^{1|1} \setminus \left\{ \bq^{-N} {\bx_0}, \bq^{-N + 1} {\bx_0}, \dots, \bq^{N-1} {\bx_0}, \bq^N {\bx_0} \right\}.
\]

Recall our families $X$, $E$, and the quotient $\pi: X \twoheadrightarrow E = X/\mathbb{Z}$ over $S = S^{2|2}$ from \ref{no:ellipticfull}. Let $\cA_E$ be the corresponding chiral algebra on $E$, and let $a=a(x,\theta)$ be (as in \ref{no:x.2.d}) a section of $\cA$ on $U_0$ such that the sum $\pi_* a := \sum_{n \in \mathbb{Z}} \bq^n \cdot a$ is well defined in the formal neighborhood of the elliptic supercurve $q=0$. It is a $\mathbb{Z}$-equivariant section of $\cA$ defined on $U$, or by \ref{no:universal property} can be thought of as a section of $\cA_E$ defined on $E \setminus \pi({\bx_0})$. Let $h(\pi_*a)$ denote its image under the quotient map
\[
\cA_E \otimes \cT_E \rightarrow \cA_E \twoheadrightarrow h(\cA_E) \rightarrow 0.
\]
Alternatively, we can think of $h(\pi_*a)$ as the image of $\pi_* a$ in $h(\cA)(U)$. We identify fibers of $\cA_{E}$ and $\cA$ at $\bx_0 = (1,0)$ with $V$, and we use this identification to think of $\varphi_M$ as a linear functional on the fiber of $\cA_E$ at $\bx_0$. With this notation we can now state the following
\begin{prop}\label{trace.is.cb} The functional $\varphi_M$ is well-defined and is a conformal block of $\cA_E$. That is, it satisfies \eqref{eq:conformal-definition}. 
\begin{equation}
\varphi_M( h(\pi_* a) \cdot b) = 0, \qquad \forall b \in \cA_{\bx_0}. 
\label{eq:7.prop.1}
\end{equation}
\end{prop}
\label{no:prop1}
\end{nolabel}
\begin{proof}
From associativity of the chiral multiplication $\mu^M$ applied to $h(\cA) \otimes \cA \otimes \cM$ we obtain that for any $m \in \cM_{\bx_0}$ we have
\[
\mu^M(h(\pi_* a) b, m) = h(\pi_* a) \cdot \mu^M(b, m) - (-1)^{ab} \mu^M(b,  h(\pi_* a) \cdot m).
\]
Taking the supertrace yields 
\begin{equation} \varphi_M(h(\pi_*a) \cdot b) = \str_{\cM_{\bx_0}} h(\pi_* a) \cY^\cM_{\bx_0}(b) \bq^{-1}_M - (-1)^{ab} \str_{\cM_{\bx_0}} \cY^\cM_{\bx_0}(b) h(\pi_* a) \bq^{-1}_M. 
\label{eq:7.2b.1}
\end{equation}
where $(-1)^{ab}$ is equal to $-1$ if both $a$ and $b$ are odd sections and $1$ otherwise. 
Using the supersymmetry of the trace and the $\mathbb{Z}$-equivariance of $\cA$ (with respect to the action generated by $\bq$) we can write the first term in the RHS of \eqref{eq:7.2b.1} as 
\begin{multline} \str_{\cM_{\bx_0}} h(\pi_* a) \cY^\cM_{\bx_0}(b) \bq^{-1}_M  = (-1)^{ab} \str_{\cM_{\bx_0}} \cY^\cM_{\bx_0}(b) \bq^{-1}_M h(\pi_* a) \\
= (-1)^{ab}\str_{\cM_{\bx_0}} \cY^\cM_{\bx_0}(b) h(\bq^{-1} \cdot \pi_*a) \bq^{-1}_M = (-1)^{ab} \str_{\cM_{\bx_0}} \cY^\cM_{\bx_0}(b) h(\pi_*a) \bq^{-1}_M 
\label{eq:7.prop.2}
\end{multline}
from which we see that the RHS of \eqref{eq:7.2b.1} vanishes. 
\end{proof}

\begin{rem}
This proposition is an example where the algebro-geometric approach to vertex algebras is powerful. The same proof in local coordinates with manipulations of power series (which we present in the remainder of this section) becomes more involved as one needs to be careful in which domains we are expanding the power series.

Using the global coordinates on $\mathbb{G}_m^{1|1}$ we identify the section $a$ with a function $a(x,\theta): \mathbb{G}^{1|1}_m \rightarrow V$, and $b$ with a vector $b \in V$ and for any section $a$ we have 
\[ h(a)b = \res_{\bx-\bz}i_{|z|>|x-z|}   Y(a(x,\theta),x-z, \theta-\xi)b. \]
Here in the RHS we first view it as a power series in $V( (\bx - \bz) )( (\bz))$ by expanding in the domain $|z| > |x - z|$ and then take the residue to obtain a Laurant series in $V ( (\bz)) \simeq V \otimes \cK_\bz$. 
Denote by $\bq_V$ the action of $\bq \in GL(1|1)$ on $V$ as explained in \ref{no:generalelliptic}, and by $\bq_M$ the action on $M$. Denote by $\mathbf{x} = (x,\theta) \mapsto \bq \cdot \mathbf{x}$ the action on $\mathbb{G}_m^{1|1}$ as in \eqref{eq:x.3.generalzaction}. And recall that $\Ber \bq = y^{-1}$. The action of $n \in \mathbb{Z}$ on sections $a(x,\theta)$ is given by:
\[ a(\mathbf{x}) \mapsto \bq_V^{n} a( \bq^n \cdot \mathbf{x}) \Ber \bq^n. \]
It follows that 
\[ h(\pi_* a) \cdot  b = \sum_{n \in \mathbb{Z}} \res_{\bx-\bz} i_{|z|>|x-z|} Y\bigl( \bq_V^n a( \bq^n \cdot \mathbf{x} ), \mathbf{x} - \bz \bigr)b \Ber \bq^n,\]
and thus
\begin{equation} \varphi_M( h(\pi_* a) \cdot b) = \sum_{n \in \mathbb{Z}} \res_{\bx-\bz} i_{|z|>|x-z|} \str_M Y^M \Bigl( Y \bigl( \bq^n_V a ( \bq^n \cdot \mathbf{x}), \mathbf{x} -\bz\bigr)b, \bz \Bigr) \Ber \bq^n \bq^{-1}_M|_{\bz = \bx_0}.
\label{eq:7.3.1a}
\end{equation}
The associativity of the chiral multiplication is phrased in local coordinates as in Lemma \ref{lem:associativity-in-coords}. From this the RHS of \eqref{eq:7.3.1a} is written as
\begin{equation} 
\begin{aligned} 
\varphi_M(h(\pi_* a) b) = {} & \sum_{n \in \mathbb{Z}} \res_{\bx} i_{|x|>|z|} \str_M Y^M \Bigl( \bq_V^n a \bigl(\bq^n \cdot \mathbf{x}\bigr),\mathbf{x}  \Bigr)  Y^M \bigl(b, \bz \bigr) \Ber \bq^n \bq^{-1}_M|_{\bz = \bx_0} \\
&- (-1)^{ab} \sum_{n \in \mathbb{Z}} \res_{\bx} i_{|z|>|x|} \str_M Y^M (b,\bz) Y^M \Bigl( \bq_V^n a \bigl(\bq^n \cdot \mathbf{x}\bigr),\mathbf{x} \Bigr) \Ber \bq^n \bq^{-1}_M|_{\bz = \bx_0}
\end{aligned}
\label{eq:7.3.2b}
\end{equation}
The equivariance of $\cA$  and $\cM$ under the action of $GL(1|1) \subset Aff^{1|1}$ and in particular under the action of $\mathbb{Z}$ generated by $\bq$ is expressed in local coordinates by
\begin{equation} Y^M( c, \bx) = \bq_M Y^M (\bq^{-1}_V c, \bq \cdot \bx) \bq^{-1}_M. 
\label{eq:7.3.3}
\end{equation}
Using the supersymmetry of the trace and \eqref{eq:7.3.3} with $c = \bq^{n}a$ we express the first term on the RHS of \eqref{eq:7.3.2b} as 
\begin{multline}
(-1)^{ab} \sum_{n \in \mathbb{Z}} \res_{\bx} i_{|x|>|z|} \str_M   Y^M \bigl(b, \bz \bigr) \Ber \bq^n \bq^{-1}_M  Y^M \Bigl( \bq_V^n a \bigl(\bq^n \cdot \mathbf{x}\bigr),\mathbf{x} \Bigr)\Bigr|_{\bz=\bx_0} \\
= (-1)^{ab} \sum_{n \in \mathbb{Z}} \res_{\bx} i_{|x|>|z|} \str_M   Y^M \bigl(b, \bz \bigr) \Ber \bq^n   Y^M \Bigl( \bq_V^{n-1} a \bigl(\bq^n \cdot \mathbf{x}\bigr),\bq \cdot \mathbf{x} \Bigr) \bq^{-1}_M\bigr|_{\bz = \bx_0}
\label{eq:7.3.4}
\end{multline}
Using formula \eqref{eq:huangs} to change coordinates $\bx \mapsto \bq \cdot \bx$ and noting that the residue transforms by the Berezinian of the change of coordinates:
\[ \res_{\bq x} = \Ber \bq \res_{\bx}, \]
we write \eqref{eq:7.3.4} as 
\begin{equation}
(-1)^{ab} \sum_{n \in \mathbb{Z}} \res_{\bx} i_{|x|>|z|} \str_M   Y^M \bigl(b, \bz \bigr) \Ber \bq^{n-1}   Y^M \Bigl( \bq_V^{n-1} a \bigl(\bq^{n-1} \cdot \mathbf{x}\bigr), \mathbf{x} \Bigr) \bq^{-1}_M\bigr|_{\bz = \bx_0}
\label{eq:7.3.5}
\end{equation}
which clearly cancels the second term in the RHS in \eqref{eq:7.3.2b}.
\label{rem:bdrules}
\end{rem}

\section{The connection} \label{sec:connection}
In this section we show that the conformal blocks $\varphi_M$ obtained as (super)traces of modules satisfy explicit differential equations with respect to the moduli parameters of $\bq$. This shows that they give rise to flat sections of particular bundles with connection over the moduli spaces of elliptic (super)curves. To produce the differential equations we modify the proof of Proposition \ref{trace.is.cb} to sections of $\cA$ that are quasiperiodic in a certain sense, and analogous to Weierstrass's $\zeta$ function.
We construct these sections using renormalised versions of \ref{no:x.1.e} (and \ref{no:x.2.d} in the super case)
This argument corresponds in the physics literature to \emph{inserting the superconformal field into the correlator}. We start with the usual case of elliptic curves to re-derive Zhu's result in \cite{zhu}. We recall first some classical results on elliptic functions and modular forms referring the reader to \cite{appell,chandra,zagier-jacobi}.  
\begin{nolabel}
Let $\Gamma \subset \mathbb{C}$ be a rank $2$ lattice. The series
\[  \sum_{\stackrel{\gamma \in \Gamma}{|\gamma| > 2 R > 0}} \left( \frac{1}{(t - \gamma)^2} -  \frac{1}{\gamma^2} \right) \]
converges uniformly in the disk $\{ t : |t| \leq R \}$, therefore
\[ \wp(t) = \wp_\Gamma (t) := \frac{1}{t^2} +   \sum_{ 0 \neq \gamma \in \Gamma}\left( \frac{1}{(t - \gamma)^2} -  \frac{1}{\gamma^2} \right) \]
converges absolutely for $t \notin \Gamma$. The function so defined is Weierstrass's $\wp$-function. It is holomorphic in the whole plane except for points in $\Gamma$ where it has poles of order $2$, zero residue and principal part $1/t^2$. It is an elliptic function with periods given by elements in $\Gamma$.  

The series
\[  \sum_{\stackrel{\gamma \in \Gamma}{|\gamma| > 2 R > 0}} \left( \frac{1}{t - \gamma} + \frac{1}{\gamma} + \frac{t}{\gamma^2} \right) \]
converges uniformly in the disk $\{ t : |t| \leq R \}$, therefore
\[ \zeta(t) = \zeta_\Gamma(t) := \frac{1}{t} + \sum_{0 \neq \gamma \in \Gamma}  \left( \frac{1}{t - \gamma} + \frac{1}{\gamma} + \frac{t}{\gamma^2} \right) \]
converges absolutely for $t \notin \Gamma$. The function so defined is Weierstrass's $\zeta$-function. It is holomorphic in the whole plane except for points in $\Gamma$ where it has simple poles with residue $1$. It is \emph{not} an elliptic function since it has only one pole in a fundamental domain for the lattice $\Gamma$. Choose an isomorphism $\Gamma \simeq \mathbb{Z}^2$, that is a set $\{\omega_1, \omega_2\}$ of generators of $\Gamma$. Let $\left\{ \omega^*_i \right\} \subset \Gamma^\vee := \Hom (\Gamma, \mathbb{Z})$ be the dual basis and construct
\[ \eta \in \Hom (\Gamma, \mathbb{C}), \qquad \eta(\gamma) = \sum_{i=1}^2 \zeta\left( \frac{\omega_i}{2} \right) \omega_i^*(\gamma). \]
We summarize what we know of $\zeta(t)$ in the following
\begin{thm*}\hfill
\begin{enumerate}
\item $\zeta(t) = - \zeta(-t)$.
\item $\zeta'(t) = - \wp(t)$. 
\item $\zeta(t + \gamma) = \zeta(t) + 2 \eta(\gamma)$, $\forall \gamma \in \Gamma$. 
\item If $\tfrac{\omega_1}{\omega_2} \in \mathbb{H}$ then \[ \zeta\left( \frac{\omega_1}{2} \right) \omega_2 - \zeta\left( \frac{\omega_2}{2} \right) \omega_1 = \pi i. \]
\item Near the origin $t = 0$, $\wp(t)$ and  $\zeta(t)$ have expansions
\[ \begin{aligned}
\wp(t) &= \frac{1}{t^2} + b_1 t^2 + b_2 t^4 + \dots,\\
\zeta(t) &= \frac{1}{t} - \frac{b_1}{3} t^3 - \frac{b_2}{5} t^5 - \dots, 
\end{aligned}
\]
where by definition $b_n = (2n + 1) \sum_{0 \neq \gamma \in \Gamma} \gamma^{-2n-2}$. 
\end{enumerate}
\end{thm*}
\label{no:weierzeta}
\end{nolabel}
\begin{nolabel}
It will be useful for us to consider Hermite's modification of Weierstrass's function \cite{hermite}. It follows from (c) of \ref{no:weierzeta} that if we define
\[
\widetilde{\zeta}(t) = \zeta(t) - \frac{2 \zeta\left( \frac{\omega_1}{2} \right)}{\omega_1} t,
\]
then this function is periodic with respect to $\omega_1$: $\widetilde{\zeta}(t + \omega_1) = \widetilde{\zeta}(t) $. Therefore we may expand it in powers of $e^{2 \pi i t/\omega_1}$.  Recall the {\em cot identity}
\begin{equation}
\pi i + 2 \pi i \frac{1}{e^{2\pi i x} - 1} = \pi \mathrm{cot } \pi x =  \frac{1}{x} + \sum_{m \geq 1} \left( \frac{1}{x + m} + \frac{1}{x-m} \right). 
\label{eq:8.cotident}
\end{equation}
Let $\left\{ \omega_1, \omega_2 \right\}$ be generators of $\Gamma$ as in \ref{no:weierzeta} and denote $\tau = \omega_2/\omega_1$. We expand now to obtain
\begin{equation*}
\zeta(\omega_1 t) = \frac{1}{\omega_1 t} +  \omega_1 t b_0 + \frac{1}{\omega_1}   \sum_{n \neq 0}  \sum_{m \in \mathbb{Z}} \left( \frac{1}{t + n \tau + m} - \frac{1}{n\tau+m} \right) + \frac{1}{\omega_1} \sum_{m \neq 0} \left( \frac{1}{t + m} - \frac{1}{m} \right).
\end{equation*}
Here $b_0$ is defined as in (e) of \ref{no:weierzeta}, with the proviso that this non absolutely convergent sum is to be evaluated by summing first along subsets of $\Gamma$ of the form $n \omega_2 + \Z \omega_1$, and then summing on $n \in \Z$.
Using \eqref{eq:8.cotident} we can replace
\[ \sum_{m \in \mathbb{Z}} \left( \frac{1}{t + n\tau + m} - \frac{1}{n \tau + m} \right) = \pi \cot \pi (t + n \tau) - \pi \cot \pi n \tau = \frac{2 \pi i}{e^{2 \pi i (t + n \tau)} -1} - \frac{2 \pi i}{e^{2 \pi i n \tau} - 1}, \]
from where 
\[ \zeta(t) =  \frac{\pi i}{\omega_1} + \frac{2 \pi i}{\omega_1} \frac{1}{e^{2 \pi i t/\omega_1} -1} +  t b_0 + \frac{2 \pi i}{\omega_1} \sum_{n \neq 0} \left( \frac{1}{e^{2 \pi i (t/\omega_1 + n \tau)} - 1} - \frac{1}{e^{2 \pi i n \tau} -1 } \right).\]
In particular, we see that 
\begin{equation}
\zeta\left( \frac{\omega_1}{2} \right) = \frac{\omega_1 b_0}{2}, 
\label{eq:8.omega1}
\end{equation}
hence denoting $q = e^{2 \pi i \tau}$ and making the exponential change of coordinates $t \mapsto x=e^{2 \pi i t/\omega_1}$ we have
\begin{equation}
\frac{\omega_1}{2 \pi i} \, \widetilde{\zeta}(t) = \bar{\zeta}(x) := \frac{1}{2} + \frac{1}{x - 1} + \sum_{n \neq 0} \left( \frac{1}{q^n x-1} -  \frac{1}{q^n -1 }  \right).
\label{eq:8.zetatilde}
\end{equation}
\label{no:correction}
\end{nolabel}
\begin{nolabel}
The function $\widetilde{\zeta}(t)$ is not elliptic. In fact it follows from (d) in \ref{no:weierzeta} and \eqref{eq:8.omega1} that 
\[ \zeta\left( \frac{\omega_2}{2}  \right) = \frac{\omega_2 b_0}{2} - \frac{\pi i }{\omega_1}, \]
hence \ref{no:weierzeta}(c) reads now
\[\zeta(t + \omega_2) = \zeta(t) + \omega_2 b_0 - \frac{2 \pi i}{ \omega_1}, \]
and we arrive at the quasi-periodicity of $\widetilde{\zeta}$:
\begin{equation} \label{eq:zetatilde-period}
\widetilde{\zeta}(t + \omega_2) = \widetilde{\zeta}(t) - \frac{2 \pi i}{ \omega_1} \quad \Rightarrow \quad
 \bar{\zeta}(q\cdot x) = \bar{\zeta}(x) - 1. 
\end{equation}
\label{no:8.3}
In a similar fashion recall the identities
\begin{equation}
\frac{\pi^2}{\sin^2 \pi t} = \sum_{m \in \mathbb{Z} } \frac{1}{(t - m)^2}, \qquad   \frac{\pi^2}{6}=\sum_{m \geq 1} \frac{1}{m^2},
\label{eq:sin-ident}
\end{equation}
and expand
\[ \wp(\omega_1 t) = \frac{1}{ \omega_1^2  t^2} + \frac{1}{\omega_1^2} \sum_{n \neq 0} \sum_{m \in \mathbb{Z}} \left( \frac{1}{(t + n \tau + m)^2} - \frac{1}{(n\tau + m )^2} \right) + \frac{1}{\omega_1^2} \sum_{m \neq 0} \left( \frac{1}{(t+m)^2} - \frac{1}{m^2} \right).\] 
Using \eqref{eq:sin-ident} we replace
\begin{align*}
& \sum_{m \in \mathbb{Z}} \left( \frac{1}{(t + n \tau + m)^2} - \frac{1}{(n\tau + m )^2} \right)
= \frac{\pi^2}{ \sin^2 \pi (t + n \tau)}  - \frac{\pi^2}{ \sin^2 \pi  n \tau} \\
& \phantom{thequickbrownfoxjumped} = (2 \pi i)^2 \left( \frac{1}{e^{-2  \pi i  (t + n \tau) } ( e^{2 \pi i (t + n \tau)} - 1)^2 }  - \frac{1}{e^{- 2\pi i n \tau} (e^{2 \pi i n \tau} - 1)^2} \right),
\end{align*}
from where we obtain the Fourier expansion:
\[ \left( \frac{\omega_1}{2 \pi i} \right)^2 \wp (t) = \frac{e^{2 \pi i t/\omega_1}}{(e^{2 \pi i t/\omega_1} -1)^2} + \sum_{n\neq 0} \left( \frac{e^{2 \pi i (t/\omega_1 + n \tau)} }{(e^{2 \pi i (t/\omega_1 + n \tau)} - 1)^2} - \frac{e^{2 \pi i n \tau} }{(e^{ 2 \pi i n \tau} - 1)^2} \right)  + \frac{1}{12}. 
\]
In terms of the exponential coordinate $x = e^{2 \pi i t/\omega_1}$ we define
\[ \bar{\wp}(x) := \frac{x}{(x - 1)^2} + \sum_{n \neq 0} \left( \frac{q^n x}{(q^n x - 1)^2} - \frac{q^n}{(q^n - 1)^2} \right) + \frac{1}{12} - \frac{\omega_1 \zeta\left( \frac{\omega_1}{2} \right)}{2 \pi}  \]
and we obtain therefore that the modified zeta function $\bar{\zeta}$ in the exponential coordinate $x = e^{2 \pi i t}$ satisfies
\[
\bar\zeta(q x) = \bar\zeta(x) - 1, \qquad \bar\zeta'(x) = - \frac{\bar{\wp}(x)}{x}.
\]
\end{nolabel}

\begin{nolabel}\label{no:modularity}
The series $b_n$ appearing in Theorem \ref{no:weierzeta} (e) are of course the Eisenstein series which are well known to be modular forms for $SL(2, \Z)$, i.e., they transform nicely with respect to the parameter $\tau$ under modular transformations $\tau \mapsto \tfrac{a \tau + b}{c \tau + d}$ with $ab - cd = 1$. We will need more generally the following 
\begin{defn}\cite{zagier-jacobi} For two natural numbers 
 $k$ and $l$ a \emph{Jacobi Form} of weight $k$ and index $l$ is a holomorphic function $\varphi: \mathbb{H} \times \mathbb{C} \rightarrow \mathbb{C}$ satisfying the following two transformation equations
\begin{subequations} \label{eq:zagier-transformation}
\begin{gather}
\label{eq:zagier1}
\varphi\left( \tau, \alpha + m \tau + n \right) = e^{- 2 \pi i l (m^2 \tau + 2 m \alpha)} \varphi( \tau , \alpha ), \qquad (m,n)  \in \mathbb{Z}^2 \\ 
\varphi\left( \frac{a \tau + b}{c \tau + d}, \frac{\alpha}{c \tau +d} \right) = (c\tau + d)^{k} e^{2 \pi i l \frac{c \alpha^2}{c \tau + d} } \varphi\left( \tau , \alpha \right), \qquad \begin{pmatrix} a & b \\ c & d \end{pmatrix} \in SL(2, \mathbb{Z}) \label{eq:zagier2}
\end{gather}
\end{subequations}
In addition we require that for each $(\begin{smallmatrix}a & b \\ c & d \end{smallmatrix}) \in SL(2, \Z)$ the function on the LHS of \eqref{eq:zagier2} has a Fourier expansion of the form $\sum c(n,r) q^n y^r$ with $q = e^{2 \pi i \tau}$, $y= e^{2 \pi i \alpha}$ and $c(n,r) = 0$ unless $ n \geq r^2/4l$. If $\varphi$ satisfies the stronger condition that $c(n,r) = 0$ if $n> r^2/4l$ it is called a \emph{cusp form} and if it satisfies the weaker condition of $c(n,r) = 0$ if $n < 0$ then it is called a \emph{weak Jacobi form}. We denote the space of Jacobi forms of weight $k$ and index $l$ by $J_{k,l}$ and the space of weak Jacobi forms by $J'_{k,l}$. In the case of $J'_{k,l}$ the weight can be negative as long as $k \geq -2l$.  

One defines a Jacobi form with character by inserting a character of $SL(2, \mathbb{Z})\ltimes \mathbb{Z}^2$ in \eqref{eq:zagier1}--\eqref{eq:zagier2} as usual. In particular this is useful to allow $l$ to be a rational number instead of a natural number. We will be interested in the case of $l$ being a half integer in the Appendix. For $t=0$, $\varphi(\tau, 0)$ is a usual modular form of weight $k$, while for fixed $\tau$ we obtain elliptic functions. The Weierstrass function $\wp$ is a meromorphic Jacobi form of weight $2$ and index $0$. More examples can be produced as Eisenstein series as defined in \cite[Thm 2.1]{zagier-jacobi}:
\begin{equation}\label{eq:zagier-eisenstein}
E_{k,m} (\tau, \alpha) = \frac{1}{2} \sum_{\stackrel{c,d \in \mathbb{Z}}{(c,d)=1}}\sum_{\lambda \in \mathbb{Z}} (c\tau+d)^{-k} \exp \left( 2 \pi i m \left( \lambda^2 \frac{a \tau + b}{c \tau + d} + 2 \lambda \frac{\alpha}{c \tau + d} - \frac{c \alpha^2}{c \tau + d} \right) \right),
\end{equation}
which converges absolutely and uniformly when $k \geq 4$. In each summand $a,b$ are chosen so that $a d - bc = 1$. We refer the reader to the appendix for more examples. 
\label{defn:zagier}
\end{defn}
\end{nolabel}
\begin{rem}
In the main body of the article we will not be concerned with convergence of these forms, only their transformation properties will be of importance. They can be thought of as formal power series in $\mathbb{C}[ [q]]( (y))$. If the index is a semi-integer then they can be thought as formal power series in $\mathbb{C}[ [q]]( (y^{1/2}))$. Both $J_{*,*}$ and $J'_{*,*}$ form graded subrings with finite dimensional graded components.  
\label{rem:zagier-forms}
\end{rem}
\begin{nolabel}
Let $E := \mathbb{C}/\Gamma$ be the elliptic curve with period lattice $\Gamma$ for the differential $\omega = dt$. Under the exponential map $t \mapsto e^{2 \pi i t/\omega_1}$ we have $\pi: \mathbb{C}^* \rightarrow E \simeq \mathbb{C}^*/\mathbb{Z}$. Consider the meromorphic function $f(x)=(x-1)^{-1}$ in $\mathbb{C}^*$. We may think of the function $\tfrac{2 \pi i}{\omega_1} \widetilde{\zeta}$ as the push-forward $\pi_* f$. As such we would write analogously to \eqref{eq:x.1.d.1} 
\[ (\pi_* f)(x) = \sum_{n \in \mathbb{Z}} f(q^n x) = \sum_{n \in \mathbb{Z}} \frac{1}{q^n x - 1}. \]
This sum is however not convergent neither in the analytic topology nor in the formal neighborhood of $q=0$. In order to obtain a convergent sum we need to correct each summand as in \eqref{eq:8.zetatilde} at the cost of losing $\mathbb{Z}$-equivariance.  
\label{no:8.4}
\end{nolabel}
\begin{nolabel}
In what follows we will consider the lattice $\Gamma = \mathbb{Z} \oplus \mathbb{Z} \tau$, that is we put $\omega_1 = 1$. Let $V$ be a quasi-conformal vertex algebra and let $\cA$ and $\cA_E$ be the chiral algebras associated to $V$ in the setup of \ref{no:x.1.d}. Let $a \in V$ be a vector of conformal weight $\Delta_a \geq 1$,
We mark the point $z = 1 \in X$ and consider a section of $\cA$ given by 
\[
\widetilde{a}(x) = \frac{x^{\Delta_a-1} a}{x-1}
\]
meromorphic with a simple pole at $x = 1$. The sum \eqref{eq:x.1.d.1} of $\Z$-translates of $\widetilde{a}(x)$ is not convergent. Let us correct each term in an analogous fashion to Weierstrass's function, setting
\begin{equation}
\pi_*a(x) := \frac{x^{\Delta_a -1}a}{2} + \frac{x^{\Delta_a - 1} a}{x - 1} + \sum_{n \neq 0} \left( \frac{x^{\Delta_a -1} a}{q^n x - 1} - \frac{x^{\Delta_a-1} a}{q^n - 1} \right) = x^{\Delta_a -1} \bar{\zeta}(x) a.
\label{eq:8.5.1}
\end{equation}
This is a well defined section of $\cA$ however it is no longer $\mathbb{Z}$-equivariant. In fact under \eqref{eq:x.1.1} we have
\begin{equation} q\cdot \pi_* a(x) = \pi_* a(x) - x^{\Delta_a - 1} a. 
\label{eq:8.5.2}
\end{equation}
Let us repeat the proof of Proposition \ref{trace.is.cb} with this $\pi_* a(x)$. In the last equality of \eqref{eq:7.prop.2} we obtain an extra term coming from \eqref{eq:8.5.2}. In coordinates this gives
\begin{align}
\begin{split} 
\varphi_M (h(\pi_* a) \cdot b) 
 &=  (-1)^{ab} \res_x \str_{M} Y^M(b,z) x^{\Delta_a - 1} Y^M(a,x) q^{L_0} \\
 &=  (-1)^{ab} \str_M Y^M(b,z) a_0 q^{L_0}.
\end{split}
\label{eq:8.5.3}
\end{align}
\label{no:8.5}
\end{nolabel}
\begin{nolabel}
Now let $\omega \in V$ be the conformal vector. We want to evaluate \eqref{eq:8.5.3} with $a = \omega$. The last term reads
\[
\str_M Y^M(b,z) L_0 q^{L_0} = q \frac{d}{dq} \str_M Y^M(b,z) q^{L_0} = q \frac{d}{dq} \varphi_M(b).
\]
When we set $z=1$ and trivialize using the local coordinate, the RHS of \eqref{eq:8.5.3} is given by (c.f., Remark \ref{rem:bdrules})
\[
\varphi_M \Bigl( \res_{x-z} x \bar{\zeta}(x) L(x-z) b \Bigr) \Bigr|_{z=1} = \varphi_M \Bigl( \res_x (x+1)\bar{\zeta}(x+1) L(x) b\Bigr).
\]
We have shown that the conformal block $\varphi_M$ satisfies the differential equation
\begin{equation}
q \frac{d}{dq} \varphi_M (b) = \varphi_M \Bigl( \res_x (x+1)\bar{\zeta}(x+1) L(x) b\Bigr).
\label{eq:8.6.1}
\end{equation}
This equation coincides with that found by Zhu.
\label{no:8.6}
\end{nolabel}

\begin{nolabel}
Consider now the setup of \ref{no:ellipticfull} and assume $y \neq 1$ (the situation for $y=1$ is treated in an analogous way)  and introduce the following super zeta function $\bar{\zeta}(\bx;\bq) = \bar{\zeta}(x,\theta; q, y, \varepsilon, \delta)$ which we will denote with the same letter as in the non-super case:
\[
\bar{\zeta}(\bx; \bq) = \bar{\zeta}(x) \left( 1 - \frac{\varepsilon \delta \bar{\wp}(x)}{1-y} \right) + \frac{\theta \varepsilon}{1 - y} \frac{\bar{\wp}(x)}{x}.
\]
\begin{lem*}
The super zeta function satisfies
$\bar{\zeta}(\bq \bx ; \bq) = \bar{\zeta}(\bx;\bq) - 1$. 
\end{lem*}
\begin{proof}
Expanding in Taylor series with respect to $\theta$ and using the action of $\bq$ given by \eqref{eq:x.3.generalzaction}  we have
\begin{align*}
\bar{\zeta}(\bq \bx; \bq)
&= \bar{\zeta}(qx + q \varepsilon \theta) \left( 1 - \frac{\varepsilon \delta \bar{\wp}(qx + q \varepsilon \theta)}{1 - y} \right) + \frac{qy \theta \varepsilon + q \delta x \varepsilon}{1 - y} \frac{\bar{\wp}(qx + q \varepsilon \theta) }{qx + q \varepsilon \theta} \\
&= \bar{\zeta}(qx) \left( 1 - \frac{\varepsilon \delta \bar{\wp}(qx)}{1 - y} \right) + q \varepsilon \theta \bar{\zeta}'(qx) + \frac{y \theta \varepsilon + \delta x \varepsilon}{1 - y} \frac{\bar{\wp}(qx)}{x} \\
&= (\bar{\zeta}(x) -1) \left( 1 - \frac{\varepsilon \delta\bar{\wp}(x)}{1 - y} \right) - q \varepsilon \theta \frac{\bar{\wp}(qx)}{qx} + \frac{y \theta \varepsilon + \delta x \varepsilon}{1-y} \frac{\bar{\wp}(x)}{x}
= \bar{\zeta}(\bx; \bq) - 1.
\end{align*}
\end{proof}
\label{no:9.7}
\end{nolabel}
\begin{nolabel}
Let $V$ be a quasi-conformal $N_W=1$ SUSY vertex algebra and let $\cA$ and $\cA_E$ be the chiral algebras associated to $V$ as in \ref{no:ellipticfull}--\ref{no:generalelliptic}. Let $\{a, Q_0 a\}$ span a $1|1$ dimensional representation of $\gl(1|1)$ in $V$ with $\Delta_a \geq 1$ and $c_a = -1$. Consider the sections of $\cA$ given by
\[
a(\bx) = \frac{x^{\Delta_a} a}{x-z}, \qquad Q_0 a(\bx) = \frac{x^{\Delta_a - 1}\theta Q_0a}{x-z}.
\]
The sums of translates of these sections are divergent and do not give rise to well defined sections of $\cA_E$. Indeed quotienting out the odd parameters $\varepsilon, \delta$ reduces us to the situation of \eqref{eq:x.2.3} and Lemma \ref{no:x.2.d}, and we see from there that the sums are nonconvergent. Now, it is tempting to correct each summand by a term proprtional to $x^{\Delta_a} a$ (resp. $x^{\Delta_a -1}\theta Q_0a$) in analogy with \ref{no:8.6} so as to obtain a pair of sections
\[
x^{\Delta_a} \zeta(\bx) a, \qquad  \text{( resp. $x^{\Delta_a - 1} \theta \zeta(\bx) Q_0 a$} ).
\]
However there is a problem: the proposed corrections $x^{\Delta_a} a$ and $x^{\Delta_a -1}\theta Q_0a$ are \emph{not} $\mathbb{Z}$-equivariant (see the following remark), and so the sums do not possess the required quasi-periodicity. To correctly extend our arguments to the super case and derive differential equations satisfied by $\varphi_M$, we need to analyse the failure of $\mathbb{Z}$-equivariance of these sections.
\label{no:9.8}
\end{nolabel}
\begin{rem}
Let us return for a moment to the non super case. Another way that we might have obtained the correction \eqref{eq:8.5.1} is as follows. For $a \in V$ with conformal weight $\Delta_a$, the constant section $a$ is translation invariant. Under the exponential change of coordinates $x = e^{2 \pi i t} - 1$ this section reads $x^{\Delta_a - 1} a$ which in fact is invariant under $x \mapsto q x$. This is exploited in finding the failure of $\mathbb{Z}$-equivariance of $\pi_* a(x)$ in \eqref{eq:8.5.2}. 

In the same vein we consider an $N_W=1$ SUSY vertex algebra $V$ and let $\cA_{\mathbb{A}^{1|1}}$ the corresponding chiral algebra on the affine superline. Let $\{a, Q_0 a\}$ be a pair of \emph{superpartners} and consider the constant section $a$ of $\cA_{\mathbb{A}^{1|1}} $. It is invariant by affine translations. Under the exponential change of coordiantes $(x-1, \theta) = \rho(t,\zeta)$ given by $x = e^{2 \pi i t}$, $\theta = x \zeta$ this section reads $x^{\Delta_a} \rho^{-1}a + \theta x^{\Delta_a-1} Q_0 \rho^{-1}a$ as explained in \ref{no:zhu1}. We can apply this to the superpartners $\{j, h\} \subset V$ to produce a differential equation for $\varphi$. However we require \emph{two} differential equations to describe a connection on our family (we restrict attention to the family $y \neq 1$, which has dimension $2|0$). 
\label{rem:9.9}
\end{rem}
\begin{nolabel}
To avoid repeating notation and for simplicity we will restrict our attention to $V$ a strongly conformal $N_W=1$ SUSY vertex algebra. We put $a = h = H_{-1}\vac$ so that $Q_0a = j = J_{-1}\vac$. Note that in this case we have $\Delta_a = -c_a = 1$ and we are in the situation of \ref{no:linear-vectors}. If we restrict ourselves to \emph{linear sections} which are homogeous combinations of $h [dx d\theta]$ and $j \left[ dxd\theta \right]$ we see from \eqref{eq:action-linear-vectors} that in order to find $\mathbb{Z}$-equivariant sections we are looking for a $2 \times 2$ matrix in the centralizer of $\bq$. Since $\bq$ is a generic element of $GL(1|1)$ such a matrix would be a polynomial in $\bq$. There is always the identity matrix (corresponding to the Euler vector field) and $\bq$ itself. We arrive at the following
\begin{prop*}
Let $V$ be a strongly conformal $N_W=1$ SUSY vertex algebra and $\cA$ be the corresponding chiral algebra on $X/S$ as in \ref{no:ellipticfull}. 
Define the following two sections of $\cA$ by 
\begin{equation}\label{eq:9.10.1}
a_1(\bx) :=  h x + j \theta, \qquad a_2 (\bx) = a_2(\bx; \bq) :=  h (x + \varepsilon \theta) + j (\delta x + (y - \varepsilon \delta) \theta) 
\end{equation}
Then $a_1(\bx)$ and $a_2(\bx)$ are $\mathbb{Z}$-invariant, namely $\bq \cdot a_i (\bx) = a_i(\bx)$. Moreover, the section $a_1(\bx)$ is flat over $S$.  
\end{prop*}
\begin{proof}
Invariance follows by \eqref{eq:action-linear-vectors}, the flatness condition simply means that $a_1(\bx)$ does not depend on $\bq$. 
\end{proof}
\label{no:9.10}
\end{nolabel}
\begin{nolabel}
We define now: 
\[ \pi_* a_i(\bx) := \zeta(\bx; \bq) a_i(\bx), \qquad i = 1,2, \]
where $\zeta$ is the super Weierstrass zeta function defined in \ref{no:9.7} and $a_i(\bx)$ are given by \eqref{eq:9.10.1}. These sections so defined are meromorphic sections of $\cA$ which fail to be equivariant in the following way:
\[
\bq \pi_* a_i(\bx) = \pi_* a_i(\bx) - a_i(\bx).
\]
We can repeat the argument in \ref{no:8.5} and \eqref{eq:8.5.3} to obtain that
\begin{gather*}
\varphi_M( h\bigl(\pi_* a_1(\bx)) \cdot b\bigr) =  \str_M Y^M(b, \bz) 
 \res_\bx Y^M(a_1 (\bx), \bx) \bq^{-1}_M =  \str_M Y^M(b,\bz) L_0 \, \bq^{-1}_M \\ 
\varphi_M( h\bigl(\pi_* a_2(\bx)) \cdot b\bigr) =  \str_M Y^M(b,\bz) \left( L_0 - J_0 - \varepsilon H_0 + \delta Q_0 + (y - \varepsilon\delta) J_0 \right) \, \bq^{-1}_M  
\end{gather*}
Examining the definition \eqref{eq:x.6.12} of $\bq$ (or rather its inverse) reveals that $Q_0$ acts on it as $\partial_\delta$. Since $L_0$ is in the center of $GL(1|1)$ we have that the action of $L_0$ is given by $q \partial_q$. Indeed if we expand \eqref{eq:x.6.12} using the basis \eqref{eq:gl11} we obtain
\begin{equation*}
\bq^{-1} = 
\begin{pmatrix}
1 & -\delta\\ 
0 & 1 \end{pmatrix}
\begin{pmatrix}
q^{-1} & 0 \\ 
0 & q^{-1}y^{-1} 
\end{pmatrix}
\begin{pmatrix}
1 & 0 \\ 
- \varepsilon & 1 
\end{pmatrix} 
= q^{-1}y^{-1} \begin{pmatrix} y - \varepsilon \delta & - \delta \\ -\varepsilon & 1 
\end{pmatrix},
\end{equation*}
which can also be factorized as
\[ 
q^{-1} 
\begin{pmatrix}
1 & 0 \\ - \frac{\varepsilon}{y} & 1 \end{pmatrix}
\begin{pmatrix}
1 & 0 \\ 0 &  y^{-1} \end{pmatrix}
\begin{pmatrix}
1 & -\frac{\delta}{y} \\ 0 & 1 \end{pmatrix}
\left( 1- \frac{\varepsilon\delta}{y} \right) 
 = \left[ q \left( 1 + \frac{\varepsilon \delta}{y} \right) \right]^{L_0}  \left( 1 - \frac{\varepsilon}{y} H_0 \right) y^{J_0} \left( 1 + \frac{\delta}{y} Q_0 \right). \]
From here it follows that  
\[ \partial_\varepsilon \bq^{-1} =  y^{-1} (\delta L_0 - H_0) \bq^{-1} \quad \Rightarrow \quad H_0 \bq^{-1} =  \bigl( \delta q \partial_q - y \partial_\varepsilon \bigr) \bq^{-1} \]
Finally noticing that $\exp(\delta Q_0) J_0 = (J_0 - \delta Q_0) \exp(\delta Q_0)$ we obtain from \eqref{eq:x.6.12} that
\[ y \partial_y \bq^{-1} = (J_0 - \delta Q_0) \bq^{-1} \quad \Rightarrow \quad J_0 \bq^{-1} = \bigl(y \partial_y + \delta \partial_{\delta} \bigr) \bq^{-1}  \]
Collecting we obtain 
\begin{equation}
\begin{gathered}
\varphi_M\bigl(h(\pi_*(a_1(\bx))) \cdot b \bigr) = q \partial_q \varphi_M(b) \\ 
\varphi_M\bigl(h(\pi_*(a_2(\bx))) \cdot b \bigr) =  \Bigl( q \partial_q (1 - \varepsilon \delta) + y \partial y (y - 1 - \varepsilon \delta) + y (\varepsilon \partial_\varepsilon  + \delta \partial_\delta) \Bigr)\varphi_M (b)
\label{eq:9.11.1}
\end{gathered}
\end{equation}
On the other hand, putting $\bz = (1, 0)$, by Remark \ref{rem:bdrules} we have that the LHS of \eqref{eq:9.11.1} is given in coordinates by 
\[ \varphi_M \left( \res_\bx \bar{\zeta}(x+1,\theta; \bq) Y^M(a_i(x+1,\theta; \bq), \bx \bigr) b \right)\]
We arrive to the main 
\end{nolabel}
\begin{thm}
Let $V$ be a strongly conformal $N_W=1$ SUSY vertex algebra. Let $j = J_{-1} \vac$ and $h = H_{-1} \vac$ define the superconformal structure in $V$. Let $M$ be a positive energy  $V$-module and let $\varphi_M$ be the conformal block associated to it as in \eqref{eq:7.1.1}. Then $\varphi_M$ satisfies the following system of differential equations
\begin{subequations} \label{eq:primera-connection}
\begin{equation}
q \partial_q \varphi_M(b) = \varphi_M  \Biggl( \res_{\bx} \bar{\zeta}(x + 1, \theta) \Bigl( (x+1) Y^M(h, \bx) + \theta Y^M(j, \bx) \Bigr) b  \Biggr)  
\label{eq:9.12.1} 
\end{equation}
\begin{multline}
\Bigl( q \partial_q (1 - \varepsilon \delta) + y \partial y (y - 1 - \varepsilon \delta) + y (\varepsilon \partial_\varepsilon + \delta \partial_\delta) \Bigr)\varphi_M (b) = \\ 
 \varphi_M \Biggl(  \res_\bx \bar{\zeta}(x+1, \theta) \Bigl[ \Bigl( (x+1) + \varepsilon \theta \Bigr) Y^M(h, \bx) + \Bigl( \delta (x+1) + (y - \varepsilon\delta)\theta \Bigr) Y^M(j, \bx) \Bigr] b \Biggr) \label{eq:9.12.1c}
\end{multline}
\end{subequations}
\label{thm:connection.1}
\end{thm}
\begin{nolabel}
Recall from \ref{no:x.4.1} that the family \ref{no:ellipticfull} is trivial in the odd directions $\varepsilon, \delta$ in the locus $y \neq 1$. Restricting therefore to the family $\varepsilon = \delta = 0$ of \ref{no:x.2.c} with $y \neq 1$ we obtain a  much simpler system. The first equation \eqref{eq:9.12.1} reads in the same way, subtracting this equation from \eqref{eq:9.12.1c}  and dividing by $(y-1)$ we obtain 
\begin{equation}\label{eq:9.12.1b}
y \partial_y \varphi_M(b) = \varphi_M \Biggl( \res_\bx \bar{\zeta}(x+1, \theta) \theta Y^M(j, \bx) b \Biggr). \tag{\ref*{eq:9.12.1c}'}
\end{equation}
In fact we can describe \eqref{eq:primera-connection}  terms of usual fields \eqref{eq:expansion-usual} simply as 
\begin{subequations}\label{eq:9.13.1}
\begin{align}
q \partial_q \varphi_M(b) &= \varphi_M \left( \res_x (x+1) \bar{\zeta}(x+1)  L(x)   b \right) \label{eq:9.13.1.a} \\ 
y \partial_y \varphi_M(b) &= \varphi_M \Bigl( \res_x \bar{\zeta}(x+1) J(x) b \Bigr). \label{eq:9.13.1.b}
\end{align}
\end{subequations}
\label{no:9.13}
\end{nolabel}

\section{Modularity} \label{sec:modularity}
In this section we show how the differential equations \eqref{eq:9.13.1.a} and \eqref{eq:9.13.1.b} transform under a natural action of the Jacobi group $SL(2, \Z) \ltimes \Z^2$. This allows us to interpret these equations as the action of a flat connection on the family $E^0$ of elliptic supercurves and to interpret the conformal blocks $\varphi_M$ as \emph{flat sections}. The first step is to describe the differential equations \eqref{eq:9.12.1}-\eqref{eq:9.12.1b} in terms of the parameters $\tau$ and $\alpha$ instead of $q = e^{2 \pi i \tau}$ and $y=e^{2 \pi i \alpha}$. For this we must use $\cA_{\mathbb{A}^{1|1}}$ and the logarithmic coordinates $\bt=(t,\zeta)$ instead of $\cA_{\mathbb{G}_m^{1|1}}$ and the coordinates $\bx = (e^{2 \pi i t}, e^{2 \pi i t} \zeta)$. The differential equations are \emph{not} quite equivariant under the action of the Jacobi group introduced in Section \ref{no:jacobisuper}, but we prove that they become equivariant after we have corrected $\varphi_M$ by a factor of $y^{C/6}$ (this factor replaces the factor $q^{-C/24}$ which appears in the non-super case). We prove this by explicitly performing the changes of coordinates \eqref{eq:strangeequiva} and \eqref{eq:strangeequiv} and observing (in \eqref{eq:correction-z22}--\eqref{eq:correction-modular2}) the appearance of the Jacobi cocycle.


It follows that, in order for the differential equations in the coordinates $(\tau, \alpha)$ of $S = \mathbb{H}\times \mathbb{C}$ to be $SL(2, \Z) \ltimes \Z^2$-equivariant, it is necessary that the supertrace functions define a vector valued Jacobi form. This statement (for $V$ satisfying suitable finiteness conditions) has been proved for vertex operator algebras by Krauel and Mason in \cite{krauel}. In the superconformal setting the techniques used in this article can be used as an alternative way of proving the result of \cite{krauel} along the more geometric lines of \cite{zhu}.

Of course the non super case of this discussion is the content of Zhu's Theorem \cite{zhu}. We describe that setting in detail in \ref{no:10.1}--\ref{no:10.5a} before moving to the super case in \ref{no:10.5}--\ref{thm:final}. Throughout this section we will denote by $C$ the central charge to avoid confusing it with the entry in the modular matrix $(\begin{smallmatrix}a & b \\ c & d \end{smallmatrix}) \in SL(2,\mathbb{Z})$. 

\begin{nolabel}
We begin by writing \eqref{eq:8.6.1} in terms of the coordinates $t$, $\tau$ instead of $q= e^{2 \pi i \tau}$ and $x = e^{2 \pi i t}$. In deriving $\eqref{eq:8.6.1}$ we used the coordinate $x-1$ induced at the point $1 \in \mathbb{G}_m/\mathbb{Z}$ to identify the fiber of $\cA_E$ at $1$ with $V$. Recall the set-up of \ref{no:zhu0}. From the definition of $\varphi_M$ in \eqref{eq:7.1.1.prior} and Huang's formula \eqref{eq:huangs} we see that the LHS of \eqref{eq:8.6.1} reads in the new coordinate $t$ as
\[
q \partial_q \varphi_M(b) = \frac{1}{2 \pi i} \partial_\tau \varphi_M( \rho \cdot b)
\]
(i.e., when using $t$ to identify the fibers $\cA_x$ with $V$ and $\cM_x$ with $M$ as explained in \ref{no:zhu0}). Again from \eqref{eq:huangs} (see \cite[Thm. 4.2.1]{zhu}\footnote{Though note that Zhu's operator $T_\phi$ is $\rho^{-1}$ in our notation. See also \cite[7.2.2]{frenkelzvi}.}) the Virasoro field $L(x)$ becomes
\[
L(x-1)
= e^{-4 \pi i t} \rho^{-1} L(t) \rho + \frac{C}{24} e^{-4 \pi i t} \id_V.
\]
Since the Jacobian of the change of coordinates $\rho$ is $2 \pi i e^{2 \pi i t}$ we have $\res_x = 2 \pi i \res_t e^{2 \pi i t}$. Therefore the RHS of \eqref{eq:8.6.1} reads in the logarithmic coordinate as
\[
\varphi_M \bigl( \res_x x \bar{\zeta}(x) L(x - 1) b \bigr)
= \varphi_M \Bigl( \rho \cdot \res_t \widetilde{\zeta}(t) \bigl( (2 \pi i)^{-2} \rho^{-1} L(t) \rho + \frac{C}{24} \id_V \bigr) b \Bigr).
\]
Altogether, the differential equation \eqref{eq:8.6.1} reads in the coordinates $t, \tau$ as
\[
\begin{aligned}
\frac{1}{2 \pi i} \partial_\tau \varphi_M(\rho \cdot b)
&= \varphi_M \left( \rho \cdot \res_t \widetilde{\zeta}(t) \left( (2 \pi i )^{-2} \rho^{-1} L(t)\rho + \frac{C}{24} \id_V \right) b \right) \\
&= \varphi_M \left( \res_t \widetilde{\zeta}(t) \left( (2 \pi i )^{-2} L(t) + \frac{C}{24} \id_V \right) \rho \cdot  b \right).
\end{aligned}
\]
It follows that if we define
\[
\widetilde{\varphi}_M := e^{-2 \pi i \frac{C}{24} \tau} \varphi_M = q^{-C/24} \varphi_M,
\]
and note that $ \res_t \widetilde{\zeta}(t) b  = b$, then $\widetilde{\varphi}_M$  satisfies the differential equation:
\begin{equation}
\partial_\tau \widetilde{\varphi}_M (b) = \frac{1}{2 \pi i}\widetilde{\varphi}_M \left( \res_t \widetilde{\zeta}(t) L(t) b \right),
\label{eq:10.1.1}
\end{equation}
\label{no:10.1}
\end{nolabel}
\begin{nolabel}\label{no:10.2}
As described in Section \ref{no:usualuniversal} the modular group $SL(2,\mathbb{Z})$ acts on the pair $(t, \tau)$ by 
\begin{equation}\label{eq:10.2.1}A= \begin{pmatrix} a & b \\ c & d \end{pmatrix}: (t, \tau) \mapsto (t', \tau') = \left( \frac{t}{c\tau + d}, \frac{a \tau + b}{c \tau + d} \right), \qquad ad - bc = 1,\end{equation}
and this action descends to an isomorphism $\mathbb{A}^1/(\mathbb{Z} + \mathbb{Z} \cdot \tau)  \simeq \mathbb{A}^1/(\mathbb{Z} + \mathbb{Z}\cdot A\tau)$ for $A \in SL(2, \mathbb{Z})$.  In other words we have an action of $SL(2, \mathbb{Z})$ on the family $E/S$ of \ref{no:x.1.d}. The quotient $SL(2,\mathbb{Z})\backslash E$ is a universal family of elliptic curves.
\end{nolabel}
\begin{nolabel}\label{no:10.3}
Let $(t',  \tau') = A \cdot (t, \tau)$ as in \eqref{eq:10.2.1}. 
Recall the power series expansion of $\widetilde{\zeta}(t)$ from Theorem \ref{no:weierzeta}(e), and \eqref{eq:8.omega1}:
\[
\widetilde{\zeta}(t) = \frac{1}{t} - b_0 t - \frac{b_1t^3}{3} - \frac{b_2 t^5}{5} - \cdots,
\]
where the Eisenstein series $b_i = b_i(\tau)$ satisfies \cite[p.69 (50)]{appell}:
\[
b_i( \tau' ) = (c \tau + d)^{(2 i + 2)}b_i(\tau), \quad i \geq 1, \qquad b_0(\tau') = (c \tau + d)^2 b_0(\tau) - 2 \pi i c (c \tau + d).
\]
It follows that the zeta function transforms as
\[
\widetilde{\zeta}(t, \tau) = (c \tau + d)^{-1} \widetilde{\zeta}(t', \tau') - 2 \pi i c.
\]
Similarly we have $\res_{t} = (c \tau + d) \res_{t'}$ and 
$\partial_\tau = \frac{1}{(c \tau + d)^2} \partial_{\tau'}$.  Finally we need to compute how the section $b$ and the field $L(t)$ transform if we use the coordinate $t'$ instead of $t$ to trivialize $\cA$.Let us denote by $\sigma(t)= t'$ the change of coordinates. We simply have $\sigma= \sigma_t = (c\tau + d)^{L_0}$ hence $b$ reads as $(c\tau + d)^{-\Delta_b}b$ in the new coordinates, where $\Delta_b$ is the conformal weight of $b$. Similarly using \eqref{eq:huangs} we obtain $L(t) = (c \tau + d)^{-2} \sigma L(t') \sigma^{-1}$. We now have
\begin{equation}\label{eq:last-zhu}
\begin{aligned}
2 \pi i  \partial_\tau \widetilde{\varphi}_M(b) &= \widetilde{\varphi}_M \left( \res_t \widetilde{\zeta}(t) L(t) b \right)   \\
\frac{2 \pi i}{ (c \tau + d)^{2}} \partial_{\tau'} \widetilde{\varphi}_M (\sigma^{-1} b) &= \frac{1}{(c\tau+d)^2} \widetilde{\varphi}_M \left( \sigma_t^{-1} \res_{t'} \Bigl(\widetilde{\zeta}(t') - 2 \pi i  c (c \tau +d)\Bigr) \sigma L(t') \sigma^{-1}  b \right) \\
2 \pi i \partial_{\tau'} \widetilde{\varphi}_M(\sigma^{-1} b) &= \widetilde{\varphi}_M \left( \res_{t'} \widetilde{\zeta}(t') L(t') \sigma^{-1} b \right)
\end{aligned}
\end{equation}
where in the last equation we used Proposition \ref{no:prop1}. We arrive at the following
\begin{prop*}
Let $V$ be a conformal vertex algebra and consider the trivial bundle $\cE$ with fiber $V^*$ on $\mathbb{H}$. Consider the connection $\nabla$ on $\cE$ given in the coordinate $\tau$ of $\mathbb{H}$ by 
\[
\nabla = d + \frac{1}{2\pi i}\res_t \widetilde{\zeta}(t) L(t) d\tau.
\]
\begin{enumerate}
\item Each positive energy module $M$ of $V$ gives rise to a section $\widetilde{\varphi}_M = q^{-C/24} \varphi_M$ of $\cE$, as defined in section \ref{no:zhu0b},
flat with respect to $\nabla$.

\item Changing coordinates $\tau \mapsto \tau' = \tfrac{a \tau + b}{c \tau +d}$ and defining $\widetilde{\varphi}'_M := (c\tau + d)^{L_{0}} \cdot \widetilde{\varphi}_M $ we have $\nabla' \widetilde{\varphi}_M' = 0$, that is $\widetilde{\varphi}_M'$ satisfies the same differential equation with respect to $\tau'$ as $\widetilde{\varphi}$ satisfies with respect to $\tau$.
\item Let $E_{\tau'}$ be the elliptic curve $\mathbb{C}/(\mathbb{Z} + \mathbb{Z} \tau')$, $\cA_{E_{\tau'}}$ the chiral algebra associated to $V$ on $E_{\tau'}$, then $\widetilde{\varphi}_M'$ gives rise to a conformal block of $\cA_{E_{\tau'}}$ on $E_{\tau'}$, i.e., it satisfies \eqref{eq:7.prop.1}. 
\end{enumerate}
\end{prop*}
\begin{proof}
(a) is the content of \eqref{eq:10.1.1}. (b) follows by definition of the action of $\Aut \cO$ on $V^*$, namely $\widetilde{\varphi}'_M(b) = \widetilde{\varphi}\Bigl( (c\tau+d)^{-L_{0}} b \Bigr)$ and therefore the proposition follows from \eqref{eq:last-zhu} and the fact that $\sigma_{\bt} = \sigma = (c \tau + d)^{L_0}$. In order to prove c), letting $E_\tau$ be the elliptic curve with parameter $\tau$ and  $\cA_{E_{\tau}}$ the corresponding chiral algebra, the change of coordinates $\sigma: t \mapsto \tfrac{t}{c \tau + d}$ induces an isomorphism $E_\tau \simeq E_{\tau'}$ identifying $\cA_{E_{\tau}}$ with $\sigma^* \cA_{E_{\tau'}}$. Under this isomorphism $\widetilde{\varphi}_M'$ is identified with $\widetilde{\varphi}_M$ and the latter is a conformal block by Proposition \ref{no:prop1}. 
\end{proof}
\end{nolabel}
\begin{rem}
As discussed in \ref{no:zhu0b} $\varphi_M \in (\cA_{\mathbb{G}_m})_1^*$ coincides with $\Theta_M \in V^*$ when the coordinate $x = e^{2\pi i t}$ is used to idenify fibres of $\cA$ with $V$. In the proposition above we have used $t$ to trivialise. It follows that any statement on modular transformations of the explicit functions $\Theta_M$ will involve $\rho$ or Zhu's second vertex algebra structure \ref{no:zhu0}, as we see in the statement of his main theorem, which we reproduce below.

We also remark that the factor $(c\tau + d)^{L_{0}}$ arose naturally by analyzing the action of $SL(2,\mathbb{Z})$ on the differential equation $\nabla \widetilde{\varphi}_M = 0$ through changes of coordinates. In particular, if $\widetilde{\varphi}_M$ were invariant by this transformation, that is to say, if $\widetilde{\varphi}_M = \widetilde{\varphi}_M'$ then it would follow that $\nabla$ could be viewed as a connection on a bundle over $\mathbb{H}/SL(2, \mathbb{Z})$, and that $\widetilde{\varphi}_M$ would be a flat section of this bundle. However the condition $\widetilde{\varphi}_M = \widetilde{\varphi}_M'$ is in general false.
Instead, Zhu proved the following 
\begin{thm*}\cite[Thm 5.3.3.]{zhu}. Let $V$ be a rational $C_2$ cofinite vertex operator algebra, then the trace functions $q^{-C/24} \tr_M u_0 q^{L_0}$ converge to holomorphic functions in $\tau \in \mathbb{H}$ and their linear span is invariant under the action of $SL(2, \mathbb{Z})$ given by $\tau \mapsto \frac{a\tau+b}{c\tau+d}$ and $u \mapsto (c\tau+d)^{L_{[0]}} u$.
\end{thm*}
In particular, for fixed $b \in V$ primary of conformal weight $\Delta$, the span of the functions $\widetilde{\varphi}_M(b)$, as $M$ ranges over the set of irreducible positive energy $V$-modules, constitutes a \emph{vector valued} modular form of weight $\Delta$. Zhu proved this theorem in the following way. First one notices that the space of conformal blocks is finite dimensional (here the $C_2$ cofiniteness condition is essential). Then one shows that the trace functions $\varphi_M$ form a basis of this space where $M$ runs through the (finite) list of simple $V$-modules. The result then follows simply from c) in Proposition \ref{no:10.3} since $\widetilde{\varphi}_M'$ is a conformal block, and hence is a finite linear combination of $\widetilde{\varphi}_M$ for $V$-modules $M$. We have to remark though that in Zhu's approach, the differential equation \eqref{eq:10.1.1} is taken as part of the \emph{definition} of a conformal block in genus one, rather than as a consequence of the general geometric definition in \ref{no:5.6}. 
We summarize this discussion in the following
\label{rem:zhu-need-for-modularity}
\end{rem}
\begin{thm}[\cite{zhu}]\label{no:10.5a}
Let $V$ be a rational $C_2$-cofinite  vertex algebra and consider the family of elliptic curves $\pi: X/S \twoheadrightarrow E/S$ defined in \ref{no:x.1.d}. Let $M$ be an irreducible positive energy $V$-module. Let $\cA$ (resp. $\cA_E$) be the chiral algebra on $X$ (resp. $E$) associated to $V$ and $\cM$ be the $\cA$-module supported at $x_0 = 1$ associated to $M$. Define $\varphi_M$ by \eqref{eq:7.1.1.prior} and put $\widetilde{\varphi}_M := q^{-C/24}\varphi_M$. Then 
\begin{enumerate}

\item $\widetilde{\varphi}_M$ gives rise to a conformal block of $\cA_E$ on $E$. 

\item The linear space spanned by the corrected $\widetilde{\varphi}_M$ for $M$ running through the (finite) list of irreducible $V$-modules is invariant with respect to the action of $SL(2, \mathbb{Z})$ defined by
\[
\left[ \widetilde{\varphi}_M \cdot
\begin{pmatrix}
a & b \\ c & d
\end{pmatrix} 
\right](v; \tau) = \widetilde{\varphi}_M\left( (c\tau + d)^{-L_0} v; \frac{a \tau + b}{c \tau+d} \right). \qquad 
\begin{pmatrix}
a & b \\ c & d 
\end{pmatrix} \in SL(2,\mathbb{Z}),
\]
In particular, for $v \in V$ primary of conformal weight $\Delta$, the collection of $\widetilde{\varphi}_M(b)$ gives rise to a vector valued modular form of weight $\Delta$. 
\item $\varphi_M$ satisfies the differential equation  expressed in local coordinates by \eqref{eq:8.6.1}. This differential equation is not invariant by the action of $SL(2,\mathbb{Z})$. However, the correction $\widetilde{\varphi}_M$ satisfies the differential equation \eqref{eq:10.1.1} which is invariant by $SL(2,\mathbb{Z})$. In other words, $\nabla$ defined as in Proposition \ref{no:10.3} is $SL(2,\mathbb{Z})$ equivariant and therefore descends to a connection on the bundle of conformal blocks in the quotient $\mathbb{H}/SL(2,\mathbb{Z})$. The conformal blocks $\widetilde{\varphi}_M$ are flat sections with respect to $\nabla$.  
\end{enumerate}
\end{thm}

\begin{nolabel}\label{no:10.5}
We now proceed to extend Zhu's results above to the supersymmetric setting. For simplicity we will consider the setup of \ref{no:x.2.c} and \ref{no:x.2.d}, i.e., we consider the family $E^0$ of elliptic supercuves over the purely even base $S = S^{2|0}$ parametrized by $(\tau, \alpha)$. The family over the $1|2$ dimensional base restricting \ref{no:ellipticfull} to $\alpha = 0$ can be treated with the same techniques. We proceed in the same way as in the non-super case, namely we first need to express the differential equations \eqref{eq:9.12.1} in terms of the logarithmic coordinates $(t, \zeta)$ of $\mathbb{A}^{1|1}$ instead of $(x= e^{2\pi i t}, \theta = e^{2 \pi i t} \zeta)$ of $\mathbb{G}_m^{1|1}$ and we need to use the coordinates $\tau, \alpha$ instead of $q = e^{2 \pi i \tau}$ and $y = e^{2 \pi i \alpha }$.  

Let $V$ be a strongly conformal $N_W=1$ SUSY vertex algebra and let $\pi: X/S \rightarrow E^0/S$ be the family of elliptic supercurves from \ref{no:x.2.c}. Let $\cA$ and $\cA_{E^0}$ be the chiral algebras on $X$ and $E^0$ corresponding to $V$. Let $M$ be a $V$-module and let $\cM_\bz$ be the corresponding module supported at $\bz = (1,0)$. We have used the coordinates $(x-1, \theta)$ induced at the point $\bz$ to identify the fiber of $\cA$ and $\cM$ at $\bz$, now we wish to pass to the coordinates $(t, \zeta)$, where $\rho (t,\zeta) = (e^{2 \pi i t} -1, e^{2 \pi i t} \zeta)$. Under this change of coordinates the LHS of \eqref{eq:9.12.1} and \eqref{eq:9.12.1b} read 
\[ \frac{1}{2 \pi i} \partial_\tau \varphi_M(\rho \cdot b) , \qquad \frac{1}{2 \pi i} \partial_\alpha \varphi_M(\rho \cdot  b). \]

The Berezinian of $\rho$ equals $2 \pi i$, therefore $\res_\bx = 2 \pi i \res_\bt$. Putting $\rho = (F, \Psi)$ as in Example \ref{ex:5.8}, the system \eqref{eq:5.1.1} simplifies to give 
\[ \delta_1 = \alpha_1 = 0, \qquad \tau_1 = \frac{1}{2} 2 \pi i,\]
and \eqref{eq:5.8.2b} and its inverse read\footnote{we abuse notation by denoting the restriction of $\rho_t^{\pm 1}$ from $V$ to the subspace spanned by $\vac, h, j$, by the same letter.}
\[ 
\rho_\bt^{-1} = 
\begin{pmatrix}
1 & 0 &  \frac{C}{6} 2 \pi i \\ 
0 & (2 \pi i)^2 e^{2 \pi i t} & 0 \\
0 & (2 \pi i)^2 e^{2 \pi i t} \zeta & 2 \pi i e^{2 \pi i t}
\end{pmatrix}, \qquad 
\rho_\bt = 
\begin{pmatrix}
1 &  \frac{C}{6} e^{-2 \pi i t}\zeta & -\frac{C}{6} e^{-2 \pi i t}\\ 
0 & \frac{e^{-2 \pi i t}}{(2\pi i)^2} & 0 \\ 
0 & - \frac{e^{-2 \pi i t}\zeta}{2 \pi i} & \frac{e^{-2 \pi i t}}{2 \pi i} 
\end{pmatrix}.
\]
These formulas, and the change of coordinate formula for superfields (and their actions on $M$ as well), yield
\begin{subequations} \label{eq:10.5.1}
\begin{align}
Y(j, \rho(\bt)) &=\frac{e^{-2 \pi i t}}{2 \pi i} \rho^{-1} Y(j, \bt) \rho - e^{-2 \pi i t} \frac{C}{6}, \\
Y(h, \rho(\bt)) &= \frac{e^{-2 \pi i t}}{(2 \pi i)^2} \rho^{-1}Y(h,\bt) \rho - \frac{e^{-2 \pi i t}\zeta}{2 \pi i} \rho^{-1}Y(j,\bt) \rho + \frac{C}{6} e^{-2 \pi i t}\zeta.
\end{align}
\end{subequations}
Altogether, equations \eqref{eq:9.12.1} and \eqref{eq:9.12.1b} read in the logarithmic coordinates as 
\begin{equation} \label{eq:8.log.subs}
\begin{aligned}
 \partial_\tau  \varphi_M( \rho \cdot b) &= \frac{1}{2 \pi i} \varphi_M\left( \res_\bt \widetilde{\zeta}(t) Y^M(h,\bt) \rho \cdot  b  \right), \\ 
\partial_\alpha \varphi_M(\rho \cdot b) &= \varphi_M\left( \res_\bt \widetilde{\zeta}(t) \zeta \left(  Y^M(j, \bt) - 2 \pi i \frac{C}{6}  \right) \rho \cdot  b \right).
\end{aligned}
\end{equation}
It is now easy to see that the correction 
\begin{equation}\label{eq:10.5.3}
\widetilde{\varphi}_M(b) = e^{2 \pi i \alpha  C/6} \varphi_M = y^{C/6} \varphi_M
\end{equation}
satisfies the differential equations
\begin{subequations}\label{eq:10.5.4}
\begin{align}
 \partial_\alpha \widetilde{\varphi}_M(b) &= \widetilde{\varphi}_M \left( \res_\bt \widetilde{\zeta}(t) \zeta Y^M(j, \bt)b  \right), \label{eq:10.5.4.b} \\ 
 \partial_\tau \widetilde{\varphi}_M(b) &=\frac{1}{2\pi i} \widetilde{\varphi}_M \left( \res_\bt \widetilde{\zeta}(t) Y^M(h, \bt) b \right). \label{eq:10.5.4.a}
\end{align}
\end{subequations}
\end{nolabel}
\begin{lem}
Let $\beta, \gamma$ be two even constants and consider the change of coordinates $\sigma(t, \zeta) = (t/\gamma, e^{\beta t} \zeta)$. Then we have 
\[
\sigma = e^{-\beta J_1} \gamma^{L_0 - J_0}.
\]
\label{lem:10.7}
\end{lem}
\begin{proof}
Expand
\[ e^{-\beta J_1} \gamma^{L_0-J_0} \cdot (t, \zeta) = e^{\beta t \zeta \partial_\zeta} \gamma^{- t\partial_t} (t, \zeta) = e^{\beta t \zeta \partial_\zeta } \left( \frac{t}{\gamma}, \zeta \right) = \left( \frac{t}{\gamma}, e^{\beta t} \zeta \right) = \sigma(\bt). \]
\end{proof}
\begin{nolabel}
As explained in \ref{no:jacobisuper} the Jacobi group $SL(2, \mathbb{Z}) \ltimes \mathbb{Z}^2$ acts on $\mathbb{H} \times \mathbb{C}$ by \eqref{eq:sl2ext2}-\eqref{eq:sl2ext3} and the family $E^{0}$ is equivariant for this action with \eqref{eq:strangeequiva}-\eqref{eq:strangeequiv}. 
Consider $\gamma \in \mathbb{Z}^2$ and let $(\tau', \alpha') = \gamma \cdot (\tau, \alpha)$ as in \eqref{eq:sl2ext3} and $\sigma(\bt) = \bt' = (t', \zeta') = \gamma \cdot (t, \zeta) = \gamma \cdot \bt$ as in \eqref{eq:strangeequiva}. The LHS of \eqref{eq:10.5.4.a}-\eqref{eq:10.5.4.b} transforms as 
\[ \partial_\tau = \partial_{\tau'} + m \partial_{\alpha'}, \qquad \partial_{\alpha} = \partial_{\alpha'}. \]
 The Berezinian of the change of coordinate \eqref{eq:strangeequiva} is $e^{-2 \pi i m t}$ from which $\res_\bt = e^{2 \pi i m t} \res_{\bt'}$. The system \eqref{eq:5.1.1} is easily solved and \eqref{eq:5.8.2b} is given by
\[ 
\begin{pmatrix}
1 & \frac{C}{6} (2 \pi i m)^2 \zeta & \frac{C}{3} 2 \pi i m \\ 
0 & e^{-2 \pi i m t} & 0 \\ 
0 & 2 \pi i m \zeta & 1 
\end{pmatrix}.
\]
It follows that the superfields transform as
\begin{gather*}
Y(h, \bt) = e^{-2 \pi i m t} \sigma Y(h, \bt') \sigma^{-1} + 2 \pi i m \zeta \sigma Y(j, \bt')\sigma^{-1} + \frac{C}{6} (2 \pi i m)^2  \\
Y(j,\bt) = \sigma Y(j, \bt') \sigma^{-1} + \frac{C}{3} 2 \pi i m 
\end{gather*}
Equation \eqref{eq:10.5.4.b} is therefore given in the coordinates $\bt'$ by
\begin{subequations}
\begin{equation}
\partial_{\alpha'} \widetilde{\varphi}_M(\sigma^{-1} b) = \widetilde{\varphi}_M \Biggl( \res_{\bt'} \widetilde{\zeta}(\bt') \zeta'\left( Y^M(j, \bt') + \frac{C}{3} 2 \pi i m \right)  \sigma^{-1} b \Biggr).
\end{equation}
On the other hand equation \eqref{eq:10.5.4.a} is given by
\begin{align}
\begin{split}
\partial_{\tau'} \widetilde{\varphi}_M(\sigma^{-1} b)
= {} & - m \partial_{\alpha'} \widetilde{\varphi}_M(\sigma^{-1} b) + \\
&+ \frac{1}{2 \pi i} \widetilde{\varphi}_M \Biggl( \res_{\bt'} \widetilde{\zeta}(\bt') \Bigl( Y^M(h, \bt') + 2 \pi i m \zeta' Y^M(j, \bt') +  \frac{C}{6} (2 \pi i m)^2 \zeta'\Bigr) \sigma^{-1} b \Biggr) \\
= {} & \frac{1}{2 \pi i} \widetilde{\varphi}_M \Biggl( \res_{\bt'} \widetilde{\zeta}(\bt') \Bigl( Y^M(h, \bt')  - \frac{C}{6}\zeta' (2 \pi i m)^2  \Bigr)\sigma^{-1} b \Biggr). 
\end{split}
\end{align}
\end{subequations}
We easily see that the normalised conformal block
\begin{align}\label{eq:correction-z2}
\begin{split}
\Psi_M(b; \tau', \alpha')
&:= q'^{\frac{C}{6} m^2} y'^{-\frac{C}{6} 2 m} \widetilde{\varphi}_M (\sigma^{-1} b)
= \exp \left( \frac{C}{6} 2 \pi i (m^2 \tau' - 2 m \alpha') \right) \widetilde{\varphi}_M (\sigma^{-1} b; \tau, \alpha) \\
&= \exp \left( - \frac{C}{6} 2 \pi i (m^2 \tau + 2 m \alpha) \right) e^{- \frac{C}{3} 2 \pi i n} \widetilde{\varphi}_M(\sigma^{-1} b)
\end{split}
\end{align}
satisfies the same equations \eqref{eq:10.5.4} in $\alpha', \tau'$ as $\widetilde{\varphi}_M$ does in $\alpha, \tau$. Namely
\[
\begin{aligned}
 \partial_{\alpha'} \Psi_M(b) &= \Psi_M \left( \res_\bt \widetilde{\zeta}(t) \zeta Y^M(j, \bt)b  \right),  \\ 
 \partial_{\tau'} \Psi_M(b) &=\frac{1}{2\pi i} \Psi_M \left( \res_\bt \widetilde{\zeta}(t) Y^M(h, \bt)  b \right).
\end{aligned}
\]

Consider now the modular transformation $(\tau', \alpha') = A \cdot (\tau, \alpha)$ as in \eqref{eq:sl2ext2} and $\bt' = (t', \zeta') = A \cdot \bt$ as in \eqref{eq:strangeequiv}. Let $\sigma$ be the automorphism of $V$ associated to the change of coordinates \eqref{eq:strangeequiv}. The LHS of \eqref{eq:10.5.4.a}-\eqref{eq:10.5.4.b} transforms by 
\[ \partial_{\tau} = \frac{\partial_{\tau'} - c \alpha \partial_{\alpha'} }{(c \tau + d)^2} \qquad \partial_{\alpha} = \frac{\partial_{\alpha'}}{c\tau+d}.\]
Computing the Berezinian of the change of coordinates \eqref{eq:strangeequiv} we obtain \[ \res_{\bt} = \Ber\sigma^{-1} \res_{\bt'} =  (c\tau + d) e^{- 2 \pi i t \frac{c \alpha}{c\tau+d}}\res_{\bt'}.  \]
Put $\beta = -(2 \pi i)  \tfrac{c \alpha }{c\tau + d}$ and $\gamma = c\tau+d$. Refering again to Example \ref{ex:5.8}, the system \eqref{eq:5.1.1} for $\sigma = (F,\Psi)$ given by \eqref{eq:strangeequiv} is:
\begin{gather*}
q = \frac{1}{\gamma}, \qquad \varepsilon_0 = 0, \qquad \delta_0 = \gamma \beta e^{\beta t} \zeta, \qquad 
y = \gamma e^{\beta t},\\  \qquad \tau_1 = \varepsilon_1 = 0, \qquad \alpha_1 = \beta, \qquad \delta_1 = \frac{1}{2} \beta^2 \zeta.
\end{gather*}
Plugging into \eqref{eq:5.8.2b} gives the action of $\sigma^{-1}_{\bt}$ in the subspace of $V$ with basis $\vac, h, j$ as left multiplication by the matrix
\[  
\begin{pmatrix}
1 & \frac{C}{6} \beta^2 \zeta & \frac{C}{3} \beta \\ 
0 & \frac{e^{-\beta t}}{\gamma^2} & 0  \\ 
0 & \frac{\beta \zeta}{\gamma} & \frac{1}{\gamma}
\end{pmatrix}.
\]
Therefore the fields change coordinates as
\begin{equation}\label{eq:10.7pre}
\begin{gathered}
Y(j, \bt) = \gamma^{-1} \sigma Y(j, \bt') \sigma^{-1} + \frac{C}{3} \beta,\\
Y(h, \bt) = \gamma^{-2} e^{-\beta t} \sigma Y(h, \bt')\sigma^{-1} + \gamma^{-1} \beta \zeta \sigma Y(j, \bt') \sigma^{-1} + \frac{C}{6} \beta^2 \zeta.
\end{gathered}
\end{equation}
We see that in the coordinates $\bt'$ \eqref{eq:10.5.4.b} reads
\begin{subequations}
\begin{equation*}
\partial_{\alpha'} \widetilde{\varphi}_M( \sigma^{-1} b) = \widetilde{\varphi}_M \Bigr( \res_{\bt'} \widetilde{\zeta}(\bt') \zeta' Y(j,\bt') \sigma^{-1} b \Bigr) + \widetilde{\varphi}_M \Bigl( \res_{\bt'} \widetilde{\zeta}(\bt') \zeta' \frac{C}{3} \gamma \beta \sigma^{-1} b \Bigr).
\end{equation*}
We express $\gamma \beta$ in terms of $\alpha'$ and $\tau'$ using 
\[ c \tau + d = - (c \tau' - a)^{-1} \quad \Rightarrow \quad \gamma \beta = - 2 \pi i c \alpha =  2 \pi i \frac{c \alpha'}{c \tau' - a} \]
to obtain
\begin{equation}
\partial_{\alpha'} \widetilde{\varphi}_M(\sigma^{-1} b) = \widetilde{\varphi}_M \Bigr( \res_{\bt'} \widetilde{\zeta}(\bt') \zeta' Y(j,\bt') \sigma^{-1} b \Bigr) + \widetilde{\varphi}_M \Bigl( \res_{\bt'} \widetilde{\zeta}(\bt') \zeta' 2 \pi i  \frac{C}{3} \frac{c \alpha'}{c \tau' - a}\sigma^{-1} b \Bigr).
\label{eq:10.7aaa}
\end{equation}
We now express \eqref{eq:10.5.4.a} in the new parameters $\tau', \alpha'$:
\[
\left( \partial_{\tau'} - c \alpha \partial_{\alpha'}  \right) \widetilde{\varphi}_M(\sigma^{-1} b) =\frac{1}{2 \pi i} \widetilde{\varphi}_M \Biggl( \res_{\bt'} \widetilde{\zeta}(\bt') \Bigl( Y(h,\bt') + \beta \gamma  \zeta' Y(j,\bt') + \frac{C}{6} \gamma^2 \zeta' \beta^2  \Bigr)\sigma^{-1} b \Biggr)
\]
And using \eqref{eq:10.7aaa} we write this as 
\begin{equation}\label{eq:10.8}
\begin{aligned}
\partial_{\tau'} \widetilde{\varphi}_M(\sigma^{-1} b) &=\frac{1}{2 \pi i} \widetilde{\varphi}_M \Biggl( \res_{\bt'} \widetilde{\zeta}(\bt') \Bigl( Y(h,\bt') - \frac{C}{6} (2 \pi i c \alpha)^2 \Bigr) \sigma^{-1} b \Biggl)\\
&= \frac{1}{2 \pi i} \widetilde{\varphi}_M \Biggl( \res_{\bt'} \widetilde{\zeta}(\bt') \left[ Y(h,\bt') - \frac{C}{6} \left(2 \pi i \frac{c \alpha'}{c \tau' - a} \right)^2 \right] \sigma^{-1} b \Biggl).
\end{aligned}
\end{equation}
\end{subequations}
We see easily that the corrected conformal block 
\begin{equation}\label{eq:correction-modular1}
\begin{aligned}
\Psi_M (b) = \Psi_M(b; \tau', \alpha')
&:= \exp \left( 2 \pi i \frac{C}{6} \left( - \frac{c \alpha'^2}{c \tau' - a} \right)  \right) \widetilde{\varphi}_M (\sigma^{-1}b;\tau, \alpha) \\
&= \exp \left( 2 \pi i \frac{C}{6} \left( \frac{c \alpha^2}{c \tau + d} \right)  \right) \widetilde{\varphi}_M(\sigma^{-1} b; \tau, \alpha)
\end{aligned}
\end{equation} 
satisfies the same equations at $\alpha', \tau'$ as $\widetilde{\varphi}_M$ does at $\alpha,\tau$. 
\end{nolabel} 
\begin{nolabel}\label{no:10.9}
If $b \in V$ is a primary vector of conformal weight $\Delta$ and charge $0$ we have from Lemma \ref{lem:10.7} for the action of $\mathbb{Z}^2$ that the correction \eqref{eq:correction-z2} is given by
\begin{subequations}
\begin{equation}\label{eq:correction-z22} \Psi_M (b; \tau, \alpha) =
\exp \left(\frac{C}{6} 2 \pi i (m^2 \tau + 2  m \alpha) \right) e^{\frac{C}{3} 2 \pi i n} \widetilde{\varphi}_M\left( b; \tau, \alpha + m \tau + n \right),
\end{equation}
while for the action of the modular group $SL(2, \mathbb{Z})$, the correction \eqref{eq:correction-modular1} is given by
\begin{equation}\label{eq:correction-modular2}
\Psi_M (b;\tau, \alpha) = (c\tau + d)^{ -\Delta} \exp \left( 2 \pi i \frac{C}{6} \left( - \frac{c \alpha^2}{c \tau + d} \right)  \right) \widetilde{\varphi}_M \left( b; \frac{a \tau + b}{c \tau + d}, \frac{\alpha}{c \tau + d} \right). 
\end{equation}
\end{subequations}
We recognize precisely the transformation properties of Jacobi forms \eqref{eq:zagier-transformation} with weight $k = \Delta$, index $l = C/6$ and character  if $C$ is not divisible by $6$. We have arrived at the super-analog to Proposition \ref{no:10.3}. 
\end{nolabel}
\begin{thm}\label{thm:10.10}
Let $V$ be a conformal $N_W=1$ SUSY vertex algebra and consider the trivial bundle $\cE$ with fiber $V^*$ on $\mathbb{H} \times \mathbb{C}$. Consider the connection $\nabla$ on $\cE$ given in the coordinates $\tau, \alpha$ by 
\[
\nabla = d + \frac{1}{2 \pi i } \res_\bt \widetilde{\zeta}(t)  \Bigl( Y(h, \bt)d \tau  + 2 \pi i \,  \zeta\, Y(j, \bt) d \alpha \Bigr).
\]
\begin{enumerate}
\item $\nabla$ is flat and each positive energy module $M$ of $V$ gives rise to a section $\widetilde{\varphi}_M = y^{C/6} \varphi_M$ of $\cE$, as defined in Section \ref{no:7.1},
flat with respect to $\nabla$. 

\item Changing coordinates $\tau \mapsto \tau' = \tau$, $\alpha \mapsto \alpha' = \alpha + m \tau + n$ for $(m,n) \in \mathbb{Z}^2$ and defining $\Psi_M$ by \eqref{eq:correction-z2} yields $\nabla' \Psi_M = 0$, that is, $\Psi_M$ satisfies the same differential equation with respect to $\tau', \alpha'$ as $\widetilde{\varphi}_M$ does with respect to $\tau, \alpha$.

\item Changing coordinates $\tau \mapsto \tau' = \tfrac{a \tau + b}{c \tau + d}$, $\alpha \mapsto \alpha' = \tfrac{\alpha}{\tau + d}$ and defining $\Psi_M$ by \eqref{eq:correction-modular1} yields $\nabla' \Psi_M = 0$, that is, $\Psi_M$ satisfies the same differential equation with respect to $\tau', \alpha'$ as $\widetilde{\varphi}_M$ does with respect to $\tau, \alpha$. 

\item Let $E_{\tau', \alpha'}$ be the superelliptic curve with parameters $\tau', \alpha'$ as in b) (resp. c), and $\cA_{E_{\tau', \alpha'}}$ be the corresponding chiral algebra. then $\Psi_M$ defined as in b) (resp. c) is a conformal block in $C(E_{\tau', \alpha'}, (1,0); \cA_{E_{\tau', \alpha'}})$.  
\end{enumerate}
\end{thm}
\begin{proof}
The only part that we have not  proved is flatness of $\nabla$. Expanding in terms of usual fields as in \eqref{eq:expansion-usual} we see that  the commutator $[\nabla_\tau, \nabla_\alpha]$ is given by  
\begin{multline} \label{eq:multi-flat}
\left[ \frac{1}{2 \pi i} \res_t \tilde{\zeta}(t) \left( L(t) + \partial_t J(t) \right), \res_{t'}  2 \pi i \tilde{\zeta}(t') J(t') \right] =\\  \res_t \res_{t'} \tilde{\zeta}(t) \tilde{\zeta}(t') \left( \delta(t,t') \partial_{t'} J(t') + J(t') \partial_{t'} \delta(t,t') - \frac{C}{12} \partial^2_{t'} \delta(t,t')  \right) = \\ \res_{t'} \left( \tilde{\zeta}(t')^2 \partial_{t'}  J(t') + J(t') \tilde{\zeta}(t') \partial_{t'} \tilde{\zeta}(t') - \frac{C}{12} \tilde{\zeta}(t') \partial^2_{t'} \tilde{\zeta}(t')\right) =  \\ - \res_{t'} \tilde{\zeta}(t') \left( J(t') \partial_{t'} \tilde{\zeta}(t')  + \frac{C}{12} \partial_{t'}^2 \tilde{\zeta}(t') \right) \end{multline}
Note that the last term in parenthesis in \eqref{eq:multi-flat} is meromorphic with poles only at the lattice points $m \tau + n$,  as $\partial_{t'} \tilde{\zeta}(t') = - \wp(t') - b_0$ is the elliptic function described in \ref{no:weierzeta}.  As in \eqref{eq:8.5.3} and using \eqref{eq:zetatilde-period} the first term acts on $\varphi_M$ in the same way as  
\[  2 \pi i \res_{t} J(t) \partial_t \tilde{\zeta}(t) = - 2 \pi i \res_t \tilde{\zeta}(t) \partial_t J(t), \]
since the failure of $\tilde{\zeta}(t) J(t) \partial_t \tilde{\zeta}(t)$ to be elliptic is that when $t \mapsto t + \tau$ this product gets added a term $2 \pi i J(t) \partial_t \tilde{\zeta}(t)$. 
Applying again the same argument we obtain that the first term in the RHS of \eqref{eq:multi-flat} acts on $\varphi_M$ as $(2\pi i)^2\res_t \partial_t J(t) = 0$. The second term vanishes because of the same argument. 
\end{proof}
\begin{nolabel}
As in Remark \ref{rem:zhu-need-for-modularity} we see, simply by analyzing the changes of coordinates given by the action of the Jacobi group $SL(2,\mathbb{Z}) \ltimes \mathbb{Z}^2$ in our family of elliptic supercurves \ref{no:jacobisuper}, that we obtain the cocycle defining the transformation properties of Jacobi forms in the corrected conformal blocks \eqref{eq:correction-z22}-\eqref{eq:correction-modular2}. Once again, strict invariance of the system of differential equations \eqref{eq:10.5.4} would be implied if we had the equality $\Psi_M (b; \tau', \alpha') = \widetilde{\varphi}_M(b;\tau',\alpha')$, which in turn would mean that the trace function $\widetilde{\varphi}_M(b;\tau, \alpha)$ is a Jacobi form of weight $\Delta$ and index $C/6$ for a primary vector $b$. Once again this condition is false in general. Instead one has the following result which is a particular case of a theorem of Krauel and Mason.
\begin{thm*}\cite[Thm 1.1]{krauel} Let $V$ be a conformal $N_W=1$ SUSY vertex algebra which is simple and strongly regular as a usual vertex operator superalgebra. Then the linear span of the trace functions $\tilde{\varphi}_M$ is invariant under the action of the Jacobi group, that is $\Psi_M$ is a linear combination of the $\widetilde{\varphi}_{M_i}$ for $M_i$ running through the (finite) list of irreducible $V$-modules. 
\end{thm*}
\end{nolabel}
\begin{rem}\hfill
\begin{enumerate}
\item The condition of $V$ being simple and strongly regular is a technical condition entailing rationality and $C_2$ cofiniteness, this finiteness condition is sufficient to guarantee finite dimensionality of the spaces of conformal blocks. For a definition and a survey of these vertex operator algebras we refer the reader to \cite{krauel,mason-surv}.
\item The trace functions considered in \cite{krauel} are more general than ours in that they allow any (and many) $U(1)$ currents instead of just the $U(1)$ current $J$ from the superconformal structure. Presumably these traces would appear when one considers hyperelliptic curves of higher genus. 
\item As mentioned in the introduction the theorem in \cite{krauel} is proved by carefully using a result of Miyamoto which in turn is a complicated computation. The theorem can also be proved by the geometric method of Zhu as described in Remark \ref{rem:zhu-need-for-modularity}, i.e., one uses rationality and $C_2$ cofiniteness to prove the finite dimensionality of the space of conformal blocks and that $\{\widetilde{\varphi}_M\}$ forms a basis of this space as $M$ runs through the finite set of irreducible $V$-modules.  In this approach we need to define conformal blocks for superelliptic curves as solutions to the differential equations \eqref{eq:10.5.4}.  Jacobi invariance follows simply from d) in Theorem \ref{thm:10.10}. 
\item When $b$ is a primary field  of charge $0$ and conformal weight $\Delta$ we obtain that the collection $\widetilde{\varphi}_M(b)$ are Jacobi forms of weight $\Delta$ and index $C/6$. If $b$ is not primary, then the general transformation formula is given by \eqref{eq:correction-z2}-\eqref{eq:correction-modular1}. Note that these formulas involve the automorphism $\sigma^{-1}$ defined in Lemma \ref{lem:10.7}, and in particular involve the action of $\exp(\beta J_1)$ where $\beta = 2 \pi i m$ in the case of the $\mathbb{Z}^2$ action by \eqref{eq:strangeequiva} and $\beta = - 2 \pi i \tfrac{c \alpha}{ c \tau + d}$ in the case of the modular group action. This explains the more complicated transformation rules (as quasi-modular forms) found in \cite[Thm 1.2]{krauel} as arising from the geomeric change of coordinates inducing the isomorphism of elliptic supercurves $E_{\tau, \alpha} \simeq E_{\tau', \alpha'}$ as in \ref{no:jacobisuper}.
\end{enumerate}
\label{rem:9.12}
\end{rem}
We summarize this discussion in 
\begin{thm}
Let $V$ be a conformal $N_W=1$ SUSY vertex algebra which simple and strongly regular as a usual vertex operator superalgebra and let $\{M_i\}$ be the finite set of its irreducible modules. Define $\varphi_M$ by \eqref{eq:7.1.1} and put $\widetilde{\varphi}_M = y^{C/6} \varphi_M$. Then
\begin{enumerate}

\item
$\widetilde{\varphi}_M$ gives rise to a conformal block of $\cA_{E_{\tau, \alpha}}$ over the elliptic supercurve with parameters $\tau, \alpha$ defined in \ref{no:x.2.c}. 

\item The linear span of the $\widetilde{\varphi}_M$ for $M$ running through $\{M_i\}$ is invariant under the following action of the Jacobi group: 
\begin{equation}
\begin{aligned}
{[}\widetilde{\varphi}_M \cdot (m,n)](b; \tau, \alpha)
&:= \exp \left(  \frac{C}{6} 2 \pi i (m^2 \tau + 2 m \alpha) \right) e^{ \frac{C}{3} 2 \pi i n} \times  \\ & \qquad \widetilde{\varphi}_M \left( \sigma^{-1} b; \tau, \alpha + m \tau + n \right) \qquad &&  m,n \in \mathbb{Z}^2, \\
\left[ \widetilde{\varphi}_M \cdot 
\begin{pmatrix}
a & b \\ c & d
\end{pmatrix}
\right](b, \tau, \alpha) &:= 
  \exp \left( 2 \pi i \frac{C}{6} \left( \frac{- c \alpha^2}{c \tau + d} \right)  \right) \times \\ & \quad \widetilde{\varphi}_M \left( \sigma^{-1} b, \frac{a \tau + b}{c \tau + d}, \frac{\alpha}{c \tau + d}\right), \qquad && 
\begin{pmatrix}
a & b \\ c & d
\end{pmatrix} \in SL(2,\mathbb{Z})
\end{aligned}
\end{equation}

\item These conformal blocks are flat with respect to the flat connection $\nabla$ defined by
\[
\nabla = d + \frac{1}{2 \pi i } \res_\bt \widetilde{\zeta}(t)  \Bigl( Y^M(h, \bt)d \tau  + 2 \pi i \,  \zeta\, Y^M(j, \bt) d \alpha \Bigr),
\]
that is $\nabla \widetilde{\varphi}_M(b) = 0$.

\item The connection $\nabla$ is equivariant with respect to the action of the Jacobi group from b). That is under the change of coordinates
\[
(\tau', \alpha') = (\tau, \alpha + m \tau + n)
\quad \text{and} \quad 
(t, \zeta) \mapsto (t, e^{2 \pi i m t} \zeta)
\]
we have
\begin{equation}
\begin{aligned}
& \nabla' = d + \frac{1}{2 \pi i } \res_{\bt'} \widetilde{\zeta}(t')  \Bigl( Y^M(h, \bt')d \tau' + 2 \pi i \,  \zeta'\, Y^M(j, \bt') d \alpha' \Bigr), \\
& \nabla' \left[ \widetilde{\varphi}_M \cdot 
(m,n) 
 \right] = 0, \qquad (m,n) \in \mathbb{Z}^2,
\end{aligned}
\end{equation}
and under the change of coordinates
\[
(\tau', \alpha') = \left( \tfrac{a \tau + b}{c \tau + d}, \tfrac{\alpha}{c\tau + d} \right)
\quad \text{and} \quad 
(t', \zeta') = \left( \tfrac{t}{c\tau + d}, e^{- 2 \pi i t \tfrac{c \alpha}{c \tau + d}} \zeta \right)
\]
we have
\begin{equation}
\begin{aligned}
& \nabla' = d + \frac{1}{2 \pi i } \res_{\bt'} \widetilde{\zeta}(t')  \Bigl( Y(h, \bt')d \tau' + 2 \pi i \,  \zeta'\, Y(j, \bt') d \alpha' \Bigr), \\
& \nabla' \left[ \widetilde{\varphi}_M \cdot 
\begin{pmatrix}
a & b \\ 
c & d
\end{pmatrix}
 \right] = 0, \qquad 
\begin{pmatrix}
a & b \\ 
c & d
\end{pmatrix} \in SL(2, \mathbb{Z}).
\end{aligned}
\end{equation}

\end{enumerate}
\label{thm:final}
\end{thm}
\begin{nolabel}
Theorem \ref{thm:final} can be interpreted as follows: the spaces of conformal blocks of an $N_W=1$ SUSY vertex algebra forms an  $SL(2, \mathbb{Z}) \ltimes \mathbb{Z}^2$-equivariant vector bundle with flat connection on $\mathbb{H} \times \mathbb{C}$. This can be identified with a left $\cD$-module on the orbifold quotient $\cM^{0} = \mathbb{H} \times \mathbb{C} \git SL(2, \mathbb{Z}) \ltimes \mathbb{Z}^2$ of \ref{no:newfine}. The conformal blocks $\widetilde{\varphi}_M$ associated to each irreducible module $M$ are flat sections. 
\label{no:ultimo}
\end{nolabel}

\section{Examples}\label{sec:examples}
\begin{nolabel}\label{no:examples1}
Let $\Gamma$ be an even unimodular lattice of rank $r$ (which necessarily is a multiple of $8$). We have the usual lattice vertex algebra $V_\Gamma$. Consider the purely even space $\fh = \mathbb{C} \otimes_\mathbb{Z} \Gamma$ with its symmetric bilinear pairing $(,)$ linearly extended from $\Gamma$. It generates a sub-vertex algebra of \emph{free Bosons} isomorphic to $\mathrm{Cur}(\fh)$ as in Example \ref{ex:examples-conformal-algebras} c). Consider also the purely odd space $W := \Pi (\mathbb{C} \otimes_{\mathbb{Z}} \Gamma)$ together with its super-skew-symmetric bilinear form $\langle, \rangle$ induced from the same pairing in $\Gamma$. We have the algebra of \emph{free Fermions} $F(W)$ as in \ref{ex:examples-conformal-algebras} d). We consider the SUSY lattice vertex algebra \cite{heluani6} $V$ which as a vertex algebra is simply the tensor product $V_\Gamma \otimes F(W)$.  The superconformal structure is generated as follows.
Write $W = W^+ \oplus W^-$ with $W^\pm$ isotropic subspaces for $\langle, \rangle$ (recall $W$ is a complex vector space of even dimension). Choose a basis $\left\{ \bar{\alpha
}_i \right\}_{i = 1}^{r/2}$ of $W^+$ and its dual basis $\left\{ \bar{\alpha}^i \right\}$ of $W^-$ satisfying $\langle \alpha_i, \alpha^j \rangle = \delta_{i}^j$. 
 The conformal vector of $F(W)$ is 
\[ L_W = -  \sum_{i = 1 }^{r/2}  \bar{\alpha}_i \partial \bar{\alpha}^i \]  It has central charge $c_W = -r$ and each $\bar{\alpha} \in W^+$ is primary of conformal weight $1$ while $\bar{\alpha} \in W^-$ is primary of conformal weight $0$.  

Changing the parity of the decomposition of $W$ we obtain $\fh = \fh^+ \oplus \fh^-$.  Let $\left\{ \alpha_i \right\}_{i = 1}^{r/2}$ be a basis of $\fh^+$ obtained by changing the parity of the $\left\{ \bar{\alpha}_i \right\}$  and let $\left\{ \alpha^i \right\}$ be its dual basis $(\alpha_i, \alpha^j) = \delta_i^j$ in $\fh^-$. We have the conformal vector of $\mathrm{Cur}(\fh)$ which extends to a conformal vector of $V_\Gamma$ given by \[ L_{\fh} =  \sum_{i =1 }^r \alpha_i \alpha^i. \]
It has central charge $c_{\fh} = r$ and each $\alpha \in \fh$ is primary of conformal weight $1$. In addition, for each $\gamma \in \Gamma$ we have a vertex operator $e_\gamma \in V$ primary of conformal weight $(\gamma, \gamma)/2 \in \mathbb{Z}$. We define the conformal vector of $V$ as
\[ L = L_W + L_\fh, \]
it has central charge $0$. Now $L$ together with
\[ 
J = \sum_{i = 1}^{r/2} \bar{\alpha}_i \bar{\alpha}^i, \qquad Q = \sum_{i = 1}^{r /2} \alpha_i \bar{\alpha}_i, \qquad \text{and} \qquad H = \sum_{i =1}^{r/2} \alpha^i \bar{\alpha}^i. \]
satisfy the relations \eqref{eq:topological} of the topological algebra with the central parameter $C=3r/2$. With respect to the $U(1)$ current $J$, the vectors in $W^\pm$ have charge $\pm 1$ while the vectors in $\fh$ and the vertex operators $e_\gamma$ have charge $0$. The $N_W=1$ SUSY vertex algebra structure on $V$ is simply given by defining $D = Q_{-1}$ and the superfields 
\[ Y(a,x,\theta) = Y(a,x) + \theta Y(Q_{-1} a, x). \]
We easily compute now the character of this vertex algebra to be
\begin{align}
\chi_V(q, y)
&:= y^{C/6}\str_V q^{L_0} y^{J_0} \nonumber \\
&= y^{r/4}\left[ \prod_{n \geq 1} \frac{(1 - y q^{n})(1 - y^{-1} q^{n-1})}{(1-q^{n})^2} \right]^{r/2} \sum_{\gamma \in \Gamma} q^{(\gamma,\gamma)/2}. \label{eq:character1}
\end{align}
The denominator comes from the Bosons in $\fh$. The first term in the numerator comes from the positively charged Fermions in $W^+$ and the second term from the negatively charged Fermions in $W^-$. Finally, we recognize in the last factor the $\Theta(\tau)$ function of the lattice, which comes from the zero charged vertex operators $e_\gamma$. Using the Jacobi triple product identity
\[ \prod_{n\geq 1} (1-q^n)(1 - y q^{n})(1 - y^{-1}q^{n-1}) = \sum_{k \in \mathbb{Z}} (-y)^k q^{k^2/2 + k/2},\]
we can express our character as  
\begin{align}
\chi_V(q,y)
&= y^{C/6} \eta(\tau)^{-3r/2} \left[ \sum_{k \in \mathbb{Z}} (-1)^k y^k q^{(k+1/2)^2/2} \right]^{r/2} \Theta_\Gamma(\tau) \nonumber \\
&= y^{r/4} \eta(q)^{-C} \Theta_\Gamma(\tau) \left[ \sqrt{-1} y^{-1/2} \theta(\tau,\alpha)  \right]^{r/2}
= \eta(\tau)^{-C} \Theta_\Gamma(\tau) \theta(\tau,\alpha)^{C/3}. \label{eq:character2}
\end{align}
Here $\eta(\tau) = q^{-1/24} \prod_{n \geq 1} (1-q^n)$ is Dedekind's function which is a modular form of weight $1/2$, and the lattice theta function $\Theta_\Gamma(\tau)$ is a modular form of weight $r/2 = C/3$. The theta function $\theta(\tau, \alpha)$, which is given by \cite[V.1.1]{chandra}
\begin{equation}\label{eq:jacobi-theta-def} \theta(\tau, \alpha) = \frac{1}{\sqrt{-1}} \sum_{k \in \mathbb{Z}} (-1)^k q^{\frac{(k+1/2)^2}{2}} y^{k + 1/2}, \end{equation}
is a Jacobi form with character of weight $1/2$ and index $1/2$. We used in \eqref{eq:character2} that $8|r$ in order to cancel the terms with $\sqrt{-1}$ that appear to a power multiple of $4$. It follows that our character is a Jacobi form of weight $-C/2 + C/3 + C/6 = 0$ and index $C/6$. Note that in this case $C/6 = r/4 \in 2  \mathbb{Z}$.  
\end{nolabel}
\begin{nolabel}\label{no:chiralderham}
Let $M$ be a compact smooth Calabi-Yau manifold of complex dimension $n$. We have a sheaf $\Omega^{\mathrm{ch}}_M$ of vertex operator algebras on $M$ called the \emph{chiral de Rham complex} of $M$ \cite{malikov}. It carries a topological structure with parameter $C = 3n$. Its sheaf cohomology then $V = H^* (M, \Omega^{\mathrm{ch}}_M)$ is a super-vertex algebra with a superconformal structure as in \ref{defn:topological-quasi-conformal}.   Borisov and Libgober proved in \cite{borisov-libgober} that the graded character of $V$ coincides with the \emph{2-variable elliptic genus} of $M$ which is known to be a Jacobi form of weight $0$ and index $C/6$.  It would be interesting to know in which cases (if any) is $V$ itself a holomorphic vertex algebra.  
\end{nolabel}

\appendix
\section{Convergence of $1$-point functions}

In this appendix we prove the convergence of the one point functions on an elliptic supercurve under some finiteness condition on $V$. Our approach follows Zhu's proof in the usual case which roughly goes as follows. One replaces the base field of the complex numbers by the ring $QM_* \cong \mathbb{C}[b_0, b_1, b_2]$ of (quasi)-modular forms, viewed as a subring of convergent formal power series in $\mathbb{C}[[q]]$. One uses a finiteness condition on the vertex operator algebra $V$ (the $C_2$-cofiniteness condition) to show that certain quotient of $V \otimes QM_*$ is a finite $QM_*$-module. The Noetherian property of $QM_*$ implies that the sequence of vectors $L[-2]^i a$ for any $a \in V$ spans a finite dimensional $QM_*$-submodule of this quotient and this shows that conformal blocks satisfy a finite order ODE with coefficients in the ring $QM_*$. The convergence result follows from standard results in the theory of differential equations.

In our situation, we need to replace modular forms (as functions of $q$) by Jacobi modular forms (functions of $q$ and $y$). A technical difficulty arises in that the ring of Jacobi forms is not finitely generated, therefore it is not Noetherian. This forces us to use a larger (Noetherian) ring of forms $R$, containing \emph{weak Jacobi forms} and \emph{quasi-Jacobi forms}. We then show that the $\varphi_M$ descend to a certain explicit quotient of $V \otimes R$. Using a certain finiteness condition (which we call \emph{charge cofiniteness}, but which turns out to be equivalent to $C_2$-cofiniteness) we show the quotient to be a finite $R$-module. Convergence of $\varphi_M$ is established by repeating Zhu's proof.

\begin{defn}
Recall the subspace $C_2(V) = \text{span}\{a_{-2}b | a, b \in V\}$ of the vertex algebra $V$ introduced by Zhu. If $C_2(V)$ has finite codimension in $V$ then $V$ is said to be $C_2$-cofinite. Now let $V$ be a conformal $N_W=1$ SUSY vertex algebra with grading $V = \oplus_{n\in \mathbb{Z}} V_n$ by charge. We define the subspace $C_\mathrm{ch}(V) \subset V$ to be the span of
\[
\left\{ a_{(-n)} b | n = \max (2, |\mathrm{charge}\, a|) \right\} = \left\{ a_{(-n)}b | n \geq \max (2, |\mathrm{charge}\, a|) \right\},
\]
and we say that $V$ is \emph{charge cofinite} if $C_{\mathrm{ch}}(V)$ has finite codimension in $V$.
\label{defn:c2-cofinite-super}
\end{defn}

\begin{lem}
Let $V$ be an $N_W=1$ SUSY vertex algebra. Then $V$ is charge cofinite if and only if it is $C_2$-cofinite.  
\label{lem:a.1}
\end{lem}
\begin{proof}
It is clear that charge cofiniteness implies $C_2$-cofiniteness. Suppose now that $V$ is $C_2$ cofinite and fix $\{v^1, \ldots, v^k\}$ spanning $V$ modulo $C_2(V)$. Put $N = \max( \{2\} \cup \{|\text{charge}(v^i)|\}_{i} )$ and $D \subset V$ the span of the vectors $v^{i_1}_{-n_1} \cdots v^{i_s}_{-n_s} \vac$ where $n_1 > n_2 > \cdots > n_s > 0$ and $n_1 > N$. By Proposition \ref{GabNieResult} $V / D$ is finite dimensional. On the other hand $D \subset C_\text{ch}(V)$, hence $V / C_\text{ch}(V)$ is finite dimensional.
\end{proof}

\begin{prop}[\cite{gaberdiel-neitzke} Proposition 8, \cite{miyamotoC2} Lemma 2.4]\label{GabNieResult}
Let $V$ be a $C_2$-cofinite vertex algebra graded by non negative conformal weight. Let $\{v^1, \ldots, v^k\}$ span $V$ modulo $C_2(V)$. Then $V$ is spanned by the vectors $v^{i_1}_{-n_1} \cdots v^{i_s}_{-n_s} \vac$ where $n_1 > n_2 > \cdots > n_s > 0$.
\end{prop}



\begin{nolabel}
Zhu's theorem on modular invariance of trace functions fundamentally involves the identification of the chiral algebras on $\mathbb{A}^1$ and on $\mathbb{G}_m$ via the exponential map. This is manifest in the \emph{two} vertex algebra structures $Y(a, z)$ and $Y[a, z]$ \cite{zhu} on $V$ that appear in the theorem statement. The $N_W=1$ situation is similar. Let $V$ be a conformal $N_W=1$ SUSY vertex algebra with corresponding chiral algebras $\cA$ on $\mathbb{A}^{1|1}$ and $\cA_{\mathbb{G}_m^{1|1}}$ on $\mathbb{G}_m^{1|1}$. Recall (c.f., \eqref{eq:exponential}) the exponential map
\[
\rho(\bt) =\rho(t,\zeta) = (x-1, \theta) = (e^{2 \pi i t} - 1, e^{2\pi i t}\zeta).
\]
\begin{lem*} Let $(V, \vac, Y(\cdot, \bt), j,h)$ be a conformal $N_W=1$ SUSY vertex algebra with conformal vectors $h$ and $j$. Let $a \in V$ be a vector of conformal weight $\Delta$. Define 
\begin{multline}
Y[a,\bt]
= Y[a,t,\zeta]
= \rho^{-1} Y(\rho \, a, t,\zeta) \rho \\
= e^{2 \pi i t \Delta} Y\left(a,e^{2 \pi i t} - 1, e^{2 \pi i t}\zeta \right) + e^{2 \pi i t \Delta} \zeta Y \left( Q_0 a, e^{2 \pi i t} -1, e^{2 \pi i t} \zeta \right).
\end{multline}
and 
\[
\widetilde{h} = \frac{1}{(2 \pi i)^2} h, \qquad \widetilde{j} = \frac{1}{2 \pi i} j - \frac{C}{6}.
\]
Then $(V, \vac, Y[\cdot, \bt], \tilde{j}, \tilde{h})$ is a conformal $N_W=1$ SUSY vertex algebra. 
\end{lem*}
\begin{proof}
The vertex algebra $(V, \vac, Y[\cdot, \bt])$ is obtained from $(V, \vac, Y(\cdot, \bt))$ by transport of structure from the linear map $\rho \in \Aut(V)$. Putting $\bt = (0,0)$ in \eqref{eq:10.5.1} shows that $\rho\, j = \widetilde{j}$ and $\rho\, h = \widetilde{h}$. Hence $\tilde{h}$ and $\tilde{j}$ are the conformal vectors of $(V, \vac, Y[\cdot, \bt])$.
\end{proof}
\label{no:a.1}
\end{nolabel}
\begin{nolabel}
The vertex algebra structure $(V, \vac, Y[a,\bt])$ is naturally associated to the logarithmic coordinates $\bt$ of $\mathbb{A}^{1|1}$, and the original vertex algebra structure $(V, \vac, Y(a,\bx))$ to the exponentiated coordinates $\bx$ of $\mathbb{G}_m^{1|1}$.


We expand the superconformal generators \eqref{eq:expansion-usual} as 
\[
Y[\tilde{h},\bt] = H[t] + \zeta \bigl(L[t] + \partial_t J[t]\bigr),
\qquad
Y[\tilde{j},\bt] = J[t] - \zeta Q[t],
\]
where 
\[ 
\begin{aligned}
L[t] &= \sum_{n \in \mathbb{Z}} L_{[n]} t^{-2-n}, &&& J[t] &= \sum_{n \in \mathbb{Z}} J_{[n]} t^{-1 -n}, \\
Q[t] &= \sum_{n \in \mathbb{Z}} Q_{[n]} t^{-2-n}, &&& H[t] &= \sum_{n in \mathbb{Z}} H_{[n]} t^{-1-n}.
\end{aligned}
\]
These fields satisfy the commutation relations \eqref{eq:topological-commutators}. In general we use parentheses together with square brackets to denote Fourier modes of the superfield $Y[a,\bt]$, i.e.,
\[
Y[a,\bt] = \sum_{n \in \mathbb{Z}} t^{-1-n} \bigl( a_{([n])}  + \zeta (Q_{[-1]} a)_{([n])} \bigr).
\]
Finally, for a $V$-module $M$ and a vector $a \in V$ we will denote by $o(a) \in \End(M)$ the Fourier mode of $a$ preserving conformal weights of $M$, that is $a_{(\Delta - 1)}$ when the conformal weight of $a$ is $\Delta$. 
\end{nolabel}
\begin{nolabel}
The exponential change of coordinates is upper triangular: $J_{0} = J_{[0]} + \sum_{k > 0} l_k J_{[k]}$ for some constants $l_k$ and $J_{[0]} = J_{0} + \sum_{k > 0} m_k J_k$ for constants $m_k$. Similarly for each $a \in V$ we have $a_{(-2)} = a_{([-2])} + \sum_{k\geq -1} n_k a_{([k])}$ and $a_{([-2])} = a_{(-2)} + \sum_{k \geq -1}  p_k a_{(k)}$ for constants $n_k$ and $p_k$. In particular $(V, \vac, Y(\cdot, \bx))$ is charge cofinite if and only if $(V, \vac, Y[\cdot, \bt])$ is charge cofinite. 
\label{no:cofiniteness-two-vertex-structures}
\end{nolabel}
\begin{nolabel}
Recall from \ref{no:x.2.c} the family $X/S$ and its quotient $E^0/S$ by $\Z^2$. We had also the corresponding chiral algebras $\cA$ and $\cA_E$ associated to $V$. Let $\bo \in E$ denote the image of the origin $(0,0) \in \mathbb{A}^{1|1}$. Let $\varphi \in C(E, \bo; \cA_{E})$ be a conformal block. We identify sections of $\cA_{E}$ with $\mathbb{Z}^2$-invariant sections of $\cA$, and we use the coordinates $\bt = (t,\zeta)$ and the trivialization $[d\bt]$ of $\omega_X$ to identify sections of $\cA$ with functions $a(\bt) : \mathbb{A}^{1|1} \rightarrow V$. If $a \in V$ is of charge $-1$ with respect to $J_{[0]}$ then we have a section $\wp(t) a [d\bt] \in H^0(E \setminus \bo, \cA_E)$. Similarly if $a$ has charge $0$ then we have $\zeta \wp(t) a [d\bt] \in H^0(E\setminus \bo, \cA_E)$. Thus by definition \eqref{eq:conformal-definition} we have
\begin{gather*} \varphi \Bigl( \res_{\bt} \wp(t) Y[a,\bt] b \Bigr) = 0 \qquad b \in V,\quad J_{[0]} a = -a \\ 
 \varphi \Bigl( \res_{\bt} \zeta \wp(t) Y[a,\bt] b\Bigr) = 0, \qquad b \in V, \quad J_{[0]}a =0. 
\end{gather*}
Using the expansion of $\wp(t)$ near $t=0$ from \ref{no:weierzeta} we obtain 
\begin{gather}
\varphi \Bigl( (Q_{[-1]}a)_{([-2])} b + \sum_{k \geq 1}  b_{k} (Q_{[-1]}a)_{([2k])} b \Bigr) = 0, \qquad J_{[0]}a = -a \\
\varphi \Bigl( a_{([-2])} b + \sum_{k \geq 1} b_k a_{([2k])} b \Bigr) = 0, \qquad J_{[0]} a = 0.
\label{eq:a.2.1}
\end{gather}
\label{no:a.3}
\end{nolabel}
\begin{nolabel}\label{no:constants}
If $a \in V$ is of charge $0$ then we have the constant section $\zeta a [d\bt] \in H^0(E, \cA_E)$. It follows that
\begin{equation} \label{eq:a.4.2}
\varphi \Bigl( a_{([0])} b \Bigr) = 0, \qquad b \in V, \quad J_{[0]}a = 0.
\end{equation}
In particular, if $J_{[0]}b = m b$ for some $m \neq 0$ then we may substitute $a = j$ and see that $\varphi(b) = 0$. Hence conformal blocks can be thought as linear functionals on the neutral subspace $V_0 \subset V$.  

If $a \in V$ is of charge $-1$ then we have the constant section $a [d\bt] \in H^0(E, \cA_E)$. It follows that
\[
\varphi \Bigl( (Q_{[-1]}a)_{([0])} b \Bigr) = 0, \qquad b \in V, \quad J_{[0]}a = -a.
\]
Substituting $a = h$ into this equation yields $\varphi (L_{[-1]} b) = 0$, i.e., $\varphi$ vanishes on derivatives. We could also have obtained this from \eqref{eq:a.4.2} with $a = Q_{[-1]}h$.
\end{nolabel}
\begin{nolabel}
To construct global sections for elements $a$ with charge different than $0,-1$ we need to consider Jacobi forms of non-vanishing index.
\begin{lem*}\label{lem:A.1.1}
Let $\phi = \in J_{k,m}$ (or even  $J'_{k,m}$). The meromorphic function 
\[
\varphi(t, \tau, \alpha) = \frac{\phi (\tau, t+\alpha)}{\phi(\tau,t)}
\]
satisfies
\[
\varphi(t+ \lambda \tau + \mu, \tau, \alpha) = e^{4 \pi i m \lambda \alpha } \varphi(t, \tau, \alpha).
\]
\end{lem*}
An element $a \in V$ of charge $2m > 0$ (resp. $2m -1 >  0$), gives rise to a well defined meromorphic section $\zeta \varphi (t) a [d\bt]$ (resp. $\varphi (t) a [d\bt]$) of $\cA_E$. The poles of these sections are located on the zeroes of the function $\phi(\tau, t)$. For vectors of negative charge one may use $\varphi(t, \tau, -\alpha)$ in place of $\varphi(t, \tau, \alpha)$ to construct meromorphic sections of $\cA_E$. 

The first cusp forms of index $1$ are given by \cite[$\S3$]{zagier} 
\[
\phi_{10,1} = \frac{1}{144} (E_6 E_{4,1} - E_4 E_{6,1}),
\qquad
\phi_{12,1} = \frac{1}{144} \left( E_4^2 E_{4,1} - E_6 E_{6,1} \right),
\]
where $E_{k,m}$ are the Jacobi forms \eqref{eq:zagier-eisenstein} and $E_k$ are the usual Eisenstein series modular forms. Division by the discriminant cusp form $\Delta = (2 \pi)^{12} \eta^{24} = q \prod_{n \geq 1} (1-q^n)^{24}$ gives rise to two weak Jacobi forms $\phi_{-2,1} = \phi_{10,1}/\Delta$ and $\phi_{0,1} = \phi_{12,1}/\Delta$. Let $\varphi$ be the function associated by Lemma \ref{lem:A.1.1} to $\phi_{-2,1}$ (clearly it is the same as the function associated to $\phi_{10,1}$). Since $\phi_{10,1}(\tau,\alpha)$ has a zero of order $2$ at the points $\alpha \in \mathbb{Z} + \mathbb{Z} \tau$, its Taylor expansion in $\alpha$ is
\[
\phi_{10,1}(\tau,\alpha) = (2 \pi i)^2 \Delta(\tau) \alpha^2 + O(\alpha^4).
\]
It follows that the Taylor expansion of $\varphi(t,\tau,\alpha)$ near $t = 0$ is given by
\[
\varphi (t,\tau, \alpha) = (2 \pi i t)^{-2} \frac{\phi_{10,1}(\tau,\alpha)}{\Delta} + O(t^0)
= (2 \pi i t)^{-2} \phi_{-2,1}(\tau, \alpha) + O(t^0)
\]
\label{no:other-charges-sections}
\end{nolabel}
\begin{rem} \hfill
\begin{enumerate}
\item 
The form $\phi_{-2,1} \in J'_{-2,1}$ has all its zeros located in $\alpha = \mathbb{Z} + \mathbb{Z}\tau$. Therefore the associated section $\varphi(t,\tau,\alpha) a [d\bt]$ of $\cA_E$ (where $a \in V$ is of charge $1$) is singular only at $\bo \in E$. More generally if $\varphi$ is the function associated by Lemma \ref{lem:A.1.1} to $\phi \in J'_{-2m,m}$, and $a \in V$ is of charge $2m-1$ (resp. charge $2m$) then we have the section
\begin{align*}
\varphi(\tau, \alpha, t) a [d\bt] \in H^0(E \setminus \bo, \cA)
\quad \text{resp.} \quad \zeta \varphi(\tau, \alpha, t) a [d\bt] \in H^0(E \setminus \bo, \cA).
\end{align*}
\item 
Let $M_* \cong \mathbb{C}[b_1, b_2]$ denote the ring of modular forms. The ring $J'_{*,*}$ of quasi Jacobi forms is a polynomial ring over $M_*$ with two generators $\phi_{-2,1}$ and $\phi_{0,1}$. It follows that the space $J_{-2m,m}$ is one dimensional and is spanned by $\phi_{-2,1}^m$. 
\end{enumerate}
\label{rem:singularities-charged-secitons}
\end{rem}
\begin{nolabel}
To deal with vectors of odd charge it will be more convenient for us to use weak Jacobi forms of half-integral index. First let us consider the index $1/2$ Jacobi form with trivial character $\phi_{-1,1/2} \in J'_{-1,1/2}$ defined by
\[
\phi_{-1,1/2}(\tau, \alpha) = 2 \pi \frac{\theta(\tau, \alpha)}{\eta(\tau)^3} = \frac{\theta(\tau, \alpha)}{\theta'(\tau,0)},
\]
where $\eta(\tau)$ is Dedekind's function and $\theta$ is Jacobi's theta function defined in \eqref{eq:jacobi-theta-def}. It has a Fourier expansion in positive integral powers of $q$ and semi-integer powers of $y$, as such we may view it as an element of $\mathbb{C}[ [q]]( (y^{1/2}))$. It has a simple zero at $\alpha \in \mathbb{Z} + \mathbb{Z}\tau$ and no other zeros. The corresponding function $\varphi(t, \tau, \alpha)$ given by Lemma \ref{lem:A.1.1} may be used to produce sections of $\cA_E$ associated to vectors of charge $1$ or $0$. Since $\phi_{-1,1/2}^m \in J'_{-m,m/2}$ we see that $J'_{*,*} \subset J'_{*,*}[\phi_{-1,1/2}]$ is a quadratic extension.
\label{no:half-integral-index}
\end{nolabel}
\begin{nolabel}
The function $\varphi$ associated by Lemma \ref{lem:A.1.1} to $\phi_{-1,1/2}$ was studied by many authors: it was proposed as a generating series for the \emph{elliptic genus} of Calabi-Yau manifolds in \cite{yamada}, it was studied by Zagier in \cite{zagier-invent} in the context of periods of modular forms and recently by Gritsenko in \cite{gritsenko} and Libgober in his study of \emph{quasi-Jacobi} forms \cite{libgober2009elliptic} which we now recall. 

We need the real analytic functions
\[ \lambda(\tau, \alpha) = \frac{\alpha - \overline{\alpha}}{\tau - \overline{\tau}}, \qquad \mu(\tau) = \frac{1}{\tau - \overline{\tau}}. \]
\begin{defn*}\cite[2.2-2.3]{libgober2009elliptic} \hfill
\begin{enumerate}
\item 
An \emph{almost meromorphic Jacobi} form of weight $k$, index zero and depth $(s,t)$ is a (real) meromorphic function in $\mathbb{C}\left\{q,y  \right\}[y^{-1},\lambda,\mu]$ which transforms as a Jacobi form of weight $k$ and index zero and has degree at most $s$ in $\lambda$ and at most $t$ in $\mu$. 
\item A \emph{quasi-Jacobi} form is a constant term of an almost meromorphic Jacobi form of index zero, where the latter is considered as a polynomial in $\lambda$, $\mu$. 
\end{enumerate}
\end{defn*}
Let $\wp_1(\tau,\alpha) = \tilde{\zeta}(\tau, \alpha)$ (see \ref{no:weierzeta}), $\wp_2(\tau,\alpha) = \wp(\tau, \alpha) + b_0(\tau)$, and for $k \geq 3$ let $\wp_k$ be defined by the absolutely convergent series
\[
\wp_k(\tau, \alpha) = \sum_{(m,n) \in \mathbb{Z}^2} \frac{1}{(\alpha + m \tau + n)^k}.
\]
Let us also write
\[
e_n = \lim_{\alpha \rightarrow 0} \left( \wp_n(\tau,\alpha) - \frac{1}{\alpha^n} \right) = \sum_{(\lambda,\mu) \in \mathbb{Z}^2\setminus (0,0)} \frac{1}{(\lambda \tau + \mu)^n}.
\]
As already mentioned, the ring $M_*$ of modular forms is isomorphic to $\mathbb{C}[e_4, e_6]$, while the ring $QM_*$ of quasi-modular forms (generated by definition by $e_2, e_4, \ldots$) is isomorphic to $\mathbb{C}[e_2, e_4, e_6]$. Similarly \cite[Prop. 2.8 \& 2.9]{libgober2009elliptic} the ring $J_{*,0}$ of Jacobi forms of index $0$ and weight at least $2$ is generated by $\wp_2 - e_2, \wp_3, \wp_4, \dots$, and is in fact a polynomial ring with generators $\wp_2-e_2$, $\wp_3$, and $\wp_4$. The ring $QJ_*$ of quasi-Jacobi forms is generated by $\wp_k$, $k \geq 1$ and $e_2$, it is a polynomial ring over $J_{*,0}$ with generators $\wp_1$ and $e_2$. In particular it is a Noetherian ring.  

The most important result about quasi-Jacobi relevant to our purposes is the following 
\begin{prop*}\cite[2.10]{libgober2009elliptic}
Consider the Taylor series expansion
\[
\frac{\theta(\tau,t + \alpha)}{\theta(\tau,t)} = \frac{\theta(\tau,\alpha)}{\theta'(\tau,0)} \left( \frac{1}{t} + \frac{1}{\alpha} + \sum_{i \geq 1} F_i(\tau, \alpha) t^i \right).
\]
The algebra of functions generated by the coefficients $\left\{ F_i \right\}_{i \geq 1}$ is the algebra of quasi-Jacobi forms. 
\end{prop*}
Recall that this function is the function $\varphi$ associated by Lemma \ref{lem:A.1.1} to $\phi_{-1, 1/2}$.
\label{no:libgober-1}
\end{nolabel}
\begin{lem}
Let $\phi \in J'_{-2m,m}$, then $\tfrac{\partial}{\partial \alpha} \log \phi(\tau, \alpha)$ is a quasi-Jacobi form of weight $1$. It has a single pole of order $1$ at $\alpha \in \mathbb{Z} + \mathbb{Z}\tau$ and no other singularities.  It follows that it is a multiple of $\wp_1$. 
\label{lem:quasi-log}
\end{lem}
\begin{proof}
From \cite[Thm 9.5]{zagier-jacobi} we can reduce the statement to $m=1/2$ which is (34) in \cite{libgober2009elliptic}. In fact, we may use the fact that $J'_{-2m,m}$ is one dimensional and spanned by $\phi_{-2,1}^m$ (or $\phi_{-1,1/2}^{2m}$).  
\end{proof}
\begin{lem}
Let $\phi \in J'_{-2m,m}$ then in its Taylor expansion \[ \phi(\tau,\alpha) = \alpha^{2m} \sum_{i \geq 0} \psi_i(\tau) \alpha^i,\] the first coefficient $\psi_0$ is modular invariant and $\psi_1$ transforms as a modular form of weight $1$ therefore $\psi_0$ is constant and $\psi_1 = 0$. 
\label{lem:modular-taylor-1}
\end{lem}
We will actually need the following generalization of the ``easy part'' of Libgober's Proposition:
\begin{prop}
Let $\phi \in J'_{-2m,m}$ and consider the Taylor expansion:
\begin{equation}
\frac{\phi(\tau,t+\alpha)}{\phi(\tau,t)} = \frac{\phi(\tau,\alpha)}{\phi^{(2m)}(\tau,0)} \frac{1}{t^{2m}} \left( 1 + \sum_{i \geq 1} F_i t^i \right)
\label{eq:a.13.1}
\end{equation}
then the coefficients $F_i$ are quasi-Jacobi forms. 
\label{prop:a.13}
\end{prop}
\begin{proof}
Consider the function 
\[
\Phi(t,\tau,\alpha) = \frac{t^{2m} \phi(\tau,t+\alpha) \phi^{(2m)}(\tau,0)}{\phi(\tau,t) \phi(\tau,\alpha)}.\]
It satisfies the following transformation properties:
\begin{align*}
\Phi \left( t, \tau, \alpha + \lambda \tau + \mu \right)
&= e^{-4 \pi m \lambda t } \Phi(t,\tau,\alpha), \\
\Phi \left( \frac{t}{c\tau +d}, \frac{a \tau + b}{ c\tau+d}, \frac{\alpha}{c \tau + d} \right)
&= e^{4 \pi i m  \frac{c t \alpha}{c \tau + d} } \Phi(t,\tau,\alpha). 
\end{align*}
It follows that for each $i \geq 0$, the function $H_i(\tau,\alpha)$ defined by the expansion
\[
\frac{d^2 \log \Phi}{d t^2} = \sum_{i \geq 0} t^i H_i,
\]
is a Jacobi forms of weight $i$ and index $0$. Moreover, the Taylor coefficients $F_i$ of $\Phi$ are polynomials in $H_i$ and $F_1$. It remains to show therefore that $F_1$ is a quasi-Jacobi form. Now $\phi$ has a Taylor expansion of the form
\[
\phi(\tau, t+ \alpha) = \phi(\tau, \alpha) + \phi^{(1)} (\tau, \alpha) t + \dots,
\]
on the other hand we have (using $\phi \in J_{-2m,m}$) that
\[
\phi(\tau, t) = \frac{t^{2m}}{(2m)!} \phi^{(2m)}(\tau, 0) \left( 1  + \frac{\phi^{(2m + 1)}(\tau,0)}{\phi^{(2m)}(\tau,0)} \frac{t}{2m +1} + \dots. \right).
\]
It follows that 
\begin{align*}
\frac{\phi(\tau, t+\alpha)}{\phi(\tau,t)}
= {} & \frac{(2m)!}{\phi^{(2m)}(\tau,0)}t^{-2m}\Bigl( \phi(\tau, \alpha) + \phi^{(1)}(\tau, \alpha) t + \dots \Bigr) \times \\
& \sum_{k \geq 0} (-1)^k \left(\frac{\phi^{(2m + 1)}(\tau,0)}{\phi^{(2m)}(\tau,0)} \frac{t}{2m +1} + \dots. \right)^k \\
= {} & \frac{(2m)!\phi(\tau, \alpha)}{\phi^{(2m)}(\tau,0)}t^{-2m} \left[ 1 + \left(\frac{\phi^{(1)}(\tau, \alpha)}{\phi(\tau,\alpha)} -  \frac{\phi^{(2m + 1)}(\tau,0)}{\phi^{(2m)}(\tau,0)} \frac{1}{2m +1}  \right) t + \dots \right].
\end{align*}
From here we get
\[
\frac{F_1}{(2m)!} = \frac{\phi^{(1)}(\tau, \alpha)}{\phi(\tau,\alpha)} -  \frac{\phi^{(2m + 1)}(\tau,0)}{\phi^{(2m)}(\tau,0)} \frac{1}{2m +1},
\]
and the proposition follows from lemmas \ref{lem:quasi-log} and \ref{lem:modular-taylor-1}.
\end{proof}
\begin{nolabel}\label{no:higher-charges}
Let $\phi = \phi_{-2m,m} \in J'_{-2m,m}$ and $a \in V$ of charge $2m \geq 0$. From these we produce a section of $\cA_E$ as explained in Remark \ref{rem:singularities-charged-secitons}. For a conformal block $\varphi \in C(E, \bo; \cA_E)$ we have
\[
\varphi \left( \res_\bt \zeta \frac{\phi(\tau, t+\alpha)}{\phi(\tau,\alpha)} Y[a,\bt] b \right) = 0.
\]
We expand using the Taylor expansion \eqref{eq:a.13.1} and we divide by $\phi(\tau, \alpha) / \phi^{(2m)}(\tau,0)$ (which we can do because the fraction is not identically zero) to obtain
\begin{equation}\label{eq:higher-charge-kernel}
 \varphi \left( a_{([-2m])}b  + \sum_{i \geq 1}  F_i^{(m)} a_{([-2m+i])}b \right)  = 0.
\end{equation}
For $a \in V$ of negative charges we use the function $\phi(\tau, -\alpha)$.

Let us denote by $R = \mathbb{C}[\wp_1,\wp_2,\wp_3,\wp_4,e_2,e_4,e_6]$ the ring of quasi-Jacobi forms and by $F_i^{(m)} \in R$ the functions appearing in the expansion \eqref{eq:a.13.1}. Let $V$ be an $N_W=1$ SUSY vertex algebra with charge decomposition $V = \oplus_{n \in \mathbb{Z}} V_n$ relative to $J_{[0]}$. Put $V_R := V \otimes R \subset V[[q]]((y^{1/2}))$ and let $O_q(V) \subset V_R$ be the $R$-submodule of $V_R$ generated by all elements of the form 
\begin{equation} \label{eq:acabodeinventar}
\begin{cases} a_{([-2])} b + \sum_{k \geq 1} b_k(q) a_{([2k])} b,\qquad  \qquad a \in V_0, \quad b \in V & m = 0 \\
a_{([-2m])}b  + \sum_{i \geq 1}  F_i^{(m)} a_{([-2m+i])}b, \qquad a \in V_{2m}, \quad b \in V & m >  0  \\ 
a_{([2m])}b  + \sum_{i \geq 1}  F_i^{(m)} a_{([2m+i])}b, \qquad a \in V_{2m}, \quad b \in V & m <  0  \\ 
\end{cases}
\end{equation}
\begin{lem*}
Let $V$ be a charge cofinite vertex algebra, the quotient $V_R/ O_q(V)$ is a finite $R$-module.
\end{lem*}
\begin{proof}
The proof of this Lemma is exactly as in \cite[4.4.1]{zhu} by induction in the conformal weight of $a$ noting that every vector of high enough conformal weight belongs to $C_{ch}(V)$ by the cofiniteness condition and that the terms in the sums of \eqref{eq:acabodeinventar} are of conformal weight lower than the leading term. 
\end{proof}
\end{nolabel}
\begin{nolabel}
Since $V_R / O_q(V)$ is a finite $R$-module and $R$ is Noetherian, following  \cite[4.4.2]{zhu} and using the Noetherian condition on the sequences $\left\{ (L_{[-2]}+J_{[-2]})^i a \right\}$ and $\left\{ (J_{[-1]})^i a \right\}$ we obtain
\begin{lem*}
For each $a \in V$ there exist $s,t \in \mathbb{Z}_+$ and $g_i(q, y),h_i(q,y) \in R \subset \mathbb{C}[ [q]]( (y^{1/2}))$ such that 
\begin{subequations}
\begin{gather}
(L_{[-2]}+J_{[-2]})^s a + \sum_{i = 0}^{s-1} g_i(q,y) (L_{[-2]}+J_{[-2]})^i a \in O_q(V) \\ 
(J_{[-1]})^t a + \sum_{i = 0}^{t-1} h_i(q,y) (J_{[-1]})^i a \in O_q(V)
\label{eq:pde-1}
\end{gather}
\label{eq:pde-subeq}
\end{subequations}
\end{lem*}
\label{no:zhu-pde}
\end{nolabel}
\begin{nolabel}
Let $V$ be a conformal SUSY vertex algebra and $M$ positive energy $V$-module. Define $\Theta_M$ as the formal series in $\bq$ with values in $V^*$ given by \eqref{eq:varphi_def_1}
\begin{equation}
\Theta_M(a) =  \str_M  Y\Bigl( x^{\Delta_a} a + x^{\Delta_a -1} \theta Q_0 a,\bx \Bigr) \bq^{-1}_M = \str_M o(a) \bq^{-1}_M.
\label{eq:denuevo}
\end{equation}
In the proof of Proposition \ref{no:prop1} and in Remark \ref{rem:bdrules} we used the convergence of the one-point function $\Theta_M$. In this appendix we are working in the formal setting, hence we cannot assume that $\Theta_M$ is a convergent conformal block. However, the proof in Remark \ref{rem:bdrules} can be straightforwardly adapted to prove the following 
\begin{lem*}
The $R$-linear extension of $\Theta_M$ to $V_R$ (which we also denote $\Theta_M$) satisfies, for all $a \in O_q(V)$, the relation $\Theta_M(a) = 0$.
\end{lem*}
Viewing Weierstrass $\zeta$ function as a formal power series we obtain the same differential equations \eqref{eq:9.13.1} and their expressions in logarithmic coordinates \eqref{eq:8.log.subs}. We Taylor expand the zeta function in \eqref{eq:8.log.subs} to obtain the following proposition. Note that modes of $Y[-, t]$ (and not $Y(-, t)$) appear as discussed in Remark \ref{rem:zhu-need-for-modularity}.
\begin{prop*} For any $a \in V$ we have
\begin{align} \label{eq:a.diff.taylor}
\begin{split}
q \frac{\partial}{\partial q} \str_M o(a) q^{L_0} y^{J_0}
= {} & \frac{1}{(2 \pi i)^2} \str_M o\left( \bigl( L_{[-2]} + J_{[-2]} \bigr) a \right)q^{L_0} y^{J_0} \\
&- \sum_{i \geq 0}  \frac{b_i}{2i+1} \str_M o\left( \bigl(L_{[2i+1]} - (2i + 1) J_{[2i+1]} \bigr) a  \right) q^{L_0} y^{J_0}
\end{split} \\
\begin{split} \label{eq:a.10.b}
\left( y \frac{\partial}{\partial y} + \frac{C}{6} \right) \str_M o(a) q^{L_0} y^{J_0}
= {} & \frac{1}{2 \pi i} \str_M o\left( J_{[-1]}a \right) q^{L_0} y^{J_0} \\
&- \frac{1}{2 \pi i} \sum_{i \geq 0} \frac{b_i}{2i+1} \str_M o\left( J_{[2i + 1]} a \right) q^{L_0} y^{J_0}  
\end{split}
\end{align}
\label{no:theta-is-conformal-block}
If $a$ is a primary vector, then we have $L_{[n]} a = J_{[n]} a = 0$ for all $n \geq 0$ and the sums in equations \eqref{eq:a.diff.taylor} vanish. 
\end{prop*}
\end{nolabel}
We can now apply the same technique as in Zhu's Theorem \cite[Thm 4.4.1]{zhu} to prove:
\begin{thm}\label{thm:converges}
Let $V$ be a $C_2$ cofinite conformal $N_W=1$ SUSY vertex algebra and let $M$ be a positive energy $V$-module. Then for each $a \in V$ the function $\str_M o(a) q^{L_0} y^{J_0}$ with $q = e^{2 \pi i \tau}$ and $y = e^{2 \pi i \alpha}$ converges absolutely and uniformly on every closed subset of
\[
S = \left\{ (\tau, \alpha) \: | \: \alpha \not \in \mathbb{Z} + \mathbb{Z}\tau \right\} \subset \mathbb{H} \times \mathbb{C}.
\]
\end{thm}
\begin{proof}
Let us first assume that $a$ is a primary vector. Noting that $b_i \in R$ for all $i$ we have from \eqref{eq:a.diff.taylor} and induction on $n$ that there exists $g_i(q, y), h_i(q,y) \in R$ such that 
\begin{multline*}
\str_M o\left( \left( L_{[-2]} + J_{[-2]} \right)^n a  \right) q^{L_0} y^{J_0} =  (2 \pi i)^{2 n} \left( q \frac{\partial}{ \partial q} \right)^n \str_M o(a) q^{L_0} y^{J_0} + \\  \sum_{i = 0}^{n-1} g_i(q,y) \left( q \frac{\partial}{\partial q}  \right)^i \str_M o(a) q^{L_0} y^{J_0}. 
\end{multline*}
\begin{multline*}
\str_M o \left( \left( J_{[-1]}  \right)^n  a\right) q^{L_0} y^{J_0} = (2 \pi i)^n \left( y \frac{\partial }{\partial y} \right)^n \str_M o(a) q^{L_0} y^{J_0} + \\ \sum_{i=0}^{n-1} h_i(q,y) \left( y \frac{\partial}{\partial y} \right)^i \str_M o(a) q^{L_0} y^{J_0}. 
\end{multline*}
Lemma \ref{no:zhu-pde} implies that there exists natural numbers $s,t$ and quasi-Jacobi forms $g_i, h_i \in R$ such that 
\begin{gather*}
\left(q \frac{\partial }{\partial q} \right)^s \str_M o(a) q^{L_0} y^{J_0} + \sum_{i = 0}^{s-1}  g_i(q,y) \left( q \frac{\partial}{\partial q} \right)^i \str_M o(a) q^{L_0} y^{J_0} = 0,  \\ 
\left(y \frac{\partial }{\partial y} \right)^t \str_M o(a) q^{L_0} y^{J_0} + \sum_{i = 0}^{t-1}  h_i(q,y) \left( y \frac{\partial}{\partial y} \right)^i \str_M o(a) q^{L_0} y^{J_0}=0.  
\end{gather*}
Since the functions $g_i(q,y)$ and $h_i(q,y)$ are singular only at $\alpha \in \mathbb{Z} + \mathbb{Z}\tau$ and are otherwise holomorphic, the result follows from the standard theory of differential equations. 

The reduction to the case when $a$ is not primary is more involved than in the usual case of vertex algebras due to the fact that the $N=2$ or topological algebra two sets of \emph{creation} operators instead of only $L_{[-i]}$ as in the usual case. If $a$ is not primary however, we can write it as
\begin{equation} \label{eq:a.expr.highest} 
a = L'_{[-i_k]} \dots L'_{[-i_1]} J_{[-j_l]} \dots J_{[-j_1]} b,\quad 1 \leq i_1 \leq \dots  \leq i_k, \quad 1 \leq j_1 \leq \dots \leq j_l, \end{equation}
with either $b$ primary or $Q_0 b$ primary. Here $L'_{[i]} = L_{[i]} - (i+1)J_{[i]}$ is easily seen to be another Virasoro algebra more suited to our differential equation \eqref{eq:a.diff.taylor}. Let us assume that $b$ is primary. We want to express $\Theta_M(a)$ as some differential operator with coefficients in $R$ applied to $\Theta_M(b)$ and then the result will follow from the first part of this proof.  If $i_k = 1$ then $\Theta_M (a) = 0$ by \eqref{eq:a.4.2}. If $i_k = 2$ then \eqref{eq:a.diff.taylor} expresses $\Theta_M(a)$ as a sum of $(2 \pi i)^2 q \partial_q \Theta_M(a')$ with $a = L'_{[-2]} a'$ and a linear combination (over $R$) of terms of the form $\Theta_M(c)$ with less than $k$ terms $L'$ and $l$ terms $J$ in the expression of $c$ as in \eqref{eq:a.expr.highest}. If $i_k > 2$ we use higher derivatives of the Weierstrass function to construct sections of $\cA$. In fact we have for any vector $a$ of charge $-1$ with respect to $J_{[0]}$, that $\wp_k (t) a [d\bt] \in H^0(E\setminus \bo, \cA_E)$ applying this to the conformal vector $h$ we get $\Theta_M\left( \res_{\bt} \zeta \wp_k(t) Y[h,\bt] a \right) = 0$ for any $k \geq 2$. Expanding in Taylor series we obtain the following generalization of \eqref{eq:a.4.2} which we write in the case of the conformal vector $h$:
\[ \Theta_M \left( L'_{[-(k+1)]} a \right) = \sum_{i \geq 0} f_i(q) \Theta_M \left( L'_{[i]} a \right), \]
for every $k \geq 2$ and for some quasi-modular forms $f_i$ that can be easily computed in terms of the $b_i$. It follows by induction that there exists $g_i(q,y) \in R$ and vectors $a_i \in V$, $i = 1, \dots, k$ such that 
\begin{enumerate}
\item $a_i$ may not be primary but can be written only using modes $J_{[j]}$ of the $U(1)$ current, that is: 
\[ a_i = J_{[-j_{i,l_i}]} \dots J_{[-j_{i,1}]} b_i\] for some primary vector $b_i$ and $1 \leq j_{i,1} \leq \dots j_{i,l_i}$ with $l_i \leq j_l$ for all $i$. 
\item We have \[ \Theta_M(a) = (2 \pi i)^{2k} \left( q \frac{\partial }{\partial q} \right)^k \Theta_M(a_k) + \sum_{i = 0}^{k-1} g_i(q,y) \left( q \frac{\partial}{\partial_q} \right)^i \Theta_M(a_i).\]
We have now reduced the proof of the theorem to the situation when $k = 0$ in \eqref{eq:a.expr.highest}. Define $a'$ by $a = J_{[-j_l]} a'$. If $j_l = 1$ then \eqref{eq:a.10.b} expresses $\Theta_M(a)$ as a sum of $\left( y \partial_y + \tfrac{C}{6} \right) \Theta_M(a')$ and a linear combination of terms of the form $\Theta_M(c)$ with less than $l$ modes of $J$ appearing in the expression of $c$ as in \eqref{eq:a.expr.highest}. If $j_l > 1$ we use the section $\zeta \wp_k(t) j [d\bt] \in H^0(E \setminus \bo, \cA_E)$ to obtain
\[ \Theta_M\left( J_{[-k]}a \right) = \sum_{i \geq 0} f_i(q) \Theta_M\left( J_{[i]}a \right), \]
for every $k \geq 2$ and the same quasi-modular forms $f_i$ as above. 

Collecting the above paragraph we have proved that for every $a$ as in \eqref{eq:a.expr.highest} there exists $g_{ij}(q,y) \in R$ and primary vectors $a_{ij} \in V$ such that 
\begin{equation}\label{eq:final-espero}
\Theta_M(a) = \sum_{i =0}^k \sum_{j=0}^{l_i} g_{ij}(q,y)\left( q \frac{\partial}{\partial q} \right)^i \left( y \frac{\partial}{ \partial y} \right)^j \Theta_M(a_{ij}). \end{equation}
Therefore $\Theta_M(a)$ converges since this sum is finite and each $\Theta_M(a_{ij})$ converges by the first part of the proof. The situation when $Q_0 b$ is primary in \eqref{eq:a.expr.highest} is treated in exactly the same way. 
\end{enumerate}
\end{proof}
\begin{rem}
It is clear from the proof of the theorem that the leading coefficient in \eqref{eq:final-espero} is $g_{kl} = (2 \pi i)^{2k +l}$. 
\label{rem:final-remark-leading}
\end{rem}
\def\cprime{$'$}

\end{document}